\documentclass[11pt]{article} 
\usepackage{float}
\usepackage{marginnote}
\usepackage{amsmath, amscd, amssymb,amsfonts,amsthm,amscd,graphicx,psfrag,epsfig}
\usepackage{youngtab}
\usepackage{graphicx}
\usepackage[all]{xy}
\usepackage{array}
\usepackage{amsthm}
\usepackage{stmaryrd}
\usepackage[titletoc,toc]{appendix}
\usepackage{stmaryrd}
\usepackage{mathtools}

\usepackage[
colorlinks=true,%
linkcolor=blue,%
citecolor=blue,%
filecolor=blue,%
menucolor=blue,%
urlcolor=blue]{hyperref}

\usepackage{tikz}
\usepackage{tikz-cd}

\usetikzlibrary{arrows,chains,shapes,matrix,positioning,scopes} 
\usetikzlibrary{decorations.markings}

\tikzstyle arrowstyle=[scale=1.1]
\tikzstyle directed=[postaction={decorate,decoration={markings,
    mark=at position 1 with {\arrow[arrowstyle]{latex}}}}]
\tikzstyle reverse directed=[postaction={decorate,decoration={markings,
    mark=at position .45 with {\arrowreversed[arrowstyle]{latex};}}}]
    
\tikzstyle arrowstyle=[scale=1.1]
\tikzstyle directed=[postaction={decorate,decoration={markings,
    mark=at position 1 with {\arrow[arrowstyle]{latex}}}}]
\tikzstyle reverse directed=[postaction={decorate,decoration={markings,
    mark=at position .45 with {\arrowreversed[arrowstyle]{latex};}}}]
    
 \usetikzlibrary{decorations.pathreplacing,decorations.markings,decorations.pathreplacing,arrows,shapes}   
\tikzset{
  -z>/.style={
    decoration={
      show path construction,
      lineto code={
        \path (\tikzinputsegmentfirst) -- (\tikzinputsegmentlast) coordinate[pos=#1] (mid);
        \draw (\tikzinputsegmentfirst) -- (mid);
        \draw[double distance=1.5pt, arrows = {- Latex[length=0pt 2 0]}] (mid) -- (\tikzinputsegmentlast);      }
    },decorate
  }, -z>/.default=.5,
  z->/.style={
    decoration={
      show path construction,
      lineto code={
          \path (\tikzinputsegmentfirst) -- (\tikzinputsegmentlast) coordinate[pos=#1] (mid);
                \draw[double distance=1.5pt] (\tikzinputsegmentfirst) -- (mid);
                \draw[decoration={markings, mark=at position 1 with {\arrow[scale=1.2]{latex}}},postaction={decorate}] (mid) -- (\tikzinputsegmentlast);
      }
    },decorate
  }, z->/.default=.5,
  --z>/.style={
    decoration={
      show path construction,
      lineto code={
        \path (\tikzinputsegmentfirst) -- (\tikzinputsegmentlast) coordinate[pos=#1] (mid);
        \draw [dashed](\tikzinputsegmentfirst) -- (mid);
        \draw [double distance=1.5pt, dashed, arrows = {- Latex[length=0pt 2 0]}] (mid) -- (\tikzinputsegmentlast);      }
    },decorate
  }, --z>/.default=.5,
    z-->/.style={
    decoration={
      show path construction,
      lineto code={
          \path (\tikzinputsegmentfirst) -- (\tikzinputsegmentlast) coordinate[pos=#1] (mid);
                \draw[double distance=1.5pt, dashed] (\tikzinputsegmentfirst) -- (mid);
                \draw[dashed, decoration={markings, mark=at position 1 with {\arrow[scale=1.2]{latex}}},postaction={decorate}] (mid) -- (\tikzinputsegmentlast);
      }
    },decorate
  }, z-->/.default=.5,
   x->/.style={
    decoration={
      show path construction,
      lineto code={
          \path (\tikzinputsegmentfirst) -- (\tikzinputsegmentlast) coordinate[pos=#1] (mid);
                \draw[double distance=2pt] (\tikzinputsegmentfirst) -- (mid);
                \draw[decoration={markings, mark=at position 1 with {\arrow[scale=1.2]{latex}}},postaction={decorate}] (\tikzinputsegmentfirst) -- (\tikzinputsegmentlast);
      }
    },decorate
  }, x->/.default=.5,
  -x>/.style={
    decoration={
      show path construction,
      lineto code={
        \path (\tikzinputsegmentfirst) -- (\tikzinputsegmentlast) coordinate[pos=#1] (mid);
        \draw[double distance=2pt, shorten >=5pt] (mid) -- (\tikzinputsegmentlast); 
        \draw[decoration={markings, mark=at position 1 with {\arrow[scale=1.5]{latex}}},postaction={decorate}] (\tikzinputsegmentfirst) -- (\tikzinputsegmentlast);}
    },decorate
  }, -x>/.default=.5,
   x-->/.style={
    decoration={
      show path construction,
      lineto code={
          \path (\tikzinputsegmentfirst) -- (\tikzinputsegmentlast) coordinate[pos=#1] (mid);
                \draw[double distance=2pt, dashed] (\tikzinputsegmentfirst) -- (mid);
                \draw[dashed, decoration={markings, mark=at position 1 with {\arrow[scale=1.2]{latex}}},postaction={decorate}] (\tikzinputsegmentfirst) -- (\tikzinputsegmentlast);
      }
    },decorate
  }, x-->/.default=.5,
  --x>/.style={
    decoration={
      show path construction,
      lineto code={
        \path (\tikzinputsegmentfirst) -- (\tikzinputsegmentlast) coordinate[pos=#1] (mid);
        \draw[double distance=2pt, dashed, shorten >=5pt] (mid) -- (\tikzinputsegmentlast); 
        \draw[dashed, decoration={markings, mark=at position 1 with {\arrow[scale=1.4]{latex}}},postaction={decorate}] (\tikzinputsegmentfirst) -- (\tikzinputsegmentlast);}
    },decorate
  }, --x>/.default=.5,
}

\tikzset{
    partial ellipse/.style args={#1:#2:#3}{
        insert path={+ (#1:#3) arc (#1:#2:#3)}
    }
}

\newtheorem{theorem}{Theorem}[section]

\newtheorem{lemma}[theorem]{Lemma}
\newtheorem{proposition}[theorem]{Proposition}
\newtheorem{corollary}[theorem]{Corollary}

\newtheorem{problem}[theorem]{Problem}
\newtheorem{conjecture}[theorem]{Conjecture}
\newtheorem{theorem-definition}[theorem]{Theorem-Definition}
\newtheorem{theorem-construction}[theorem]{Theorem-Construction}
\newtheorem{lemma-definition}[theorem]{Lemma--Definition}
\newtheorem{lemma-construction}[theorem]{Lemma--construction}
\newtheorem{definition}[theorem]{Definition}

\theoremstyle{definition}

\newtheorem{remark}[theorem]{Remark}
\newtheorem{example}[theorem]{Example}

\newcommand{\bg}{\begin{equation}\begin{gathered}}
\newcommand{\eg}{\end{gathered}\end{equation}}

\newcommand{\F}{\mathbb{F}}

\newcommand{\R}{\mathbb{R}}

\newcommand{\C}{\mathbb{C}}

\newcommand{\Z}{\mathbb{Z}}

\newcommand{\old}[1]{}

\newcommand{\Q}{{\mathbb Q}}

\newcommand{\G}{{\rm G}}

\newcommand{\lms}{\longmapsto}
\newcommand{\lra}{\longrightarrow}
\newcommand{\hra}{\hookrightarrow}

\newcommand{\be}{\begin{equation}}
\newcommand{\ee}{\end{equation}}
\newcommand{\bt}{\begin{theorem}}
\newcommand{\et}{\end{theorem}}
\newcommand{\bd}{\begin{definition}}
\newcommand{\ed}{\end{definition}}
\newcommand{\bp}{\begin{proposition}}
\newcommand{\ep}{\end{proposition}}

\newcommand{\bl}{\begin{lemma}}
\newcommand{\el}{\end{lemma}}
\newcommand{\bc}{\begin{corollary}}
\newcommand{\ec}{\end{corollary}}
\newcommand{\bcon}{\begin{conjecture}}
\newcommand{\econ}{\end{conjecture}}
\newcommand{\la}{\label}

\input xy
\xyoption{all}

\begin{document}

\begin{titlepage}
\title{Motivic fundamental group of ${\Bbb G}_m - \mu_N$ and 
 modular manifolds}
\author{A.B. Goncharov }
\date{}
\end{titlepage}
\stepcounter{page}
\maketitle

\tableofcontents

\begin{abstract} Let $\mu_N$ be the group of $N$-th roots of unity. 
We investigatie   geometric and combinatorial aspects 
of the mysterious relationship between the action of the motivic Galois group on 
the motivic fundamental group 
$\pi_1^{\cal M}({\Bbb G}_m - \mu_N, v_0)$, and the geometry of 
modular manifolds
\be \la{1MM}
Y_1(m;N) := \Gamma_1(m;N)\backslash {\rm GL_m}(\R)/{\rm O_m} \cdot \R^*_+.
\ee
Here   $\Gamma_1(m; N)$ is the subgroup of ${\rm GL}_m(\Z)$   fixing   $(0, \ldots, 0, 1)$ mod $N$.   

To achieve this, we consider  a canonical collection of   elements in the image of the motivic Galois Lie algebra,  spanning it over $\Q$. 
These elements are  the motivic correlators at $N-$th roots of unity.
We assign to   them chains in the complex computing cohomology of certain local systems on the space $Y_1(m;N)$. 
Their geometric  properties  reflect, in a  mysteriois way,   properties of  motivic correlators. 

This construction allows to  establish the relationship with high precision for $m \leq 4$. 
The $m=2, 3$ cases were investigated in the Hodge set-up in \cite{G97},  \cite{G98}, and in the Galois set-up in \cite{G00a}. 
\end{abstract}
 
  \section{Introduction}
  
 \subsection{Galois action on motivic fundamental groups and cyclotomic Lie algebras}
  

Denote by  $\mu_N$ the group of  $N-$th roots of unity. 
Let $\pi_1^{\cal M}({\Bbb G}_m - \mu_N, v_0)$ be the motivic fundamental group of ${\Bbb G}_m - \mu_N$ with the tangential base point  $v_0 = \partial / \partial t$ at $0$ \cite{DG}. 
It carries a depth filtration, provided  by the embedding ${\Bbb G}_m - \mu_N \hra {\Bbb G}_m$. 
 We study the action of the motivic Galois group on   
$\pi_1^{\cal M}({\Bbb G}_m - \mu_N, v_0)$,   equipped with  the depth filtration.

Before we go to the motivic set-up, let us briefly recall the   Galois point of view, where many of the structures are immediately visible.

\paragraph{The $l-$adic version.} The $l-$adic realization of the motivic fundamental group $\pi_1^{\cal M}({\Bbb G}_m - \mu_N, v_0)$ is given by the 
pro$-l$ completion $\pi^{(l)}_1({\Bbb G}_m - \mu_N, v_0)$ of the classical fundamental group, 
equipped with the action of the absolute Galois group  ${\rm Gal}(\overline \Q/\Q)$ by its automorphisms.
\vskip 2mm

 Maltsev's theory \cite[Chapter 9]{D} assigns to   $\pi^{(l)}_1({\Bbb G}_m - \mu_N, v_0)$  an $l$-adic 
pro-Lie algebra ${\Bbb L}^{(l)}_N$. 
 It    is   a free pronilpotent Lie algebra over $\Q_l$ with 
$N+1$ generators.  
The   Galois group   acts by  its  automorphisms. 

\vskip 2mm

The  Lie algebra ${\Bbb L}^{(l)}_N$ carries the canonical {\it weight}  filtration $W_\bullet$,  
preserved by the Galois action, indexed by negative integers. 
It  coincides with the lower central series of the Lie algebra 
${\Bbb L}^{(l)}_N$. So
$$
{\rm Gr}^W_{-1}{\Bbb L}^{(l)}_N = {\rm H}_1({\Bbb G}_m - \mu_N; \Q_l), ~~~~{\rm Gr}^W_{-2}{\Bbb L}^{(l)}_N = \Lambda^2{\rm H}_1({\Bbb G}_m - \mu_N; \Q_l), ~~~~\ldots 
$$

For each integer $m \geq 0$, we consider the subgroup of the Galois group 
preserving each of the amplitude $m$ subquotients $W_{\leq -k}{\Bbb L}^{(l)}_N/W_{\leq -k-m}{\Bbb L}^{(l)}_N$, where $k =-1, -2, \ldots$. 
We denote by  ${\cal F}_m\subset \overline \Q$ the subfield stabilized by this subgroup. 
 So we get a tower of fields, where ${\cal F}_0 = \Q(\zeta_{l^\infty N})$:
$$
\Q \subset  {\cal F}_0 \subset  {\cal F}_1 \subset  {\cal F}_2 \subset \ldots .
$$
The Galois groups ${\rm G}_w:= {\rm Gal}({\cal F}_w / {\cal F}_{w-1})$ are abelian $l-$adic groups. The commutators in the Galois group 
 give rise, after passing to the quotients,  to the $l-$adic graded Galois Lie algebra
\be \la{GLA}
 {\rm Lie} ~{\rm G}_\bullet:= \oplus_{w=1}^\infty {\rm Lie} ~{\rm G}_w.
\ee
Here are two examples showing how to think about this Lie algebra. 

\begin{enumerate}

\item The dual to the ${\rm Gal} (\overline{\Bbb Q}/ {\Bbb Q}(\zeta_{l^{\infty}N}))-$module ${\rm Lie} ~{\rm G}_1$ is described by the cyclotomic $N-$units
\be \la{PPP}
\left({\rm Lie} ~{\rm G}_1\right)^* = {\cal O}^*({\rm S_N}) \otimes \Q_l(1), ~~~~~~{\rm S_N} = {\rm Spec}~\Z[\zeta_N][\frac{1}{N}].
\ee
The isomorphism is given by   Kummer's theory. The   group ${\rm Gal} ({\Bbb Q}(\zeta_{N})/\Q)$ acts  on ${\rm S_N}$. 

\item The description of the Lie algebra of the Galois group ${\rm Gal} ({\cal F}_2/ {\Bbb Q})$ is just equivalent to the description of the Lie algebra 
structure on the graded ${\rm Gal} ({\Bbb Q}(\zeta_{l^\infty N})/\Q)-$module 
$$
{\rm Lie} ~{\rm G}_1 \oplus {\rm Lie} ~{\rm G}_2.
$$
The latter is given by the commutator map 
$$
[\ast, \ast]:  {\rm Lie} ~{\rm G}_1 \wedge {\rm Lie} ~{\rm G}_1 \lra {\rm Lie} ~{\rm G}_2.
$$
Just as in (\ref{PPP}),  it is more natural to describe the dual map: 
$$
 \delta:  \left({\rm Lie} ~{\rm G}_2\right)^* \lra \left({\rm Lie} ~{\rm G}_1\right)^* \wedge \left({\rm Lie} ~{\rm G}_1\right)^*.
$$
\end{enumerate}

\vskip 2mm  
Alternatively, this Galois Lie algebra can be described as follows. 
Restricting the Galois action on the Lie algebra ${\Bbb L}^{(l)}_N$ to the subgroup 
${\rm Gal} (\overline{\Bbb Q}/ {\Bbb Q}(\zeta_{l^{\infty}N}))$, we get the map 
$$
\varphi_N^{(l)}: {\rm Gal} \Bigl(\overline{\Bbb Q}/ {\Bbb Q}(\zeta_{l^{\infty}N})\Bigr) \lra {\rm Aut}~{\Bbb L}^{(l)}_N.
$$
Taking the   Lie algebra of the image of this map, we get the $l-$adic pro-Lie subalgebra
\begin{equation} \label{8.5.00.1}
{\rm Lie}~ {\rm Im} (\varphi_N^{(l)} )\subset {\rm Der} ~{\Bbb L}^{(l)}_N.
\end{equation}
The Galois Lie algebra (\ref{GLA}) is the associate graded of the pro-Lie algebra (\ref{8.5.00.1}) for the weight filtration. 

\vskip 2mm
 The  pro-Lie algebra ${\Bbb L}^{(l)}_N$ carries an addition {\it depth}  filtration,  
preserved by the Galois action. 
Namely, the inclusion
$
{\Bbb G}_m-\mu_N \hra {\Bbb G}_m
$
induces a map of the corresponding fundamental Lie algebras. The depth filtration  is given by the lower central series of 
the codimension one ideal
$$
{\cal I}^{(l)}:=  {\rm Ker}\Bigl( {\Bbb L}^{(l)}({\Bbb G}_m-\mu_N, v_{\infty})
 \lra {\Bbb L}^{(l)}({\Bbb G}_m, v_{\infty}) = \Q_l(1) \Bigr).
$$
It is numbered by the negative integers, starting from $-1$. 

\vskip 2mm
The weight and depth filtrations induce filtrations on 
${\rm Der} ~{\Bbb L}^{(l)}_N$, and hence,  on the Lie subalgebra  (\ref{8.5.00.1}). The induced filtrations are compatible. 
The associated graded ${\rm C}^{(l)}_{\bullet, \bullet}(\mu_N)$ for the weight and depth filtrations on  the Lie subalgebra  (\ref{8.5.00.1}) is a Lie algebra   bigraded by negative 
integers $-w$ and $-m$, where $w \geq m \geq 1$. We call it the {\it $l-$adic level $N$  cyclotomic Lie algebra}. 
 
 \vskip 2mm
Let us now turn to the motivic variant. Unlike the Galois framework, it takes much more effort to set it up, since it requires to have the abelian category of mixed Tate motives 
in hands, in the form predicted by Beilinson \cite{Be}.  The development of the latter  uses many  different ideas, such as  Borel's calculation of algebraic K-theory of number fields \cite{B}, 
 Bloch's higher Chow group complexes, and  Voevodsky's triangulated category of motives \cite{Vo}. 

In a sharp contrast with the absolute Galois group, the motivic Galois group itself appears only indirectly, via the Tannakian formalism -- following the  
insight of Grothendieck. We  do not have  a set-up where the motivic Galois group symmetries are present {\it a priori}. 

\vskip 2mm

There are several  benefits of the motivic point of view, justifying the effort:

\begin{itemize}

\item The motivic objects are defined over $\Q$ rather then $\Q_l$, and therefore do not depend on $l$. 

\item  Beilinson's formula for the Ext's, see (\ref{BEXT}) below,  imposes strong restrictions on the size of the  motivic Galois group, and therefore on the size of the Galois Lie algebra. 

\item  We can pass from the motivic set-up not only to l-adic, but also to the Hodge one. So combing this with the canonical real period map \cite{G08}, we can perceive 
the story on the level of real numbers. 

\end{itemize}

Let us proceed now to the motivic  cyclotomic Lie algebras.


\paragraph{Mixed Tate motives.} Let $\F$ be a number field, and ${\cal O}_{\F, S}$ the ring of $S-$integers in $\F$. Then  there is a mixed Tate category ${\cal M}_T({\cal O}_{\F, S})$ 
 of mixed Tate motives over ${\cal O}_{\F, S}$ \cite{DG} with all the   expected properties, including Beilinson's formula for the Ext's between the Tate motives $\Q(n)$:
\be  \la{BEXT}
\begin{split}
&  {\rm Hom}_{{\cal M}_T({\cal O}_{\F, S})}(\Q(0), \Q(n)) = \left\{ \begin{array}{ll}
0&: 
\quad \mbox{if $i=0, n \not = 0$}, \\ 
\Q&: \quad \mbox{if $i=0, n=0$}; \\
 \end{array}\right. \\
& {\rm Ext}_{{\cal M}_T({\cal O}_{\F, S})}^1(\Q(0), \Q(n)) = K_{2n-1}({\cal O}_{\F, S})\otimes \Q; \\
& {\rm Ext}_{{\cal M}_T({\cal O}_{\F, S})}^{>1}(\Q(0), \Q(n)) = 0. \\ 
\end{split}
\ee
The category ${\cal M}_T({\cal O}_{\F, S})$ comes with the canonical fiber functor $\omega$:
\be \la{FFW}
\omega: {\cal M}_T({\cal O}_{\F, S}) \lra {\rm Vect}, ~~~~X \lms \oplus_{n\in \Z} {\rm Hom}(\Q(-n), {\rm gr}^W_{2n}X).
\ee
Derivations of this  functor   compatible with the tensor structure   provide the fundamental Tate Lie algebra: 
$$
{\rm L}_\bullet({\cal O}_{\F, S}):= {\rm Der}^{\otimes}(\omega).
$$
It is a graded Lie algebra, with the grading by negative integers. The  functor $\omega$  induces an equivalence of tensor categories
 \be  \la{EQV}
\omega: {\cal M}_T({\cal O}_{\F, S}) \stackrel{\sim}{\lra}  \mbox{finite dimensional graded ${\rm L}_\bullet({\cal O}_{\F, S})-$modules}.
 \ee  

\paragraph{The motivic fundamental group $\pi_1^{\cal M}({\Bbb G}_m - \mu_N, v_0)$.} It was defined in  \cite{DG}.  We  discuss it in more detail in Section \ref{SEC8}. 
It is a Lie algebra pro--object in the category 
of mixed Tate motives over the scheme 
\be \la{SN}
{\rm S_N}= {\rm Spec}~\Z[\zeta_N][\frac{1}{N}].
\ee
Therefore  there is a canonical 
map of Lie algebras
\be \la{CANM}
\varphi_N: {\rm L}_\bullet({\rm S_N}) \lra {\rm Der}~\omega(\pi_1^{\cal M}({\Bbb G}_m - \mu_N, v_0)).
\ee 

 \paragraph{The cyclotomic Lie (co)algebras.} The map 
${\Bbb G}_m - \mu_N \hra {\Bbb G}_m$ provides a short exact sequence 
$$
0\lra {\cal I} \lra \pi_1^{\cal M}({\Bbb G}_m - \mu_N, v_0) \lra \pi_1^{\cal M}({\Bbb G}_m, v_0) = \Q(1) \lra 0.
$$
 The  {\it depth filtration} ${\cal F}_\bullet$ on the Lie algebra $\pi_1^{\cal M}({\Bbb G}_m - \mu_N, v_0)$ is given by the lower central series  of the  ideal ${\cal I}$. 
 It induces a depth filtration on the Lie algebra  ${\rm Der}~\omega(\pi_1^{\cal M}({\Bbb G}_m - \mu_N, v_0))$, and hence on its subalgebra  
 $\varphi_N ({\rm L}_\bullet({\rm S_N}))$, see (\ref{CANM}). 
 
 \bd The {\it cyclotomic Lie algebra} ${\rm C}_{\bullet, \bullet}(\mu_N) $ is the associate graded for the depth filtration on 
  $\varphi_N ({\rm L}_\bullet({\rm S_N}))$,  graded by the weight: 
  \be \nonumber
    \begin{split}
&{\rm C}_{-w, -m}(\mu_N):= {\rm gr}_{{\cal F}}^{-m}\varphi_N ({\rm L}_{-w}({\rm S_N}) ),\\
&{\rm C}_{\bullet, \bullet}(\mu_N) := \bigoplus_{w,m\geq 1} {\rm C}_{-w, - m}(\mu_N).\\
 \end{split}
   \ee
   \ed
   
The cyclotomic Lie algebra is bigraded by the weight $-w$ and the depth $-m$. 
  Its graded dual is the {\it level $N$ cyclotomic Lie coalgebra}:
    \be \nonumber   \begin{split}
  &{\cal C}_{w,m}(\mu_N):=  {\rm C}_{-w, -m}(\mu_N)^\vee,\\
  &  {\cal C}_{\bullet, \bullet}(\mu_N) := \bigoplus_{w,m\geq 1} {\cal C}_{w,m}(\mu_N).\\
  \end{split}  \ee  
  
It follows immediately from the properties of the motivic classical polylogarithms at roots of unity \cite{G02b}, or the unpublished 
work of Beilinson-Deligne \cite{BD}, or \cite{DG}, that the depth $1$ part of the cyclotomic Lie coalgebra is given by 
\be \la{BT}
 {\cal C}_{w, 1}(\mu_N) = K_{2w-1}({\rm S_N}) \otimes \Q.
 \ee
We   review the computation of  $K_{2w-1}({\rm S_N}) \otimes \Q$ little later. Meanwhile we just note the isomorphism
 \be \la{muiso}
 {\cal C}_{1, 1}(\mu_N) = {\cal O}^*({\rm S_N}) \otimes \Q.
 \ee

 \paragraph{The extended cyclotomic Lie coalgebra.} It turned out to be convenient to enlarge   ${\cal C}_{\bullet, \bullet}(\mu_N)$ by adding   a 
 one dimensional space of degree $1$, in a sense corresponding    
to Euler's $\gamma-$constant: 
\be \la{mmi}
\begin{split}
&\widehat {\cal C}_{\bullet, \bullet}(\mu_p):=  \widehat {\cal C}_{1,1}(\mu_N) \bigoplus \bigoplus_{w+m>1} {\cal C}_{w,m}(\mu_N);\\
&\widehat {\cal C}_{1,1}(\mu_N):= {\cal C}_{1,1}(\mu_N) \oplus \gamma^{\cal M}\cdot \Q, ~~~~~~\delta \gamma^{\cal M}:= 0.\\
   \end{split}
   \ee
 So we get a graded Lie coalgebra $\widehat {\cal C}_{\bullet, \bullet}(\mu_N)$.     The isomorphism (\ref{mmi}) reads now as follows
  \be \la{muiso}
 \widehat {\cal C}_{1, 1}(\mu_N) = \widehat {\cal O}^*({\rm S_N}) \otimes \Q:= {\cal O}^*({\rm S_N}) \otimes \Q ~\oplus ~\gamma^{\cal M}\cdot \Q.
 \ee

  \paragraph{The diagonal cyclotomic Lie (co)algebras.}  Our next step, partially motivated by the desire to simplify the exposition, is to reduce for a while our considerations by the "weight = depth" 
  quotient Lie coalgebra 
of  ${\cal C}_{\bullet, \bullet}(\mu_N)$.
 Namely, we define the diagonal cyclotomic Lie (co)algebras as follows:
  \be \nonumber
    \begin{split}
  &{\rm C}_{\bullet}(\mu_N) := \bigoplus_{w\geq 1} {\rm C}_{-w, - w}(\mu_N).\\
   &{\cal C}_{\bullet}(\mu_N) := \bigoplus_{w\geq 1} {\cal C}_{-w, - w}(\mu_N).\\
   \end{split}
   \ee
The diagonal part of the extended cyclotomic Lie coalgebra is denoted by $\widehat {\cal C}_{\bullet}(\mu_N)$. 
Since it is graded by the poisitive integers,  the weight $m$ part of its cohomology are concentrated in the degrees $[1, \ldots , m]$:
\be \la{est}
{\rm H}_{(m)}^i(\widehat {\cal C}_\bullet(\mu_N)) = 0~~ \mbox{unless $1 \leq i\leq m$}. 
\ee
   Evidently, 
$$
   {\rm H}^1_{(1)}(\widehat {{\cal C}}_{\bullet}(\mu_N))    = \widehat {{\cal C}}_{1}(\mu_N).
$$
      
\paragraph{The bigraded cohomology groups.} Since the Lie coalgebra  $\widehat {{\cal C}}_{w,m}(\mu_N)$ is bigraded, its cohomology inherit the bigrading. We denote by 
${\rm H}^\ast_{(w,m)}(\widehat {{\cal C}}_{\bullet, \bullet}(\mu_N))$ the   degree $(w,m)$ part of the cohomology. 
Since  $\widehat {{\cal C}}_{w,m}(\mu_N)=0$ for $w<m$, we get   
$$
{\rm H}^\ast_{(w,m)}(\widehat {{\cal C}}_{\bullet, \bullet}(\mu_N)) = 0 ~~~~\mbox{if $w<m$}.
$$
In the diagonal case $w=m$   we have:
\be \la{GFG}
{\rm H}^i_{(m,m)}(\widehat {{\cal C}}_{\bullet, \bullet}(\mu_N)) = {\rm H}^i_{(m)}(\widehat {{\cal C}}_{\bullet}(\mu_N)).
 \ee
 
          \subsection{  Cyclotomic Lie algebras of level $p$ and cohomology of modular manifolds} \la{ss1.2}
    

In Section \ref{ss1.2} we present the results for the diagonal cyclotomic Lie algebra  $\widehat {\cal C}_{\bullet}(\mu_p)$ in the case when $N=p$ is  prime number. 
Below we denote by  $\varepsilon: {\rm GL_m}(\Z) \lra \{\pm 1\}$   the determinant character. 

\paragraph{The weight $m$   cohomology of $\widehat {\cal C}_{\bullet}(\mu_p)$.} 
 Let
\be \la{dm}
d_m:={\rm dim} (
{\rm GL_m}(\R)/{\rm O}_m \cdot \R^*_+) = (m-1)(m+2)/2.
\ee
It is well known that for any discrete subgroup $\Gamma \subset {\rm GL}_m(\R)$ and any representation $V$ of ${\rm GL}_m$, 
\be \la{GV0}
H^i(\Gamma, V) =0 ~~~\mbox{if $i > d_m- (m-1)$}.
\ee

Theorem \ref{ALLM1} relates the top dimensional cohomology of the cyclotomic Lie algebra, see (\ref{est}),  with the top dimensional cohomology of $\Gamma_1(m;p)$,  see (\ref{GV0}), for all positive integers $m$.

\bt \la{ALLM1} For any integer $m \geq 2$, and any prime $p$, there are canonical isomorphisms
 \be \la{ALLM}
 {\rm H}^m_{(m,m)}(\widehat {{\cal C}}_{\bullet, \bullet}(\mu_p)) \stackrel{(\ref{GFG})}{=}
 {\rm H}^m_{(m)}({\widehat {\cal C}}_{\bullet}(\mu_p))  \stackrel{\sim}{\lra}  {\rm H}^{d_m- m+1}(\Gamma_1(m;p), \varepsilon^{m-1}).
  \ee
  \et
  
 Theorem \ref{ALLM1}   holds for $m=1$ if one understood the right hand side of (\ref{ALLM}) properly, see (\ref{muiso}),  (\ref{7.8.00.3a}). 
 
 \vskip 2mm
 The isomorphism (\ref{ALLM}) is defined in Section \ref{SECT1.4} by using modular symbols, see (\ref{qwe}).  It is "elementary" in the sense that it does not require any motivic technology.
 
 \vskip 2mm

 Theorem \ref{CMR1y}   tells that the weight $m$  cohomology of  the Lie coalgebra $\widehat {\cal C}_{\bullet}(\mu_p)$ are given by the cohomology of the   group $\Gamma_1(m; p)$ for $m=2,3,4$. 
    \bt \la{CMR1y} Let $p$ be a prime. The   weight $m\leq 4$ cohomology ${\rm H}_{(m)}^i(\widehat {\cal C}_\bullet(\mu_p))$ of the  diagonal  cyclotomic Lie coalgebra, see   Table \ref{table1}, are:
\begin{equation} \label{BR2aa}
\begin{split}
 &{\rm H}^i_{(2)}(\widehat {{\cal C}}_{\bullet}(\mu_p)) =  
{\rm H}^{i-1} (\Gamma_1(2;p), \varepsilon),~~~~~~i=1,2;\\
&{\rm H}^i_{(3)}(\widehat {{\cal C}}_{\bullet}(\mu_p)) =  
{\rm H}^{i}(\Gamma_1(3;p), \Q),~~~~~~~~i=1,2,3;\\
&H^i_{(4)}(\widehat {{\cal C}}_{\bullet}(\mu_p)) =   \left\{ \begin{array}{ll}
\mbox{\rm a quotient of}\\
 {\rm H}^{i+2}(\Gamma_1(4;p),  \varepsilon),& 
 \mbox{$i=1$}, \\ 
 {\rm H}^{i+2}(\Gamma_1(4;p), \varepsilon),& 
 \mbox{$i=2, 3, 4$}.\\ 
 \end{array}\right.\\
  \end{split}
  \end{equation}
\et
By (\ref{est}),  the diagonal  groups ${\rm H}_{(m)}^m(\widehat {\cal C}_\bullet(\mu_p))$     on Table \ref{table1} are the top dimensional cohomology. 

The cohomology groups in  the cuspidal range  are in red. 
\begin{table}
  \begin{center}
   \caption{The weight $m$ cohomology ${\rm H}_{(m)}^i(\widehat {\cal C}_\bullet(\mu_p))$ for $m=1,2,3,4$.}
    \label{table1}
    \begin{tabular}{c||c|c|c|c|}
     & i=1 & i=2 & i=3 & i=4\\   
      \hline \hline
       m=1 & ${\cal O}^*({\rm S_p})\otimes \Q $& 0  & 0  & 0  \\   
      \hline
      m=2 & ${\rm H}^{0} (\Gamma_1(2;p), \varepsilon)$ & \begin{color}{red}   ${\rm H}^{1} (\Gamma_1(2;p), \varepsilon)$\end{color}& 0&0\\ 
    \hline
     m=3 &  ${\rm H}^{1}(\Gamma_1(3;p), \Q)$ & \begin{color}{red} ${\rm H}^{2}(\Gamma_1(3;p), \Q)$\end{color} & \begin{color}{red}${\rm H}^{3}(\Gamma_1(3;p), \Q)$\end{color}&0\\ 
 \hline
   m=4 &  ${\rm H}^{3}(\Gamma_1(4;p),  \varepsilon)/* $& \begin{color}{red} $ {\rm H}^{4}(\Gamma_1(4;p),  \varepsilon) $\end{color}  & \begin{color}{red}  ${\rm H}^{5}(\Gamma_1(4;p),  \varepsilon) $\end{color} & ${\rm H}^{6}(\Gamma_1(4;p),  \varepsilon)$.\\
   \hline
    \end{tabular}
  \end{center}
\end{table}
  Few comments are in order.

\paragraph{1.} 
 Thanks to (\ref{GFG}), Theorem \ref{CMR1y} tells  us the  diagonal part  ${\rm H}^i_{(m,m)}(\widehat {{\cal C}}_{\bullet, \bullet}(\mu_p))$ of the cohomology of the cyclotomic Lie algebra 
 ${\cal C}_{\bullet, \bullet}(\mu_p)$ for prime $p$ and $m=2,3,4$. 
  
\paragraph{2.} For any discrete subgroup $\Gamma \subset {\rm GL}_m(\Z)$, and any ${\rm GL}_m-$module $V$, the most interesting part of 
the cohomology ${\rm H}^i(\Gamma, V)$ are the   cuspidal cohomology  ${\rm H}^i_{\rm cusp}(\Gamma, V) \subset {\rm H}^i(\Gamma, V)$.

 It is known   \cite{VZ}, \cite{BW} that  ${\rm H}^i_{\rm cusp}(\Gamma, V)=0$ unless 
$V$ is   self dual, and $i$ is in the segment  of  
$[\frac{m+1}{2}]$ consequitive  integers centered at the point $d_m/2$, where $d_m$ is given by (\ref{dm}). 

 In particular, for $m \leq 4$ the cuspidal range is the following, see also  Table \ref{table1}: 
$$
m=2: ~i=1; ~~~~~~ m=3:~i=2,3; ~~~~~~ m=4:~i= 4,5.  
$$
So by  Theorem \ref{CMR1y},   ${\rm H}_{(m)}^i(\widehat {{\cal C}}_{\bullet}(\mu_p))   $ for $m \leq 4$ catch   all the cuspidal cohomology of $\Gamma_1(m;p)$.  


\paragraph{3.}   It is well known that 
\be \la{p2-1}
{\rm dim} ~{\rm H}^1_{\rm cusp}(\Gamma_1(2;p), \varepsilon)= 1 + \frac{p^2-1}{24} - \frac{p-1}{2}.
\ee
One might wonder whether  the isomorphism 
$
{\rm H}^2_{(2)}(\widehat {{\cal C}}_{\bullet}(\mu_p)) =  
{\rm H}^{1}  (\Gamma_1(2;p), \varepsilon)
$ 
 from Theorem \ref{CMR1y} is just an accident. 
However, in contrast with (\ref{p2-1}),   the \underline{cuspidal} cohomology of 
$\Gamma_1(m;p)$ for $m>2$ are    sporadic: there is   no  rule predicting their size.  
 So   isomorphisms (\ref{BR2aa})   are not  accidental. 
 
\paragraph{4.}  We do not know   why the   weight $m$ cohomology of the Lie coalgebras ${\widehat {\cal C}}_{\bullet}(\mu_p)$ could   possibly be related to the   cohomology of $ \Gamma_1(m;p)$ for $m>2$. 

Yet for $m=2$, there is an "explanation"   via   
{\it Euler complexes} on the modular curve $Y(N)$  and their  specialization to the cusp given by the scheme ${\rm S}_N$ \cite{G05}. 
We hope that a similar picture exists for all $m$. Note   that modular manifolds $Y(m; N)$ are not algebraic  
for $m>2$. 

\vskip 2mm

Note that one evidently has canonical isomorphisms 
$$
 {\rm H}^1_{(m)}({{\cal C}}^\circ_{\bullet}(\mu_p))  =  
{\rm H}^1_{(m)}({{\cal C}}_{\bullet}(\mu_p))  ={\rm H}^1_{(m)}(\widehat {{\cal C}}_{\bullet}(\mu_p)), ~~~~\forall m>1.
$$ 

\bcon  There is a canonical isomorphism
  \be \la{JU30.19}
 {\rm H}^1_{(4)}({{\cal C}}_{\bullet}(\mu_p))  \stackrel{\sim}{\lra}  {\rm H}^{1}(\Gamma_1(2;p), S^2V_2 \otimes \varepsilon).
  \ee
\econ

\subsection{Modular complexes, modular manifolds,   and cyclotomic Lie algebras} \la{SECT1.3}

Let us now discuss the general case of the cyclotomic Lie coalgebra ${\cal C}_{\bullet, \bullet}(\mu_N)$.

 \vskip 2mm

  Let us go back to the $m=1$ case. 
  The right hand side of (\ref{BT}) was calculated by  Borel  \cite{B}.  
The result can be stated as follows:
\begin{equation} \nonumber
  \begin{split}
  K_{2w-1}({\rm S_N})\otimes \Q  =  &
H^0({\rm S_N}(\C), {\cal L}_{S^{w-1}V_1})^+\\
=& H^0(\{{\rm S_N}(\C), c\},  {\cal L}_{S^{w-1}V_1}), ~~~~  w >  1.
 \end{split} 
\end{equation}
Here ${\cal L}_{S^{w-1}V_1}$ is the local 
system on  ${\rm S_N}(\C)$ corresponding to the ${\rm GL_1}$-module $S^{w-1}V_1$. 
The $+$ means invariants of the  involution acting on ${\rm S_N}(\C)$ as the complex conjugation $c$, and on the local system via  $v \lms -v$ on $V_1$. 
In the second isomorphism   the pair 
 $\{{\rm S_N}(\C), c\}$ is a stack and ${\cal L}_{S^{w-1}V_1}$ 
as a local system on it. 
 We conclude that 
 \begin{equation} \label{7.8.00.3a} 
  \begin{split}
  {\cal C}_{w, 1}(\mu_N)   =   H^0(\{{\rm S_N}(\C), c\},  {\cal L}_{S^{w-1}V_1}), ~~~~  w >  1.
 \end{split} 
\end{equation}

Generalising (\ref{7.8.00.3a}),  we    relate the cyclotomic Lie coalgebra ${\cal C}_{\bullet, \bullet}(\mu_N)$ to the  properties of local systems on  the following modular stacks, 
 for $m=1,2,3,4, \ldots$:
   \begin{equation} \label{7.8.00.2}
Y_1(m;N):=   {\rm GL_m}(\Q) \backslash {\rm GL_m}({\Bbb A}_{\Q})/ K_1(m;N) 
\cdot \R_+^* \cdot O_m.
 \end{equation} 
Here ${\Bbb A}_{\Q}$ is the ring of adels of $\Q$. The subgroup 
$K_1(m;N) \subset \prod_p {\rm GL_m}(\Z_p)$ is defined 
by  congruence conditions at 
the primes $p|N$: if $N = \prod p^{v_p(N)}$ then its $p$-component consists of the elements of ${\rm GL_m}(\Z_p)$ stabilizing  
$(0, ..., 0, 1)$ mod $p^{v_p(N)}$. 
So
$
\Gamma_1(m;N) :=  K_1(m;N) \cap {\rm GL}_m(\Z).
$

When $m=1$, the stack
(\ref{7.8.00.2}) is  isomorphic to the stack $(S_N(\C), c)$. 

When $m>1$, the 
stack (\ref{7.8.00.2}) coincides with the locally symmetric space: 
\be \label{7.19.00.1}
Y_1(m;N)    =   \Gamma_1(m;N) \backslash {\rm GL_m}(\R)/ {\rm O(m)} \cdot \R_+^*= \Gamma_1(m;N) \backslash {\Bbb H}_{\rm L_m}.
\ee 
It  parametrizes triples $({\rm L}; v, g)$, where ${\rm L}$ is a rank $m$ lattice,    $v \in {\rm L}/ N{\rm L}$ 
is an element of order $N$, and $g$ is a positively definite quadratic form on ${\rm L}\otimes \R$. 
 
\vskip 2mm
Below we pick a rank $m$ lattice 
${\rm L_m}$. 
A choice of a basis in the lattice ${\rm L_m}$ induces an isomorphism ${\rm Aut}({\rm L_m}) =   {\rm GL_m}(\Z)$. 
Let $V_m$ be the dual to the vector space ${\rm L_m} \otimes \R$. 

\vskip 2mm
The key  tool    we use to relate the cyclotomic Lie coalgebras to local systems on the modular stacks (\ref{7.8.00.2}) for $m>1$ 
 is the  {\it modular complex} ${\rm M}_{(m)}^{\bullet}$  defined in \cite{G98}. It
 is a complex of ${\rm Aut}({\rm L_m})$-modules of amplitude  $[1,m]$. We recall its definition in Section \ref{Sec2}. Below we assume that $m>1$.

The weight $2$ 
 {modular complex} ${\rm M}_{(2)}^{\bullet}$    coincides with  the chain  
complex of the classical modular triangulation of the  
hyperbolic plane on geodesic triangles,  see Figure \ref{mod} (the  cusps are excluded). 
\begin{figure}[ht]
\centerline{\epsfbox{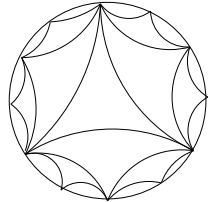}}
\caption{The modular triangulation of the hyperbolic plane.}
\label{mod}
\end{figure}
This triangulation is obtained by the action of  ${\rm GL_2}(\Z)$ on  
the geodesic triangle with  vertices at $\infty, 0, 1$. An element $\left 
(\begin{matrix} a&b \\ c& d \end{matrix}\right ) \in {\rm GL_2}(\R)$ acts  on the upper half plane by 
$$
z \lms \frac{az+b}{cz+d} \quad \mbox{if} \quad det\left (\begin{matrix} a&b \\ c& d\end{matrix} \right ) =1, 
\qquad z \lms \frac{a \overline z+b}{c\overline z+d} \quad \mbox{if} \quad det\left (\begin{matrix} a&b 
\\  c& d\end{matrix} \right ) =-1. 
$$
Note that   
${\rm GL_2}(\R)/\R_+^* \cdot {\rm SO(2)} = \C - \R$ and ${\rm O(2)/SO(2)}$ 
acts as the complex conjugation.


Then given an integer $N\geq 1$,  we consider the following   space, bigraded by the  $(m, w)$:  
 \be \la{MSMS00n}
{\rm M}_N^\ast:=  \bigoplus_{w\geq m \geq 1} {\rm M}^\ast_{(m)}\otimes_{\Gamma_1(m;N)} S^{w-m}V_m. 
\ee
   
The differential $\partial$ is  induced by the one on ${\rm M}^\ast_{(m)}$. 
The  bigrading by    $(m,w)$  is preserved by $\partial$.

\bt \la{THEOREM1.4} There is an algebra structure on  the space (\ref{MSMS00n}), 
 which is graded commutative for the cohomological grading $\ast$, and commutative for the weight and depth gradings. 
 \et
 
 We prove Theoerem \ref{THEOREM1.4} in the end of Section \ref{Sec2aa}. 
 So we arrive at a commutative  dg-algebra ${\rm M}_N^\ast$, equipped with the two extra gradings. We call it the {\it level $N$ modular dg-algebra}.

  \vskip 2mm
On the other hand, any Lie coalgebra ${\cal L}$ gives rise a  commutative dg-algebra, given by 
the exterior algebra 
 $\Lambda^*  {\cal L}$ with the differential induced via the Leibniz rule by the coproduct map $\delta: {\cal L} \lra \Lambda^2{\cal L}$. 
 It is usually presented as the standard cochain complex:
$$
  {\cal L} \lra \Lambda^2  {\cal L} \lra  \Lambda^3  {\cal L} \lra \ldots .
$$
 If the Lie coalgebra ${\cal L}$ is graded, its standard cochain complex inherits the grading. 
 
 In particular, the   cochain complex of the Lie coalgebra $\widehat {\cal C}_{\bullet, \bullet}(\mu_N)$ is a commutative dg-algebra.


\begin{theorem} \label{4.26.01.3a} 
Pick a primitive $N-$th root of unity $\zeta_N$. Then there is a canonical surjective map of commutative dg-algebras
\be \la{MGA11}
 \gamma: {\rm M}_N^\ast ~~\lra ~~\Lambda^*\Bigl(\widehat {\cal C}_{\bullet, \bullet}(\mu_N)\Bigr).
\ee
\end{theorem}

The proof of Theorem \ref{4.26.01.3a} is given in Section \ref{4.26.01.3aX}. 
To define   the map (\ref{MGA11}) we use motivic correlators at roots of unity,  and the second shuffle relations for   motivic correlators proved in  \cite{Ma19b}.  
A very similar  proof of Theorem \ref{4.26.01.3a}  
  uses   motivic multiple  polylogarithms at roots of unity defined in \cite{G02b}.  


\begin{theorem} \label{MTTHHa}
i) Let  $\Gamma$ be a  finite index subgroup of ${\rm GL_m}(\Z)$ and $V$   is 
a   ${\rm GL_m}-$module. Then  
\begin{equation} \label{3more1} 
 {\rm H}^i \left({\rm M}^{\ast}_{(m)} \otimes_{\Gamma}V\right)= ~ 
\left\{ \begin{array}{ll}
 {\rm H}^{i-1}(\Gamma, V \otimes \varepsilon) & 
 \mbox{for} \quad 1 \leq i \leq 2, \quad m=2,\\ 
{\rm H}^{i}(\Gamma, V) &   \mbox{for} \quad 1 \leq i \leq 3, \quad m=3,\\ 
{\rm H}^{i+2}(\Gamma, V\otimes \varepsilon) &  \mbox{for} \quad 1 \leq i \leq 4, \quad m=4.\\
 \end{array}\right.
\end{equation}

ii) For any $m \geq 2$, there is a canonical  isomorphism:
\be \la{gmmm}
\mu_{m}^m: {\rm H}^m \left({\rm M}^{\ast}_{(m)} \otimes_{\Gamma}V\right) \stackrel{\sim}{\lra}  
 {\rm H}^{\frac{m(m-1)}{2}}(\Gamma, V\otimes \varepsilon^{m-1}).\\ 
\ee

\end{theorem}

Theorem \ref{MTTHHa} is proved in Section \ref{SEC5}. Its $m=2$ case is given by Theorem \ref{thc2}. The $m=3$ case follows from Theorem \ref{thc3}. The $m=4$ case follows from Theorem \ref{1more} and Corollory \ref{1more}.

By  Theorem \ref{MTTHHa} 
the   cohomology (\ref{3more1}) catch   the cuspidal cohomology.  
The cuspidal cohomology can be non-zero only if $V$ is self-dual. So if  
$V=S^{w-m}V_m$,  they vanish unless $w=m$. 

\vskip 2mm
Note that for a finite index subgroup $\Gamma  \in {\rm GL_m}(\Z)$, when $m$ grows,  the possible non-zero ranges of the   cohomology groups ${\rm H}^i  ({\rm M}^{\ast}_{(m)} \otimes_{\Gamma}V )$ and ${\rm H}^i(\Gamma, V)$ are quite different;  
the first is linear in $m$, while the second is quadratic: 
\be  \la{EST}
\begin{split}
&{\rm H}^i \left({\rm M}^{\ast}_{(m)} \otimes_{\Gamma}V\right) = 0 \quad \mbox{if $i \not \in \{1, ..., m\}$};\\
& {\rm H}^i(\Gamma, V) = 0 \quad \mbox{if $i \not \in \{0, ..., \frac{m(m-1)}{2}\}$}.\\
\end{split}
\ee
 Theorem \ref{MTTHHa}~ii) tells that the top cohomology groups in the non-zero range in (\ref{EST}) coincide.


\paragraph{Conclusion.} Combining Theorems \ref{4.26.01.3a} and \ref{MTTHHa} we relate the level $N$ cyclotomic Lie coalgebra describing the structure of the motivic fundamental group 
$\pi_1^{\cal M}({\Bbb G}_m - \mu_N, v_0)$ to the geometry of locally symmetric spaces for the congruence subgroup $\Gamma_1(m, N) \subset {\rm GL_m}(\Z)$ for $m=1,2,3,4.$ 

\vskip 2mm
This relationship in the depth $m=2$ and weight $w=2$ was discovered in \cite{G97}. The  general 
depth $m\leq 3$ case was worked out in \cite{G98} in the Hodge set-up,  and in \cite{G00a} in the set-up of $l-$adic Galois representations. See also   survey \cite{G00b}.  
Our new results   are in the depth $m=4$. 
\vskip 2mm

In this paper  we work in the motivic set-up, 
using  the  motivic fundamental groups $\pi_1^{\cal M}({\Bbb G}_m - \mu_N, v_0)$. This  makes the story  more transparent.   
We use  either  motivic multiple polylogarithms  or   motivic correlators. 
Motivic correlators make   the connection more natural: 

\begin{itemize} 
\item 
Motivic correlators  live in the motivic Lie coalgebra, 
while  motivic multiple polylogarithms live in the motivic Hopf algebra. This is   why we prefer the former to the latter. 
\item Motivic correlators satisfy the shuffle and dihedral symmetry relations on the nose. 

\item 
The coproduct of motivic correlators at $N-$roots of unity  match, after a change of variables,   
the boundary of  the defined in Section \ref{SECT1.4} $(2m-2)-$cell $\psi(\alpha_0, \ldots , \alpha_m)$, $\alpha_i\in \Z/N\Z$, as well as the differential in the modular complex.  

\end{itemize} 

\paragraph{Elaborating the map $\gamma$.} The   cyclotomic Lie coalgebra $\widehat {\cal C}_{\bullet, \bullet}(\mu_N)$ is bigraded by the weight $w$ and depth $m$. 
Denote by  $\Lambda^*_{w,m}(\widehat {\cal C}_{\bullet, \bullet}(\mu_N))$ 
 the   weight $w$,  depth $m$ part of the corrersponding dg-algebra. 
Then   Theorem \ref{4.26.01.3a} means just the following. For each $w \geq m \geq 1$:

a) There exists canonical surjective map of $\Q$-vector spaces
\be \la{MMa}
\gamma_{m, w}^1: {\rm M}^1_{(m)}\otimes_{\Gamma_1(m;N)}{\rm S}^{w-m}(V_m ) \lra {\cal C}_{w,m}(\mu_N).
\ee

b) 
The  map (\ref{MMa}) extends to  a canonical surjective map of complexes
\begin{equation} \label{4.26.01.2a}
\gamma_{m, w}^\ast: {\rm M}^{*}_{(m)}\otimes_{\Gamma_1(m;N)}{\rm S}^{w-m}(V_m ) \lra 
\Lambda_{w,m}^*\Bigl(\widehat {\cal C}_{\bullet, \bullet}(\mu_N)\Bigr).
\end{equation}
The map (\ref{4.26.01.2a})    was defined in  \cite{G99}, \cite{G00a} for a variant of the cyclotomic Lie algebra.

\vskip 2mm
The   map of complexes $\gamma^*_{m,m}$ induces the following map on the last cohomology group, $\ast =m$: 
\be \nonumber
\gamma_{m}^m: {\rm H}^m \left({\rm M}^{\ast}_{(m)} \otimes_{\Gamma_1(m;N)} \Q\right) \lra  
 {\rm H}^{\frac{m(m-1)}{2}}(\Gamma_1(m;N), \varepsilon^{m-1}).\\ 
\ee
  Let us explain how to define this map   by using  Mazur's modular symbols for ${\rm GL_m}$.

 \subsection{Mazur's modular symbols for ${\rm GL_m}$ and the map $\gamma$ in simplest case} \la{SECT1.4}

\paragraph{1. Mazur's modular symbols  
for ${\rm GL}_m$.} Recall a rank $m$ lattice 
${\rm L_m}$ and the   vector space $V_m$ dual to ${\rm L_m} \otimes \R$.  Let ${\cal P}(V_m)$ (respectively $\overline {\cal P}(V_m)$) be the cone of positive (respectively non-zero, non-negative)  definite quadratic forms in    $V_m$. 
The  symmetric space 
for ${\rm GL}_m(\R)$ is given by 
\be \la{MMSa}
{\Bbb H}_{\rm L_m} :=  {\rm {\rm GL}}_m(\R)/{\rm O}(m)\cdot \R^*_+   =  {\cal P}(V_m)/\R_+^*. 
\ee
It is  compactified   by the space $\overline {\Bbb H}_{\rm L_m}:= \overline {\cal P}(V_m)/\R_+^*$.  

Any vector $f \in {\rm L_m}$ defines a degenerate    quadratic form 
$$
\varphi(f):= (f,x)^2 \in \overline {\cal P}(V_m).
$$
Pick a basis $(f_1, \ldots , f_{m})$ of the lattice ${\rm L_m}$. 
Take  the convex hull  of  the vectors $\varphi(f_k)$ in 
the cone $\overline {\cal P}(V_m)$ and project it to  $\overline {\Bbb H}_{\rm L_m}$, getting a simplex  of dimension $m-1$ in the symmetric space (\ref{MMSa}):
\be \la{SIMp}
\varphi(f_1) \ast \ldots \ast \varphi(f_m): = \{\lambda_1 \cdot \varphi(f_1) + ... + \lambda_m \cdot 
\varphi(f_m)\}/\R_{>0}^*, 
\quad \lambda_i \geq 0,
\ee 
 Its boundary lies on the boundary of the symmetric space. 
So for any   subgroup $\Gamma \subset {\rm GL}_m(\Z)$ its projection to the quotient $\Gamma\backslash {\Bbb H}_{\rm L_m}$  produces a class in the  
  Borel-Mooore homology:
\be \la{MMS}
[\varphi(f_1) \ast \ldots \ast \varphi(f_m)] \in {\rm H}^{\rm BM}_{m-1}(\Gamma\backslash {\Bbb H}_{\rm L_m}; \Z) = {\rm H}^{d_m-(m-1)}(\Gamma \backslash {\Bbb H}_{\rm L_m}; \Z), ~~~~d_m = {\rm dim}~{\Bbb H}_{\rm L_m}.
\ee
These classes are  {\it Mazur's modular symbols} \cite{Ma}, \cite{AR}.  They are  the modular symbols   associated to the Borel subgroup 
of ${\rm GL}_m$. They play an important role in arithmetic of 
certain automorphic p-adic L-functions for ${\rm GL}_m$ \cite{KMS}.  
The modular symbols (\ref{MMS})  generate  ${\rm H}^{d_m-(m-1)}(\Gamma \backslash {\Bbb H}_{\rm L_m}; \Z)$, see (\ref{dm}), 
which  is the top dimensional cohomology group of the space $\Gamma \backslash {\Bbb H}_{\rm L_m}$:
$$
{\rm H}^{i}(\Gamma \backslash {\Bbb H}_{\rm L_m}; \Z) =0 ~~~~\mbox{if $i>d_m - (m-1)$}.
$$

Recall the modular manifold 
\be \nonumber
Y_1(m;N)    =   \Gamma_1(m;N) \backslash {\rm GL_m}(\R)/ {\rm O(m)} \cdot \R_+^*= \Gamma_1(m;N) \backslash {\Bbb H}_{\rm L_m}.
\ee 
\bd \la{DEFMMS}
a) The subgroup ${\cal M}{\cal S}^{m}_{(m)}$ of  Borel-Moore $(m-1)-$cycles   in the  symmetric space $ {\Bbb H}_{\rm L_m}$ is generated by 
 the simplices (\ref{SIMp}). 
  
b)   The subgroup ${\cal M}{\cal S}^{m}(Y_1(m; N))$ of  Borel-Moore $(m-1)-$cycles   in the  space $Y_1(m;N)$ is generated by 
  the images of  simplices (\ref{SIMp}). 
  \ed

 \paragraph{2. The key map.}
Let us pick a basis in the lattice ${\rm L}_m$. The group ${\rm GL_m}(\Z)$ acts  on the set of bases of the lattice ${\rm L_m}$. The stabilizer of   simplex (\ref{SIMp}) is the semidirect product 
of the symmetric group $S_m$ and the diagonal subgroup $D_m = (\Z/2\Z)^m $. The subgroup $D_m$ preserves the  simplex orientation;  $S_m$ alters it by 
the  sign character  $S_m \to \{\pm 1\}$.   Therefore the group ${\cal M}{\cal S}^{m}(Y_1(m; N))$  is isomorphic to  the sign-coinvariants of the    group $S_m \times D_m$     acting on 
$\Z[{\rm M}_m(N)]$, where we set
$$
  {\rm M}_m(N):=  \Gamma_1(m;N) \backslash  {\rm GL_m}(\Z).
$$
The action of ${\rm GL_m}(\Z)$ on $(\Z/N\Z)^{m} $ provides a bijection
 \be \la{SMN}
  {\rm M}_m(N) = \left\{\{\alpha_1, ... ,\alpha_m\} \in (\Z/N\Z)^{m} ~|~{\rm g.c.d.}(\alpha_1, ... ,\alpha_m) =1\right\}.
\ee 
For example,  when $N=p$ is a prime, 
${\rm M}_m(p) =  (\Z/p\Z)^{m} -\{0\}$.  
So any   $\{\alpha_1, ... ,\alpha_m\} \in {\rm M}_m(N)$ gives rise to a modular symbol given by a Borel-Moore  $(m-1)-$cycle  
\be \nonumber
\begin{split}
&[ \alpha_{1}]\ast \ldots \ast [\alpha_{m}]  \in {\cal M}{\cal S}^{m}(Y_1(m; N)),\\
&[\pm  \alpha_{\sigma(1)}]\ast  ... \ast [\pm  \alpha_{\sigma(m)}] =   {\rm sgn}(\sigma) ~ [ \alpha_{1}]\ast \ldots \ast [\alpha_{m}] , ~~~~\sigma \in S_m.\\
\end{split}
\ee

 \bd Pick a primitive $N-$th root of unity $\zeta_N$ and consider a map 
\be \la{qwe}
\begin{split}
\gamma_{m}^{m}: ~&{\cal M}{\cal S}^{m}(Y_1(m; N)) \lra  \Lambda^m\widehat {\cal O}^*({\rm S_N}),\\
&[ \alpha_{1}]\ast \ldots \ast [\alpha_{m}]  \lms (1-\zeta_N^{\alpha_1}) \wedge \ldots \wedge (1-\zeta_N^{\alpha_m}).\\
\end{split}
 \ee 
 Here we set $(1-\zeta_N^{\alpha}):= \gamma^{\cal M}$ if $\alpha =0$. 
 \ed 
 
 \bt The map (\ref{qwe}) induces a surjective map of the cohomology groups
\be \la{mmKK}
\begin{split}
&{\rm H}^{\frac{m(m-1)}{2}}(\Gamma_1(m;N), \varepsilon^{m-1})\lra H_{m,m}^m(\widehat {\rm C}_{\bullet, \bullet}(\mu_N)) = \\
&{\rm Coker}\Bigl( {\cal C}_{2}(\mu_N)\otimes \Lambda^{m-2} \widehat {\cal O}^*({\rm S_N}) \stackrel{\delta}{\lra} \Lambda^m \widehat {\cal O}^*({\rm S_N})\Bigr).\\
\end{split}
\ee
 The map (\ref{mmKK}) is an isomorphism when $N=p$ is a prime.

 \et
 
We stress that the map (\ref{qwe}) does depend on the choise of the  basis $(e_1, ..., e_m)$ of the lattice ${\rm L}_m$, and of the   primitive $N-$root of unity $\zeta_N$.


 \paragraph{Acknowledgments.}  
The work   was
supported by the  NSF grants  DMS-1564385 and  DMS-1900743. The final version of the paper was prepared   at   
 IHES    (Bures sur Yvette),  MPI and MSRI (Berkeley) during the Summer and Fall of 2019.  
The hospitality and support of   IHES,  MPI and MSRI, and   the support of the Simons Foundation are  gratefully acknowledged. I am very grateful to Nikolay Malkin for making useful comments on the preliminary draft of the paper, 
and writing the Appendix.

 \section{The main construction}

In Section \ref{SS1.3} we outline the main  construction relating the cohomology of the cyclotomic Lie algebras and the geometry of  modular manifods. 
It  go back to \cite{G99}, \cite{G00b}. 

The construction consists of two parts of rather different nature: geometric and motivic. The  geometric part  is discussed in Sections \ref{SSECC2.2}--\ref{Sec2aa}, and the motivic one in 
Section \ref{SEC8}. 

The main result, Theorem \ref{4.26.01.3a}, is proved in Section \ref{4.26.01.3aX}. Its variant is proved is given in Section \ref{Sec2}. The  combinatorial sceleton of the proof 
is clarified by the results of  Sections \ref{Sec2b}--\ref{SECCTT2.6}.

\paragraph{The key ideas.} We define a collection of canonical elements in each of the spaces ${\cal C}_{w,m}(\mu_N)$. These elements    span  them  
as   vector spaces. For ${\cal C}_{1,1}(\mu_N)$ these are just the cyclotomic units. 

In general these are   motivic correlators/motivic multiple polylogarithms at $N-$th roots of unity, projected onto the associate graded  for the depth filtration. 
They satisfy the {\it double shuffle relations}. The first    is clear from the  definition, the second is much harder to prove. 
This, plus the explicit formulae for the coproduct of motivic correlators,  is  sufficient to get Theorem \ref{4.26.01.3a}. 

\vskip 2mm
To prove Theorems \ref{MTTHHa} and   Theorem \ref{CMR1y}, we   relate  further  the properties of the motivic correlators at $N-$th roots of unity 
to the geometry of modular manifolds $Y_1(m;N)$. 

We assign to each motivic cyclotomic correlator  a canonical $(2m-2)-$chain in the complex calculating the cohomology of the local system 
$S^{w-m}V_m$ on $Y_1(m;N)$. It is defined via the "sum of the plane trivalent trees" construction. 
If $w=m$,  it is just a $(2m-2)-$chain in $Y_1(m;N)$. 
 
\vskip 2mm
 The  key feature of these chains is that their boundaries reflect the coproduct of the corresponding motivic cyclotomic correlators. 
Arguing by the induction, this suggests that   relations between   motivic cyclotomic correlators should give rise to   cycles, at least after suitable corrections.

The $(2m-2)-$chains do satisfy  the  "difficult" motivic  shuffle relations. However,  if $m>2$, 
the "easy" shuffle relations do not hold for these chains. Instead, if $m=3,4$, they are boundaries of certain explicit $(2m-1)-$chains. 
These results, plus the study the Voronoi complexes for  $m=3,4$, deliver Theorems \ref{MTTHHa} and   Theorem \ref{CMR1y}. 
The situation for $m>4$ needs further investigation. 

 \subsection{An outline of the construction} \la{SS1.3}
In short, our goal is to generalize the map (\ref{qwe}): both the source and the target. 

To generalize the source, we start from introduction of the higher modular symbol complex.

We use the following strategy. First of all, we pick  a   lattice 
${\rm L_m}$. 
 
\begin{enumerate}

\item   Our main geometric construction provides   a {\it higher modular symbols complex} ${\cal M}{\cal S}^\ast_{(m)}$. It is a  complex of right ${\rm Aut}({\rm L_m})-$modules of amplitude $[1, m]$:
\be \nonumber
      {\cal M}{\cal S}^{1}_{(m)}   \stackrel{\partial}{\lra}       \ldots  \stackrel{\partial}{\lra}
  {\cal M}{\cal S}^{m-1}_{(m)}   \stackrel{\partial}{\lra}     {\cal M}{\cal S}^{m}_{(m)} .
\ee
 It sits in  the complex of Borel-Moore chains in the symmetric space for the group ${\rm GL_m}(\R)$. 

The elements of   $ {\cal M}{\cal S}^{i}_{(m)} $ are certain $(2m-1-i)-$dimensional  Borel-Moore chains. 

Its last group ${\cal M}{\cal S}^{m}_{(m)}$ is   the   one   from Definition \ref{DEFMMS}. The complex is defined in Section \ref{SSECC2.2}.
  
\item  Axiomatizing some of the properties of the   complex ${\cal M}{\cal S}^\ast_{(m)}$, 
we introduce another   complex of ${\rm Aut}({\rm L_m})-$modules $\widehat {\rm M}^*_{(m)}$ of amplitude $[1, m]$, called the 
 {\it rank $m$ 
 relaxed modular complex}:
 $$
 \widehat  {\rm M}^{1}_{(m)}  \stackrel{\partial}{\lra} \ldots \stackrel{\partial}{\lra} \widehat  {\rm M}^{m-1}_{(m)}  \stackrel{\partial}{\lra}  \widehat  {\rm M}^{m}_{(m)} .
  $$
   It is defined by imposing the dihedral and  the first shuffle relations on the generators. 
   
   The two complexes are related by the  {\it geometric  realization} map  of complexes:
 \be \la{GRM}
 \psi_{(m)}^*: \widehat {\rm M}^*_{(m)} \lra{\cal M}{\cal S}^\ast_{(m)}.
  \ee
  It is surjective by the very definition, and looks as follows: 
   \begin{displaymath} 
    \xymatrix{    
\widehat   {\rm M}^{1}_{(m)} \ar[r]^{}  \ar[d]^{\psi_{(m)}^{1}} &     \ldots \ar[r]^{} & 
\widehat   {\rm M}^{m-1}_{(m)}  \ar[r]^{}  \ar[d]^{ \psi_{(m)}^{m-1}} & \widehat {\rm M}^{m}_{(m)} \ar[d]^{\psi_{(m)}^{m}} \\   
  {\cal M}{\cal S}^{1}_{(m)} \ar[r]^{}&\ldots \ar[r]^{} &{\cal M}{\cal S}^{m-1}_{(m)} \ar[r]^{} & {\cal M}{\cal S}^{m}_{(m)} \\}
\end{displaymath}

  \item  The complex $\widehat   {\rm M}^{\ast}_{(m)}$ has a    quotient $  {\rm M}^{\ast}_{(m)} $,  called the {\it rank $m$ modular complex},  
 It is obtained by imposing the second shuffle relations. So  there is a surjective map of complexes
$$
p^*_{(m)}: \widehat   {\rm M}^{\ast}_{(m)} \lra{\rm M}^{\ast}_{(m)}.
$$

 \item The group ${\rm Aut}({\rm L}_m)$ acts on  the complexes ${\cal M}{\cal S}^\ast_{(m)}$, $\widehat {\rm M}^\ast_{(m)}$ and ${\rm M}^\ast_{(m)}$.  So 
  for any subgroup $\Gamma \subset {\rm Aut}({\rm L}_m)$ and any representation $V$ of ${\rm Aut}({\rm L}_m)$, there are  three complexes:  
 \be \nonumber
\begin{split}
& {\cal M}{\cal S}^\ast_{(m)}\otimes_{\Gamma}V, ~~~~ \widehat {\rm M}^\ast_{(m)}\otimes_{\Gamma}V, ~~~~  {\rm M}^\ast_{(m)}\otimes_{\Gamma}V.\\
\end{split}
 \ee 
In particular,  given a positive integer $N$, we  can consider the following three complexes: 
 \be \la{MSMS}
\begin{split}
{\cal M}{\cal S}_N^\ast:= &\bigoplus_{ m=1}^\infty \bigoplus_{w\geq m} {\cal M}{\cal S}^\ast_{(m)}\otimes_{\Gamma_1(m;N)} S^{w-m}V_m,\\
\widehat {\rm M}_N^\ast:= &\bigoplus_{ m=1}^\infty \bigoplus_{w\geq m} \widehat {\rm M}^\ast_{(m)}\otimes_{\Gamma_1(m;N)} S^{w-m}V_m,\\
 {\rm M}_N^\ast:= &\bigoplus_{ m=1}^\infty \bigoplus_{w\geq m}   {\rm M}^\ast_{(m)}\otimes_{\Gamma_1(m;N)} S^{w-m}V_m.\\
\end{split}
\ee
Each of them carries  three compatible gradings: the cohomological grading $\ast$, and  the gradings by    the weight $w$, and    depth $m$.   
The  $(w, m)$ bigrading      is preserved by the  cohomological differential  $\partial$.

We equip each of the   complexes (\ref{MSMS}) with a product, providing it with a structure of a commutative dg-algebra with the  cohomological grading $\ast$ and   
the differential $\partial$. Each of these dg-algebras are also bigraded by $(w,m)$. We use the following names for them:

 ${\cal M}{\cal S}_N^\ast$: the {\it level $N$ dg-algebra of higher modular symbols}. 

  $\widehat {\rm M}_N^\ast$: the {\it level $N$ relaxed modular dg-algebra}.

 ${\rm M}_N^\ast$: the {\it level $N$ modular dg-algebra}.

  \item The geometric realization maps (\ref{GRM})  give rise to the {\it geometric  realization} map  of dg-algebras 
  \be \la{MRg}
 \psi: \widehat {\rm M}^*_{N} \lra{\cal M}{\cal S}^\ast_{N}.
  \ee
  
  \item Our main motivic construction provides the {\it motivic realization} map  of dg-algebras\footnote{It is tempting to think that there exists a map of dg-algebras: 
  \be\nonumber
  ?: {\cal M}{\cal S}^\ast_{N} \lra \Lambda^\ast \widehat {\cal C}_{\bullet, \bullet}(\mu_N) ~~\mbox{such that} ~~ 
 \gamma\circ p= ?\circ \psi.
 \ee
 We do not know how to prove this. The problem is that   one might have relations between the higher modular symbols which are  not reflected in the definition of the relaxed modular complexes. }
  \be \la{MR}
 \gamma:{\rm M}^*_{N} \lra \Lambda^\ast \widehat {\cal C}_{\bullet, \bullet}(\mu_N).
  \ee
  It uses motivic correlators,    reviewed in Section \ref{SEC8}. 
Alternatively, one can use the motivic multiple polylogarithms, as we did in \cite{G00b} and \cite{G02b}, see Section \ref{Sec2}.  
 
  These maps can be organized into a   diagram of commutative dg-algebras:
 \begin{displaymath} 
    \xymatrix{
     \widehat  {\rm M}_N^\ast \ar[r]^{\psi} 
     \ar[d]^{p}  & {\cal M}{\cal S}^{\ast}_N 
     \\     
{\rm M}_{N}^\ast\ar[r]^{\gamma~~~~~}    &
 \Lambda^\ast \widehat {\cal C}_{\bullet, \bullet}(\mu_N)\\}
\end{displaymath} 
   
   \item The dg-algebra ${\rm M}_{N}^\ast$ is obtained from the one   $\widehat {\rm M}_{N}^\ast$   by imposing the second shuffle relations. However if $m>2$, 
   these relations do not vanish under the geometric realization map (\ref{MRg}). 
   
   This begs for a resolution of dg-algebras ${\rm M}_{N}^\ast$ and    ${\cal M}{\cal S}^\ast_{N}$ - see discussion at the end of Section \ref{4.26.01.3a}.  
   We leave this question for future investigation. 
   In  Section \ref{SEC5} we show  that, for $m=3,4$, the map (\ref{MRg}) sends the other shuffle relations   to boundaries of certain  chains.     
   This plays an important role in the proof of Theorem \ref{MTTHHa}.

     \item Both the generators of the modular dg-algebra and the motivic cyclotomic correlators satisfy the double shuffle relations. 
     In each case there is an "easy" set of shuffle relations, and the "difficult" one. The "easy" shuffle relations for the modular dg-algebra vanish under the geometric realization map. 
     The "easy"  shuffle relations for  motivic correlators are valid on the nose, and easy to prove.
        
         However the motivic  realization map (\ref{MR}) interchanges the "easy"   and  "difficult" shuffle relations!

   \item 
   The {\it sum over all plane trivalent trees} construction, enters   crucially   to both   geometric and   motivic   constructions, and    
  fit together in a   mysterious way. 
  In each case the "easy" shuffle relations follow from the sum over all plane trivalent trees construction.
     \end{enumerate}

\paragraph{An example: $w=m$.} The $w=m$ summand  of  ${\cal M}{\cal S}_N^\ast $ in (\ref{MSMS}) reduces to   the   $\Gamma_1(m;N)-$coinvariants, providing the higher modular symbols complex  for the modular manifold 
$Y_1(m;N)$:
$$
{\cal M}{\cal S}^\ast(Y_1(m;N)):=  {\cal M}{\cal S}^\ast_{(m)}\otimes_{\Gamma_1(m;N)}\Q. 
$$
Its last group  ${\cal M}{\cal S}^m(Y_1(m;N))$  recovers  Mazur's modular symbols for  $Y_1(m;N)$. 
 
We   map   the   modular 
  complex for $Y_1(m;N)$, understood as the  $\Gamma_1(m;N)-$coinvariants of the modular complex, to the weight $m$ part of the  cochain complex of the diagonal  cyclotomic Lie algebra $\widehat {\cal C}_{\bullet}(\mu_N)$:
 \begin{displaymath} 
    \xymatrix{    
  {\rm M}^{1}_{(m)} \otimes_{\Gamma_1(m;N)}\Q\ar[r]^{~~~~~~~~~\partial}  \ar[d]^{\gamma_{m}^{1}} &     \ldots \ar[r]^{\partial~~~~~~~} & 
  {\rm M}^{m-1}_{(m)}\otimes_{\Gamma_1(m;N)}\Q  \ar[r]^{\partial}  \ar[d]^{ \gamma_{m}^{m-1}} & {\rm M}^{m}_{(m)}\otimes_{\Gamma_1(m;N)}\Q \ar[d]^{\gamma_{m}^{m}} \\   
  {\cal C}_m(\mu_N)\ar[r]^{\delta}&\ldots \ar[r]^{\delta~~~~~~~~~~~~~} &{\cal C}_{2}(\mu_N)\otimes \Lambda^{m-2} \widehat {\cal C}_1(\mu_N)\ar[r]^{~~~~~~~\delta} &
   \Lambda^m \widehat {\cal C}_1(\mu_N)\\}
\end{displaymath}

Let us proceed  to realization of the  strategy sketched above.

\subsection{Higher modular symbols complexes} \la{SSECC2.2}

 \paragraph{1. The groups ${\cal M}{\cal S}^{1}_{(m)}$ and ${\cal M}{\cal S}^{1}(Y_1(m; N))$ via Feynman graphs.} 
 {\it Terminology}. 
A tree is a connected graph without loops. The edges of a tree 
consist of legs (i.e. external edges) 
and internal edges.So a leg is an edge of the tree 
such that one or both of its vertices have 
valency one.  A tree may have no internal edges or vertices. We do not allow 
vertices of valency two.

A {\it plane tree} is a tree with a given cyclic order of the edges incident to each vertex. There is 
natural cyclic order of 
the legs of a plane tree. To see it, shrink all internal edges to a point.

An orientation of the tree is an ordering of the set of the edges of the tree considered up to even permutations.   
 A plane  trivalent tree has a canonical orientation. Namely,   take an $\varepsilon-$neighborhood of the tree, and go around the tree 
 counting the edges, skipping the ones which has been counted. 
 
 \begin{figure}[ht]
\centerline{\epsfbox{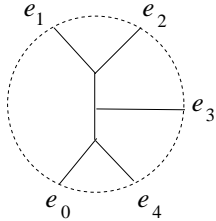}}
\caption{A plane colored tree.} 
\label{mp1aa}
\end{figure}
  
\vskip 2mm 
  {\it Coloring}.    An {\it extended basis} of the lattice ${\rm L_m}$ is a collection of vectors $e_0, e_1, \ldots , e_m$ so that 
  the vectors $e_1, \ldots , e_m$ form a basis of the lattice ${\rm L_m}$, and 
  \be \la{sum}
  e_0+ \ldots + e_m=0.
 \ee

Denote by  $\varphi(v_{1}, ... , v_{k})$ the convex hull of the forms 
$\varphi(v_{1})$, ..., $\varphi(v_{k})$ in the space of quadratic forms. 
It depends only on the vectors $\pm v_i$ defined up to a sign. 

 Take a  plane  trivalent tree ${\rm T}$ whose legs are decorated cyclically 
 by the vectors $e_0, ..., e_m$ of an extended basis for the lattice ${\rm L}_m$,  see Figure \ref{mp1aa}. 
Each edge $E$ of the  tree ${\rm T}$ provides a vector $f_E \in {\rm L_m}$ 
defined up to a sign. Namely, cutting the edge $E$ we get two trees, see Figure \ref{mp1ba}.  
Then $f_E$ is the sum of   the vectors $e_i$ at legs  of one of these trees. 
Thanks to (\ref{sum}),     the second   tree gives $-f_E$. So the  quadratic form $\varphi(f_E)$ is well defined.

\bd We introduce   a $(2m-2)-$chain in the symmetric space ${\Bbb H}_{\rm L_m}$ by setting
\be \la{FFFa}
\begin{split}
&\psi[e_1,...,e_m]:= 
 \sum_{\mbox{T}}{\rm sgn}_{\rm T} (E_1 \wedge ... \wedge E_{2m-1})\cdot 
\varphi(f_{E_1}, ... , f_{E_{2m-1}}).\\
\end{split}
\ee
Here $\{E_1, \ldots , E_{2m-1}\}$ is the set of all edges of the tree ${\rm T}$. The sum is over all plane trivalent trees ${\rm T}$ colored by $e_0,...,e_m$. The sign is for the canonical orientation of the plane trivalent tree.
\ed

\begin{figure}[ht]
\centerline{\epsfbox{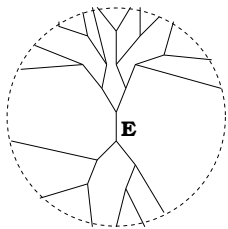}}
\caption{An internal edge $E$ separates a tree into two.} 
\label{mp1ba}
\end{figure}

When we run through the set of all bases in ${\rm L_m}$, we get a collection of  $(2m-2)-$chains, possibly intersecting,  on which the group 
${\rm Aut}({\rm L_m})$ acts. 

\paragraph{\bf Examples.} a) For $m=2$ there is one plane trivalent tree colored by $e_0, e_1, e_2$, so 
we get the triangle $\varphi(e_0, e_1, e_2)$, see the left picture on Figure \ref{mp1c}.  
\begin{figure}[ht]
\centerline{\epsfbox{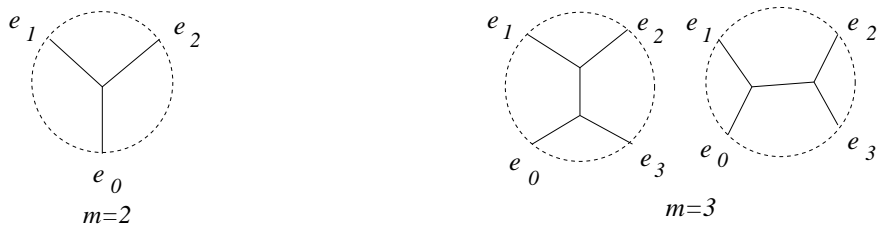}}
\caption{The plane colored trivalent trees for $m=2,3$.} 
\label{mp1c}
\end{figure}
So for $m=2$ we get   the   triangles of the modular triangulation of the hyperbolic plane, see Figure \ref{mod}. 

b) Let $f_{ij}:= e_i+e_j$. 
For $m=3$ there are two  plane trivalent trees colored by $e_0, e_1, e_2, e_3$, 
see the two pictures on the right on Figure \ref{mp1c}, so 
the chain is 
\be
\varphi(e_0, e_1, e_2, e_3, f_{01}) - \varphi(e_0, e_1, e_2, e_3, f_{12}). 
\ee

c) For $m=4$ there are five  plane trivalent trees colored by $e_0, e_1, e_2, e_3, e_4$.  
We showed one of them on Figure \ref{mp1d}; the rest are obtained by rotations of the tree. 
Therefore  
the chain is 
\be \la{Mappsi}
{\rm Cycle}_5 \varphi(e_0, e_1, e_2, e_3, e_4, f_{12}, f_{34}). 
\ee

\begin{figure}[ht]
\centerline{\epsfbox{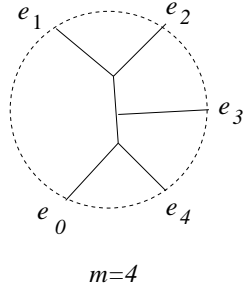}}
\caption{A plane colored trivalent tree  for $m=4$.} 
\label{mp1d}
\end{figure}

d) For $m=5$ there are  eleven  plane trivalent trees colored by $e_0, e_1, e_2, e_3, e_4, e_5$.  
They are obtained by rotations of the trees shown on Figure \ref{mp1e}. 
Therefore  
the chain is 
\be
\begin{split}
{\rm Cycle}_6 \Bigl( &\frac{1}{3}\varphi(e_0, e_1, e_2, e_3, e_4, e_5, f_{01}, f_{23}, f_{45})  + 
\varphi(e_0, e_1, e_2, e_3, e_4, e_5, f_{01}, f_{501}, f_{23}) +\\
&  \frac{1}{2}\varphi(e_0, e_1, e_2, e_3, e_4, e_5, f_{01}, f_{345}, f_{34}) 
\Bigr).\\
\end{split}
\ee
\begin{figure}[ht]
\centerline{\epsfbox{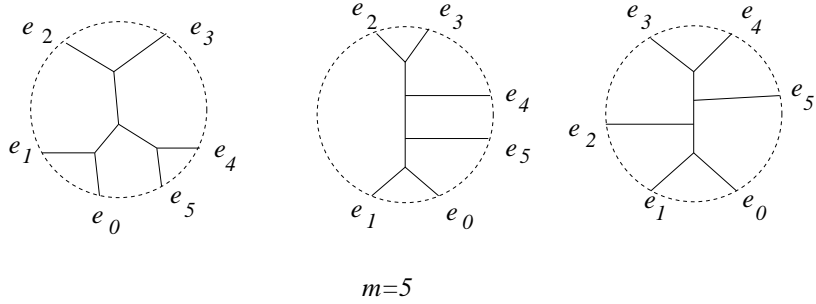}}
\caption{Three types of plane colored trivalent trees  for $m= 5$.} 
\label{mp1e}
\end{figure}

\bd i) The subgroup 
$$
{\cal M}{\cal S}^{1}_{(m)}\subset {\rm C}^{\rm BM}_{2m-2}({\Bbb H}_{{\rm L}_m}) 
$$ 
is spanned by  Borel-Moore chains $\psi[e_1, ..., e_m]$ in (\ref{FFFa})  
 assigned to all bases $(e_1, \ldots , e_m)$ of   ${\rm L_m}$.
 
ii) The subgroup 
$$
{\cal M}{\cal S}^{1}(Y_1(m; N))\subset {\rm C}^{\rm BM}_{2m-2}(Y_1(m; N)).
$$ 
 is spanned by the projections  to  $Y_1(m; N)$ of the chains $\psi[e_1, ..., e_m]$   
for bases $(e_1, \ldots , e_m)$ of  ${\rm L_m}$.\ed

\paragraph{5. The higher modular symbols complex ${\cal M}{\cal S}^{\ast}({\Bbb H}_{\rm L_m})$.} 
Take   non-negative definite quadratic forms $Q_m$ and $Q_n$ in  vector space ${\rm L}_m\otimes \R$ and  ${\rm L}_n\otimes \R$, lifting points in 
$\overline {\Bbb H}_{{\rm L}_{m}}$ and $\overline {\Bbb H}_{{\rm L}_{n}}$, respectively.  Consider 
each of them as quadratic forms in $({\rm L}_m \oplus {\rm L}_n)\otimes \R$ by extending it by zero to the second summand, and take the convex hull  of their positive multiples,  given by
$$
\lambda_1 Q_m + \lambda_2 Q_n, ~~~~\lambda_1 + \lambda_2 >0, ~~\lambda_1,  \lambda_2\geq 0.
$$
Finally, project it back to $\overline {\Bbb H}_{{\rm L}_{m}\oplus {\rm L}_n}$. 
Applying this construction to subsets of  $\overline {\Bbb H}_{{\rm L}_{m}}$ and $\overline {\Bbb H}_{{\rm L}_{n}}$ rather than   points, we get a map of complexes of  Borel-Moore chains:
$$
\ast: {\rm C}^{\rm BM}_i({\Bbb H}_{{\rm L}_m}) \otimes {\rm C}^{\rm BM}_j({\Bbb H}_{{\rm L}_n}) \lra {\rm C}^{\rm BM}_{i+j+1}({\Bbb H}_{{\rm L}_{m}\oplus {\rm L}_{n}}).
$$
We transform the homological grading into a cohomological one, with a shift by $2m-1$:
$$
{\rm C}^i({\Bbb H}_{{\rm L}_m}):= {\rm C}_{2m-1-i}({\Bbb H}_{{\rm L}_m}).
$$
Then we get a map 
\be \la{AST}
\begin{split}
&\ast: {\rm C}_{\rm BM}^i({\Bbb H}_{{\rm L}_m}) \otimes {\rm C}_{\rm BM}^j({\Bbb H}_{{\rm L}_n}) \lra {\rm C}_{\rm BM}^{i+j}({\Bbb H}_{{\rm L}_{m}\oplus {\rm L}_n}).\\
\end{split}
\ee

Now take a decomposition of the lattice ${\rm L}_{m}$ into a direct sum of $k$ sublattices  ${\rm L_{m_1}}, \ldots , {\rm L_{m_k}}$:
 \be \la{dec}
{\rm L_{m_1}}\oplus  \ldots \oplus {\rm L_{m_k}} \stackrel{\sim}{\lra} {\rm L}_{m}.
 \ee
  Then the   $\ast-$product map (\ref{AST}) for the tensor product   $\otimes_{i=1}^k{\rm C}_{\rm BM}^*({\Bbb H}_{{\rm L}_{m_i}})$ restrictis to the map
\be \la{decu}
\ast_{{\rm L_{m_1}}, \ldots, {\rm L_{m_k}} }: {\cal M}{\cal S}^{1}({\Bbb H}_{{\rm L}_{m_1}}) \otimes \ldots \otimes {\cal M}{\cal S}^{1}({\Bbb H}_{{\rm L}_{m_k}}) \lra {\rm C}_{\rm BM}^{k}({\Bbb H}_{{\rm L}_m}).
\ee
It is graded commutative: interchanging any two factors we get a minus sign. 

\bd The group ${\cal M}{\cal S}^{k}({\Bbb H}_{\rm L_m})$ is
 the sum of the images of the maps (\ref{decu}) over all direct sum decompositions (\ref{dec}): 
\be \nonumber
{\cal M}{\cal S}^{k}({\Bbb H}_{\rm L_m}) := \sum_{{\rm L}_m= {\rm L_{m_1}} \oplus \ldots \oplus {\rm L_{m_k}} } {\rm Im}\Bigl(\ast_{{\rm L_{m_1}}, \ldots, {\rm L_{m_k}} }\Bigr).
\ee
 \ed
 
Explicitely, take a basis $e_1,...,e_m$ of ${\rm L_m}$,  a partition $m = m_1 + \ldots + m_k$,  and  consider a cochain
\be \la{5467}
\psi[e_1,...,e_m]_{m_1, ..., m_k}:=\psi [e_1,...,e_{m_1}] \ast  \psi [e_{m_1+1},...,e_{m_1+m_2}]\ast  \ldots \ast \psi[e_{m_1+...+m_{k-1}+1} ,...,e_{m}].
\ee
The group ${\cal M}{\cal S}^{k}({\Bbb H}_{{\rm L}_m})$ is generated by  these cochains.

 \bt The differential in ${\rm C}_{\rm BM}^{\ast}({\Bbb H}_{\rm L_m})$ provides a differential on  ${\cal M}{\cal S}^{\ast}({\Bbb H}_{\rm L_m})$. 
   \et
  
  \begin{proof} The boundary of the chain (\ref{FFFa}) is given by the following formula, reflected on Figure \ref{mp11B}. Let us introduce an extra vector $e_0$ such that $e_0+e_1 + \ldots + e_m=0$. 
  Then 
  \be \la{DIC}
  \partial \psi[e_1,...,e_m] = -\sum_{i \not = j \in \Z/(m+1)\Z}  \psi[e_{i+1},...,e_j] \ast \psi[e_{j+1},...,e_{i-1}].
  \ee  
  The proof of this formula is given in the part 5) of Section \ref{Sec4.2}.

 \begin{figure}[ht]
\centerline{\epsfbox{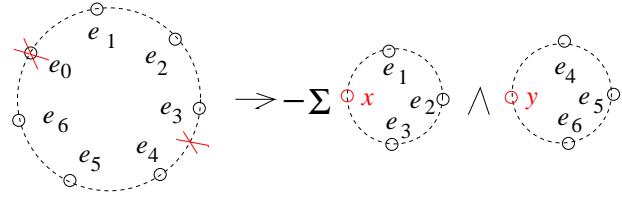}}
\caption{$x= - e_1-e_2-e_3$, $y=- e_4-e_5-e_6$. } 
\label{mp11B}
\end{figure}

 \end{proof}
 
 \paragraph{2. The level $N$ dg-algebras of higher modular symbols.}

 \begin{proof}  Denote by 
 \be \la{2m2ch}
 \psi[\alpha_1, ..., \alpha_m] \in {\cal M}{\cal S}^{1}(Y_1(m; N)), ~~~~\{\alpha_1, ..., \alpha_m\}\in {\rm M}_m(N) 
 \ee
  the $(2m-2)-$chain obtained by the projection of the chain 
 $\psi[e_1, ..., e_m] \in {\cal M}{\cal S}^{1}({\Bbb H}_{\rm L_m})$ such that $\{\alpha_1, ..., \alpha_m\}\in {\rm M}_m(N)$ describes the coset of the basis $(e_1, ..., e_m)$ of ${\rm L_m}$ 
 in the standard  basis. 
 
Recall the chosen basis $e_1,...,e_m$ of the lattice ${\rm L}_m$. Take a partition $m = m_1 + \ldots + m_k$. It provides a decomposition 
 ${\rm L}_m = {\rm L}_{m_1} \oplus \ldots \oplus {\rm L}_{m_k}$ where ${\rm L}_{m_1}:= \langle e_1,...,e_{m_1}\rangle$ etc..

Now take a collection of elements 
 $$
  (\alpha_1, ..., \alpha_{m_1}) \in {\rm M}_{m_1}(N), ~~   \ldots , ~~ (\alpha_{m- m_{k}+1} ,...,e_{m}) \in  {\rm M}_{m_k}(N).
  $$
We introduce the product $\ast$ by defining the   element
$$
\psi [\alpha_1,...,\alpha_{m_1}] \ast  \psi [\alpha_{m_1+1},...,\alpha_{m_1+m_2}]\ast  \ldots \ast \psi[\alpha_{m - m_{k}+1} ,...,e_{m}]\in {\cal M}{\cal S}^{k}(Y_1(m; N))  
$$
to be the chain obtained by the projection of the chain (\ref{5467}) to the modular manifold $Y_1(m; N)$.  \end{proof}
 
 \paragraph{Remark.} The elements $\psi [\alpha_1,...,\alpha_{p}]$ do not generate the algebra ${\rm M}^\ast$. Indeed, the latter are defined only for 
  $\{\alpha_1, ..., \alpha_p\}\in {\rm M}_p(N)$, while in  ${\rm M}^k$  we have elements $[\alpha_1, ..., \alpha_m]_{m_1, ..., m_k}$ corresponding to any $\{\alpha_1, ..., \alpha_m\}\in {\rm M}_m(N)$, e.g. 
  to the element  $(0, ..., 0, 1)$.

\subsection{Modular complexes} \la{Sec2aa}

In Section \ref{Sec2aa} we recall the definition of the rank $m$ modular complex \cite{G98},  \cite[Section 2]{G00a} 
and define the modular bicomplex. They are (bi)complexes of ${\rm GL}_m(\Z)$-modules,  
 defined purely combinatorially. Their 
  structure is   related with  the structure of the   
symmetric space 
for ${\rm GL}_m(\R)$ 
and its boundary considered as ${\rm GL}_m(\Z)$-sets. 
The precise relationship with the geometry of 
symmetric spaces is explored 
in the next Sections.

\paragraph{1. The modular complex.} 
Let ${\rm L_m}$ be a rank $m$ lattice. 
The modular complex $
{\rm M}_{(m)}^{*}$ is a complex of left ${\rm GL}_m(\Z)$-modules 
  sitting in the degrees $[1,m]$. It looks as follows:
$$
{\rm M}_{(m)}^{*}:=   {\rm M}_{(m)}^{1} \stackrel{\partial}{\lra} {\rm M}_{(m)}^{2} \stackrel{\partial}{\lra}
... \stackrel{\partial}{\lra} {\rm M}_{(m)}^{m}.
$$

We start with the following two versions of a basis of the lattice ${\rm L_m}$.

 \begin{itemize}
 
 \item 
An  {\it extended basis} of a lattice ${\rm L_m}$ is an  ordered 
$(m+1)$-tuple of the lattice vectors 
$$
\{v_1, ..., v_{m+1}\} \quad \mbox{such that} \quad v_1 + ... + v_{m+1} = 0 
\quad\mbox{and $v_1,...,v_m$ is a basis of ${\rm L_m}$}.
$$ 
If $v_1,...,v_{m+1}$ is a an extended basis of ${\rm L_m}$, then
 omitting any vector $v_i$ we get a basis. 
The group ${\rm GL}_m(\Z)$ acts from the left on the set of basis of ${\rm L_m}$, considered as columns
of vectors $(v_1, ..., v_m)$. Thus the set of extended basis is a
 left principal homogeneous ${\rm GL}_m(\Z)$-set. 

\item

 Let $u_1, ..., u_{m+1}$ be elements of the lattice ${\rm L_m}$ such that the set of elements 
$\{(u_i, 1)\}$ form a basis of  ${\rm L_m} \oplus \Z$. The lattice ${\rm L_m}$ acts on 
such sets by 
$
l: \{(u_i, 1)\} \lms \{(u_i+l, 1)\}
$.  
We call the coinvariants of this action 
{\it homogeneous affine basis} of ${\rm L_m}$ and denote them by $\{u_1: ... : u_{m+1}\}$. There is a canonical bijection 
\begin{equation} \label{5.26.1}
\begin{split}
&\{\mbox{homogeneous affine basis of ${\rm L_m}$}\} \quad \longleftrightarrow \quad 
\{\mbox{extended basis of ${\rm L_m}$}\}.\\
&~~~~~~~~~\{u_1: ... : u_{m+1}\}   \longleftrightarrow    \{u_2-u_1, u_3-u_2, ..., u_1-u_{m+1}\}.\\
\end{split}
\end{equation} 
\end{itemize}

  \paragraph{ The abelian group ${\rm M}_{(m)}^{1}$.}  
It is generated  by the elements $\langle v_1, ... , v_{m+1}\rangle$ corresponding to 
extended  basis $v_1, ... , v_{m+1}$ of ${\rm L_m}$. 
To list the relations, called  {\it the double shuffle relations}, 
we need another set of   generators corresponding to 
homogeneous affine bases of ${\rm L_m}$ via (\ref{5.26.1}):
$$
\langle u_1: ... :u_{m+1}\rangle:=  \langle v_1, v_2, ..., v_{m+1}\rangle, \quad v_i:= u_{i+1}-u_i.
$$

\vskip 2mm
{\it The double shuffle relations}. 

 Denote  by  $\Sigma_{p,q}$ is the set of 
all shuffles of $\{1,...,p\}$ and $\{p+1,...,p+q\}$, i.e. 
permutations  $\sigma \in S_{p+q}$  such that $\sigma(1) < ... < \sigma(p)$ 
and $\sigma(p+1) < ... < \sigma(p+q)$.    
Then:

\begin{itemize}

\item 
For $m=1$ we have $\langle v_1, v_2\rangle = \langle v_2, v_1\rangle$. 

  For any $1 \leq k \leq m$ one has:
\begin{equation} \label{sshh1}
\sum_{\sigma \in \Sigma_{k,m-k}} 
\langle v_{\sigma(1)}, ... ,v_{\sigma(m)}, v_{m+1}\rangle   = \  0.
\end{equation}
 \begin{equation} \label{sshh3}
\sum_{\sigma \in \Sigma_{k,m-k}} \langle u_{\sigma(1)}: ... :u_{\sigma(m)}: u_{m+1}\rangle   =   0.
\end{equation} 
\end{itemize}

\begin{theorem} \label{d3} Double shuffle relations (\ref{sshh1})-(\ref{sshh3}) 
imply   dihedral symmetry relations for $m \geq 2$:
$$
\langle v_{1}, ... ,v_{m}, v_{m+1}\rangle   = \langle v_{2}, ... ,v_{m+1}, v_{1}\rangle   =  \langle -v_{1}, ... ,-v_{m}, -v_{m+1}\rangle   = (-1)^{m+1}\langle v_{m+1}, v_{m}, ... , v_{1}\rangle .  
$$
\end{theorem}

\begin{proof} See  \cite[Theorem 2.7]{G00a}.  It follows that similar 
relations hold for the $u$-generators. \end{proof}

We picture  both types of   generators on an
 oriented circle as shown on Figure \ref{mp11AA}. The circle has $m+1$ (black) $u$-slots, and  $m+1$ (white) $v$-slots located  between the $u$-slots. 
The sum of the vectors  at the $v$-slots is zero. The vectors at the $u$-slots are defined up to a shift by a lattice vector. They are related by (\ref{5.26.1}). Reversing the orientation 
of the circle we change the sign of the generator by $(-1)^{m+1}$.
\begin{figure}[ht]
\centerline{\epsfbox{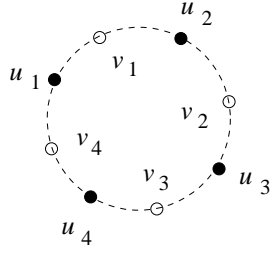}}
\caption{An oriented circle decorated by the   vectors $v_i$   and   $u_j$.} 
\label{mp11AA}
\end{figure}


  \paragraph{ The group ${\rm M}_{(m)}^{k}$.}  For each decomposition of ${\rm L_m}$ as a direct 
sum of $k$ non zero lattices ${\rm L_i}$, we consider the tensor product of the 
${\rm M}^{1}({\rm L_i})$. We consider 
 ${\rm M}^{1}({\rm L_i})$ as odd, and use the sign rule to identify $\otimes {\rm M}^{1}({\rm L_i})$ and 
$\otimes {\rm M}^{1}({\rm L_{i'}})$ when the decomposition ${\rm L}$ and ${\rm L}'$ differ only in the 
ordering of the factors. The group ${\rm M}_{(m)}^{k}$ is defined to be the sum over 
all such unordered 
decomposition ${\rm L_m}=\oplus {\rm L_i}$ of the corresponding 
$\otimes {\rm M}^{1}({\rm L_i})$.  
In other words it is generated by the elements $\langle A_1\rangle \wedge ... \wedge \langle A_k\rangle  $ where 
$A_i$ is an extended  basis of the sublattice ${\rm L_i}$ and $\langle A_i\rangle  $'s anticommute. 

Let us set 
$$
[v_1,...,v_p]:= \langle v_1,...,v_p, v_{p+1}\rangle, \quad v_1 + ... + v_p + v_{p+1} = 0.
$$
Then the group ${\rm M}_{(m)}^{k}$ is generated by  the symbols
$$
[A_1] \wedge ... \wedge [A_k]:= \quad [v_1,...,v_{m_1}]\wedge ... \wedge [v_{m_{k-1}+1},...,v_{m_k}].
$$
where $m= m_k$ and $(v_1,...,v_m)$ is
a basis of ${\rm L_m}$.

We define a homomorphism  
$\partial: {\rm M}_{(m)}^{1} \lra {\rm M}_{(m)}^{2}
$  
by setting  
\be \la{FD}
\partial:  \langle v_1,...,v_{m+1}\rangle   \lms   -{\rm Cycle}_{m+1}\Bigl(\sum_{k=1}^{m-1} 
[v_1,...,v_k] \wedge [v_{k+1},...,v_m] \Bigr),
\ee
where  the indices are modulo $m+1$  and
$$
{\rm Cycle}_{m+1} \Bigl(f(v_1,...,v_{m})\Bigr):=  \sum_{i=1}^{m+1}f(v_{i},...,v_{i+m}).
$$
This formula amounts to the following geometric procedure, see Figure \ref{mp11}. We cut the 
circle at a $v$-slot, and at a $u$-slot, 
and make two oriented circles out of the obtained arcs. Each of the arcs inherits 
some $v$-slots, including the one which was cut. We put at  the $v$-slots 
but the cutted one  the corresponding vectors $v_i$. At the cutted slot 
we put the minus sum 
of the  vectors sitting at the rest of the slots on this arc. Thus we get a generator out of each of the arcs. 
Then we take the (minus)
wedge product of these two generators, 
and  sum over all possible cuts. 

\begin{figure}[ht]
\centerline{\epsfbox{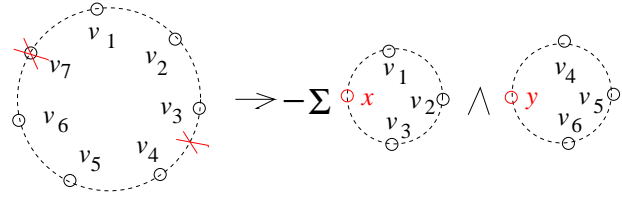}}
\caption{$x= - v_1-v_2-v_3$, $y=- v_4-v_5-v_6$. } 
\label{mp11}
\end{figure}

We get the differential in  ${\rm M}_{(m)}^{\bullet}$ 
by extending   $\partial$ 
using  the Leibniz rule:
$$
\partial( [A_1] \wedge [A_2] \wedge ... ) :=  
\partial [A_1]  \wedge [A_2] \wedge ...    -     [A_1] \wedge \partial [A_2]  \wedge ...    + \ldots
$$

 \begin{theorem} \label{16} a) The differential $\partial$ is a well defined homomorphism of abelian groups. 

b) One has $\partial^2=0$.
\end{theorem}

\begin{proof} This is   \cite[Theorem 2.8]{G00a}. \end{proof}


\paragraph{\bf 2. The modular bicomplex.}  
The {\it  rank $m$ modular bicomplex} 
${\rm M}_{(m)}^{\bullet, \bullet}({\rm L}_m)$, or simply ${\rm M}_{(m)}^{\bullet, \bullet}$,  
is a bicomplex of left 
 ${\rm GL}_m(\Z)$-modules of  the following shape:
\begin{displaymath} 
    \xymatrix{
        {\rm M}_{(m)}^{1,0} \ar[r]^{\partial}  & {\rm M}_{(m)}^{2,0} \ar[d]^{\partial'}\ar[r]^{\partial}&\ldots\ar[r]^{\partial} &  {\rm M}_{(m)}^{m,0}\ar[d]^{\partial'} \\     
 & {\rm M}_{(m)}^{2,1} \ar[r]^{\partial} &\ldots\ar[r]^{\partial} &\ar[d]^{\partial'} {\rm M}_{(m)}^{m,1}\\
 &&\ldots &\ldots \ar[d]^{\partial'} \\
 &&&{\rm M}_{(m)}^{m,m}\\}
\end{displaymath}

The group ${\rm M}_{(m)}^{s, r}$ sits in the bidegree $(s, r)$. The top row ${\rm M}_{(m)}^{\bullet, 0}$ of the bicomplex ${\rm M}_{(m)}^{\bullet, \bullet}$ 
is the   modular complex of the lattice ${\rm L_m}$. So by definition one has 
${\rm M}_{(m)}^{s,0} = {\rm M}_{(m)}^{s}$. 
The total complex associated with the bicomplex ${\rm M}_{(m)}^{\bullet, \bullet}$  is  called the rank $m$ 
{\it extended  modular complex}  and denoted by ${\Bbb M}_{(m)}^{\bullet}$. 
The modular complex is a quotient of the extended modular complex. 
To define the rest of the groups we set
$$
{\rm M}_{(m)}^{\bullet, r} := \quad \bigoplus_{{\rm L} \subset {\rm L_m}}M^{\bullet, 0}({\rm L})
$$
where the sum is over all corank $r$ sublattices of ${\rm L_m}$. 
So the group ${\rm M}_{(m)}^{s, r}$ is generated by  symbols
$$
[A_1]\wedge ... \wedge [A_{s-r}]  
$$
where  $[A_1, ..., A_{s-r}]$ 
is a basis of a corank $r$ sublattice of ${\rm L_m}$, each of the blocks 
$[A_i]$ 
satisfies the double shuffle relations, and they anticommute.

For any corank $r$ sublattice $L \subset {\rm L_m}$ one has a natural 
inclusion of the bicomplexes
$$
{\rm M}_{(m-r)}^{\bullet, \bullet}(L)[-r,-r] \subset  
{\rm M}_{(m)}^{\bullet, \bullet}({\rm L_m}).
$$
Here $[-r,-r]$ means the shift of the bicomplex by $[r,r]$ to 
the right and down.

\vskip 2mm
{\it The differential $\partial'$}.
 Consider an element $[v_1] \wedge ... \wedge [v_p] \wedge X$ where in $X$ all the blocks $[A_i]$ has length $ >1$. For such an element we set 
$$
\partial'
([v_1] \wedge ... \wedge [v_p]\wedge X ):=  \sum_{i=1}^{p}(-1)^{i-1} 
[v_1] \wedge ... \wedge  [\widehat v_i]\wedge ... \wedge [v_p] \wedge X.
$$
Then it is clearly well defined and one has $\partial'^2 =0$. 
  
\begin{proposition} \label{16a} The differentials $\partial$ and $\partial'$ provides ${\rm M}_{(m)}^{\bullet, \bullet}$ with a structure of a bicomplex. 
\end{proposition}

\begin{proof} One can easily see that it is enough to check that the composition
$$
{\rm M}_{(m)}^{1,0} \stackrel{\partial}{\lra} {\rm M}_{(m)}^{2,0} \stackrel{\partial'}{\lra} {\rm M}_{(m)}^{2,1} 
$$
is zero, which is clear from 
$$
\partial' \circ \partial \langle v_1,...,v_{m+1}\rangle  = 
- \partial'    {\rm Cycle}_{m+1}\Bigl([v_1] \wedge [v_2,...,v_m] + 
[v_2,...,v_m]\wedge [v_{m+1}]\Bigr)  = 0.
$$
\end{proof}

\paragraph{3. Modular complexes and modular cohomology of a subgroup $\Gamma \subset 
{\rm GL}_m(\Z)$.}
\begin{definition} \label{35} Let $\Gamma \subset {\rm GL}_m(\Z)$ be a finite index 
subgroup and $V$ a $\Q$-rational ${\rm GL}_m$-module. 
The modular and  extended modular complexes  
 of $\Gamma$  with coefficients in $V$ are 
defined as follows:
\be
{\rm M}_{(m)}^{\ast}(\Gamma, V):=  {\rm M}^{\ast}_{(m)}\otimes_{\Gamma} V  , \qquad 
{\Bbb M}_{(m)}^{\ast}(\Gamma, V):=   {\Bbb M}^{\ast}_{(m)} \otimes_{\Gamma} V. 
\ee
Their cohomology  are   the modular and extended modular cohomology 
of $\Gamma$ with coefficients in $V$:
$$
{\rm HM}_{(m)}^{\ast}(\Gamma, V), \qquad {\rm H}{\Bbb M}_{(m)}^{\ast}(\Gamma, V).
$$
\end{definition}

Denote by ${\cal F}(X)$ the space of $\Q$-valued functions on a set $X$. 
If $X$ is a right $G$-set then the group $G$ acts on ${\cal F}(X)$ from the left by 
$(T_gf)(x):= f(xg)$. 
By  Shapiro's lemma, for a  finite index subgroup 
$\Gamma \subset {\rm GL}_m(\Z)$ and a  
 ${\rm GL}_m$-module $V$  there is a canonical isomorphism of complexes
\begin{equation} \nonumber
\begin{split}
&{\rm M}_{(m)}^{\ast}(\Gamma, V)     =  
{\rm M}^{\ast}_{(m)}
\otimes_{{\rm GL}_m(\Z)} \Bigl( {\cal F}(\Gamma \backslash {\rm GL}_m(\Z))  \otimes_{\Q} V\Bigr). \\
&
{\Bbb M}_{(m)}^{\ast}(\Gamma, V)  =    {\Bbb M}^{\ast}_{(m)} \otimes_{{\rm GL}_m(\Z)} 
\Bigl({\cal F}(\Gamma \backslash {\rm GL}_m(\Z))  \otimes_{\Q} V\Bigr). \\
\end{split}
\end{equation}

\paragraph{4. Modular bicomplexes for ${\rm GL}_2$ and ${\rm GL}_3$.}  Here are 
explicit descriptions of   modular bicomplexes 
${\rm M}_{(m)}^{\bullet}$ for $m=1,2,3$.

{\bf 1. $m=1$}. Then $[v] = [-v]$, and  ${\rm M}_{(1)}^{1} = \Z$.

{\bf 2. $m=2$}. Then the  modular bicomplex is 
\begin{displaymath} 
    \xymatrix{   
  {\rm M}_{(2)}^{1,0} \ar[r]^{\partial} &\ar[d]^{\partial'} {\rm M}_{(2)}^{2,0}\\
 &{\rm M}_{(2)}^{2,1}\\}
\end{displaymath}

The group ${\rm M}_{(2)}^{1,0}$ is generated by the symbols $\langle v_0, v_1,v_2 \rangle $ where $ v_0 +
v_1 + v_2 =0$  and  $(v_1,v_2)$ run through all the basises in $L_2$. 
The differentials are:
\begin{equation} \label{34}
\begin{split}
&\partial: \langle v_0, v_1,v_2\rangle    \lms   -[v_1] \wedge [v_2] -  [v_2] \wedge [v_0] -  [v_0] \wedge [v_1], \\
&\partial': [v_1] \wedge [v_2]   \lms   [v_2] - [v_1]. \\
\end{split}
\ee

\begin{lemma} The double shuffle relations for $m=2$ are equivalent to the  dihedral relations:
 $$
\langle v_0, v_1,v_2\rangle  = - \langle v_1,v_0,v_2\rangle   = - \langle v_0, v_2,v_1\rangle = \langle -v_0,-v_1,-v_2\rangle.
$$
\end{lemma}

\begin{proof}  See  \cite[Lemma 7.10]{G00a}.  \end{proof}

{\bf 3. $m=3$}. The modular bicomplex looks as follows:
\begin{displaymath} 
    \xymatrix{
        {\rm M}_{(3)}^{1,0} \ar[r]^{\partial}  & {\rm M}_{(3)}^{2,0} \ar[d]^{\partial'}\ar[r]^{\partial}&  {\rm M}_{(3)}^{3,0}\ar[d]^{\partial'} \\     
 & {\rm M}_{(3)}^{2,1} \ar[r]^{\partial} &\ar[d]^{\partial'} {\rm M}_{(3)}^{2,1}\\
 &&{\rm M}_{(3)}^{2,2}\\}
\end{displaymath}

The differentials are
\be \nonumber
\begin{split}
&\partial: \quad \langle v_0, v_1, v_2, v_3\rangle  \lms 
- [v_1,v_2] \wedge [v_3]  -  [v_2,v_3] \wedge [v_0]  - [v_3,v_0] \wedge [v_1] - [v_0,v_1] \wedge [v_2] \\
& -[v_0] \wedge [v_1,v_2]   -  [v_1] \wedge [v_2,v_3] -  [v_2] \wedge [v_3, v_0]- [v_3] \wedge [v_0, v_1] ,\\
&\partial: [v_1,v_2]\wedge [v_3]   \lms  -\Bigl(   [v_1]\wedge [v_2] +  [v_2]\wedge  
 [-v_1 - v_2]  +   [-v_1 - v_2]\wedge [v_1]\Bigr)\wedge [v_3], \\
&\partial': [v_1,v_2]\wedge [v_3]   \lms   - [v_1,v_2], \\
&\partial': [v_1] \wedge [v_2] \wedge [v_3] \lms   [v_2]\wedge [v_3] - 
[v_1]\wedge [v_3]  + [v_1]\wedge [v_2]. \\
\end{split}
\ee

\paragraph{\bf 5. The relaxed modular complex.}  Consider a version ${\widehat {\rm M}}_{(m)}^{\bullet, \bullet}$ of the modular bicomplex, called the {\it relaxed modular bicomplex}, 
where the group ${\widehat  {\rm M}}^{1,0}_{(m)}$ is defined using the same generators $[v_1,...,v_m]$ 
satisfying only 
the dihedral symmetry relations and 
the first shuffle relations (\ref{sshh1}).  We do not impose 
the second shuffle relations (\ref{sshh3}). The other groups are defined   similarly: we take the same generators and forget about the second shuffle relations. The differentials are the same as before. Just as before we define 
 the   relaxed modular complex ${\widehat  {\rm M}}^{(m)}_{\bullet}$ and the 
complex ${\widehat {\Bbb M}}^{(m)}_{\bullet}$. 
 
\paragraph{6. Proof of Theorem \ref{THEOREM1.4}.}  
Given any collection of subgroups $\Gamma_{(m)} \subset {\rm GL}_m(\Z)$, $m=1, 2, \ldots $, and  injective maps 
$\Gamma_{(p)}  \times \Gamma_{(q)} \lra \Gamma_{(p+q)}$ compatible with the natural block diagonal 
embeddings  ${\rm GL}_p(\Z) \times {\rm GL}_q(\Z) \lra {\rm GL}_{p+q}(\Z)$, the space  
 \be \nonumber
{{\rm M}^\Gamma_N}^\ast:= \bigoplus_{m=1}^\infty\bigoplus_{w\geq m} {\rm M}^\ast_{(m)}\otimes_{\Gamma_{(m)}} S^{w-m}V_m 
\ee
has a natural commutative dg-algebra structure. Namely, for each positive integer $m$ we  pick a rank $m$ lattice ${\rm L}_m$ with a basis $(e_1, ..., e_m)$. 
We identify ${\rm Aut}({\rm L}_m) = {\rm GL}_m(\Z)$ and use the natural map ${\rm L}_p \oplus {\rm L}_q \lra {\rm L}_{p+q}$, provided by the chosen bases, as well as the 
product map $S^aV_p \otimes S^bV_q \lra  S^{a+b}(V_p\oplus V_q)$ to define the algebra structure on the generators. The claim that it respects the double shuffle relations is evident.  
Together with the differential on the modular complexes, we get a commutative dg-algebra.

When  $\Gamma_{(m)} := \Gamma_1(m; N)$, we get a commutative dga, called the   level $N$ modular dg-algebra. 

Another interesting case is when $\Gamma_{(m)} := \Gamma(m; N)$ is the full level $N$ congruence subgroup. 

The construction in these two cases is compatible with taking the projective limit over $N$, over the directed system for the  $N|M$ relation. 

\subsection{Motivic correlators} \la{SEC8}

Motivic correlators were defined in \cite{G08}. 
By the   definition, they live in the motivic Lie coalgebra.

 \vskip 2mm

 Let $\{s_0, s_1, ..., s_n\}$ be a collection of distinct points on ${\Bbb P}^1(\F)$, and $v_0$ a non-zero tangent vector at $s_0$, defined over a number field $\F$. 
Then there is the  motivic fundamental group   \cite{DG}:
 $$
\pi_1^{\cal M}(X, v_0), ~~~~X:= {\Bbb P}^1 - \{s_0, ..., s_n\}.
$$
 It is a pro-Lie algebra object in the  category ${\cal M}_T(\F)$.   Recall the canonical fiber functor $\omega$, see (\ref{FFW}), and the equivalence (\ref{EQV}) it induces. Therefore
  there is a canonical 
map of Lie algebras
$$
{\rm L}_\bullet(\F) \lra {\rm Der}~\omega(\pi_1^{\cal M}(X, v_0)).
$$ 
  
 There are elements  $X_{s_{i}} \in H_1(X)$ corresponding to the punctures $s_i$. In the Betti realization ${\cal X}_s$ is the homology class of
 a little loop around the puncture $s$. They have the following incarnations. 
 
 The tangential base point $v_0$ provides a canonical map $\varphi_{v_0}: \Q(1) \to \pi_1^{\cal M}(X, v_0)$. Therefore 
  the element $X_0:= \omega\circ \varphi_{v_0}(\Q(1))$ is killed by the action of ${\rm L}_\bullet(\F)$.   Furthermore, there are elements 
  $X_0, X_{s_1}, ...,   X_{s_n} \in  \omega(\pi_1^{\cal M}(X, v_0))$ such that 
  
  \begin{enumerate}
  
  \item One has 
$
  X_0 + X_{s_1} +  \ldots +  X_{s_n}  =0.
$
  
  \item The conjugacy classes of  elements $X_{s_i}$, $i=1, ..., n$, are preserved  by the action of ${\rm L}_\bullet(\F)$. 
  
  \item   The element $X_0$ is killed by the action of ${\rm L}_\bullet(\F)$. 
  
  \end{enumerate}  
  
Denote by ${\rm Der}^S\omega ( \pi_1^{\cal M}(X, v_0))$ the Lie subalgebra of all {\it special} derivations of the graded 
  Lie algebra $\omega(\pi_1^{\cal M}(X, v_0)$, defined as the derivations which have the properties 1) - 3). Then there is a canonical 
 homomorphism of Lie algebras
\be \la{HLA1}
 {\rm L}_\bullet(\F) \lra {\rm Der}^S\omega(\pi_1^{\cal M}(X, v_0).
\ee   
The target Lie algebra   has a simple combinatorial description, which we recall now. 
    
 Given a  vector space $V$, denote by ${\cal C}_{V}$ the cyclic tensor envelope of 
$V$:
$$
{\cal C}_{V}:= \oplus_{m=0}^{\infty}\left(\otimes^mV\right)_{\Z/m\Z}
$$
where the subscript $\Z/m\Z$ denotes the coinvariants of the cyclic shift.  

The subspace of {\it shuffle relations} in 
${\cal C}_{V}$  generated by the 
elements
\be \la{SRI}
\sum_{\sigma \in \Sigma_{p,q}}  (v_0 \otimes v_{\sigma(1)} 
\otimes \ldots \otimes v_{\sigma(p+q)})_{\cal C}, 
\qquad p, q \geq 1, 
\ee
 Set
\be \nonumber
{\cal C}{{\cal L}ie}_X:=  
\frac{ {\cal C}_{H^1(X)}}{\mbox{Shuffle relations}}\otimes \Q(1);
\qquad 
{\rm CLie}_{X}:= \mbox{the graded dual of 
${\cal C}{{\cal L}ie}_X$}.
\ee
  The space ${\cal C}{{\cal L}ie}_X$ is has a Lie coalgebra structure:
   \be \la{wt2a}
  \delta:  (v_0, \ldots , v_m)\lms  \sum_{i<j} (v_i, \ldots , v_j) \wedge (v_{j}, \ldots , v_{i-1}).
\ee   
   So   ${\rm CLie}_X$ is   a Lie algebra. 

\bt \la{COR11} There is a canonical isomorphism of Lie algebras:
\be \la{KAP1}
\kappa:  {\rm CLie}_X \stackrel{\sim}{\lra} {\rm Der}^S\omega(\pi_1^{\cal M}(X, v_0)).
\ee
\et

Composing the map (\ref{HLA1}) with the isomorphism (\ref{KAP1}),  and dualising the composition, 
we arrive at a canonical map of graded Lie coalgebras, called the {\it motivic correlator map}:
$$
{\rm Cor}^{\cal M}: {\cal C}{\cal L}ie_X \lra  {\cal L}_\bullet(\F). 
$$
The 
cyclic tensor products of the basis elements  ${X}_{s_i}$ span ${\cal C}{\cal L}ie_X$ as a vector space. Their images under the motivic correlator map are called {\it motivic correlators}:
\be \la{MCR}
{\rm Cor}^{\cal M}_{v_0}(s_{i_1}, \ldots , s_{i_m}):= {\rm Cor}^{\cal M}\Bigl(X_{s_{i_1}} \otimes X_{s_{i_2}} \otimes \ldots \otimes X_{s_{i_m}} \Bigr)_{\cal C} \in {\cal L}_{m-1}(\F).
\ee
 By the very definition, motivic correlator (\ref{MCR}) are cyclic  invariant:
 $$
 {\rm Cor}^{\cal M}_{v_0}(s_{i_1}, \ldots , s_{i_m})  = {\rm Cor}^{\cal M}_{a, v}(s_{i_2}, \ldots , s_{i_m}, s_{i_1}).
$$
Below we use $s_0=\infty \in {\rm P}^1$, and the tangent  vector $v_0$ corresponding to the natural coordinate on ${\Bbb A}^1$. 
Motivic correlators of  weights $>1$ do not depend on $v_0$. So we skip $v_0$ in the notation. 
Motivic correlators of weight $>1$ with the base point at $\infty$ are invariant under rescaling: 
$$
{\rm Cor}^{\cal M}_{v_0}(\lambda \cdot s_{i_1}, \ldots , \lambda \cdot s_{i_m}) = {\rm Cor}^{\cal M}_{v_0}(s_{i_1}, \ldots , s_{i_m}), ~~~~m>2.
$$

The key properties  of motivic correlators are the following:
 
 \begin{itemize}
 
 \item There is a simple formula for the coproduct, reflecting (\ref{wt2a}):
 \be \la{wt2}
  \delta: {\rm Cor}_{v_0}^{\cal M}(b_0, \ldots , b_m)\lms  \sum_{i<j} {\rm Cor}_{v_0}^{\cal M}(b_i, \ldots , b_j) \wedge{\rm Cor}_{v_0}^{\cal M}(b_{j}, \ldots , b_{i-1}).
\ee 

\item  Given an embedding $\F \hra \C$, the canonical real period of the Hodge realization of the 
 motivic correlator is calculated as the Hodge correlator integral \cite{G08}. 
 
\end{itemize}


These two properties determine  motivic correlators  uniquely when $\F$ is a number field.

\begin{figure}[ht]
\centerline{\epsfbox{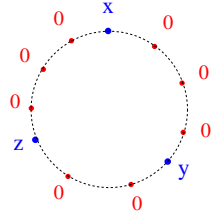}}
\caption{Motivic correlators.} 
\label{cor}
\end{figure}

\paragraph{Hodge correlators.} The real Hodge realization functor provides a map of Lie coalgebras
$$
r_{\rm Hod}: {{\cal C}}_{\bullet}(\mu_N) \lra {\cal L}^{\rm Hod}_\bullet.
$$
There is the canonical real period map \cite{G08}:
$$
{\rm P}: {\cal L}^{\rm Hod}_\bullet \lra \R.
$$
So composing them we get a canonical linear map 
\be \nonumber
{\rm P}\circ r_{\rm Hod}: {{\cal C}}_{\bullet}(\mu_N) \lra \R.
\ee

Our next goal is to calculate this map explicitly. Consider  the following current on ${\Bbb C}{\Bbb P}^1$: 
$$
G(x,y):= \log |x-y|, ~~~~x,y \in \C.
$$

Take a plane trivalent tree ${\rm T}$ whose legs are decorated cyclically by  $\zeta_0, \ldots , \zeta_m \in \mu_N$. 

Let us assign to each edge $E$ of the tree ${\rm T}$, internal or external, a current $\kappa_{\rm T}(\zeta_0, \ldots , \zeta_m)$  on 
\be \la{intv}
({\Bbb C}{\Bbb P}^1)^{\{\mbox{internal vertices of ${\rm T}$}\}}.
\ee
For each edge $E$ of the tree, we define a distribution $G_E$ on (\ref{intv}) given by the pull back of the Green function $G(x,y)$ under the map given by the projection 
 $(\ref{intv})\lra ({\Bbb C}{\Bbb P}^1)^{\{\mbox{vertices of ${E}$}\}}$. We define the current $\kappa_{\rm T}(\zeta_0, \ldots , \zeta_m)$ of the top degree $2m-2$ by 
 the sum over all trivalent trees ${\rm T}$:
 \be \nonumber
  \kappa_{\rm T}(\zeta_0, \ldots , \zeta_m):=   \sum_{\mbox{T}}
{\rm Alt}_{2m-1} \left ({\rm sgn}_{\rm T} (E_1 \wedge ... \wedge E_{2m-1})~ G_{E_1} d^\C G_{E_2} \wedge \ldots \wedge d^\C G_{E_{2m-1}}\right ).
\ee 
One proves that it is integrable, i.e. it is indeed a current. So  integrating it over (\ref{intv}), we get a real number, called the Hodge correlator \cite{G00b}, \cite{G08}:
\be \nonumber
{\rm Cor}_{\rm Hod}(\zeta_0, \ldots , \zeta_m) \in \R.
\ee

The next theorem tells that this number is the canonical real period map, calculated on the  Hodge realization of of the motivic correlator ${\rm Cor}_{\rm Mot}(\zeta_0, \ldots , \zeta_m) $. 
\bt \la{g08} \cite{G08} One has 
$$
{\rm P}\circ r_{\rm Hod} \circ {\rm Cor}_{\rm Mot}(\zeta_0, \ldots , \zeta_m)  = {\rm Cor}_{\rm Hod}(\zeta_0, \ldots , \zeta_m).
$$
\et

\subsection {Dihedral Lie coalgebras} \la{Sec2b} 

Section  \ref{Sec2b} presents the formalism of the double shuffle relations considered modulo the depth filtration, 
together with the construction a coproduc, making it into a bigraded Lie coalgebra. 
 Section \ref{SECCTT2.6}  relates it to the modular complex dga. 
Sections \ref{Sec2b} - \ref{SECCTT2.6}  clarify the   structure of the proof given in 
 Section \ref{4.26.01.3aX}. The reader may skip Sections \ref{Sec2b} - \ref{SECCTT2.6}  and proceed to Section \ref{4.26.01.3aX}.
\vskip 2mm  

Let $\G$   be a commutative group written  multiplicatively. 
The dihedral Lie coalgebra is a bigraded    Lie coalgebra over $\Q$ introduced in \cite{G98} and studied in \cite[Section 4]{G00a}:
$$
{\cal D}_{\bullet, \bullet}(\G ) = \oplus_{m \geq 1} {\cal D}_{\bullet, m}(\G ).
$$ 
Let us recall its definition. 
The space  ${\cal D}_{w, m}(\G)$  
  is spanned by the elements  
 \be  \la{DIHL2}
\{g_1,...,g_m \}_{n_1,...,n_m}, ~~~~n_i\geq 0; ~~\sum (n_i+1)=w.
 \ee
 We package them into power series, setting 
$g_0:= (g_1 \ldots g_m)^{-1}$ and writing
\be \la{DIHL1}
\{g_0, g_1,...,g_m~|~0:t_1:....:t_m\} =: \sum_{n_i \geq 0} \{g_1,...,g_m  \}_{n_1,...,n_m} t_1^{n_1}...t_m^{n_m}.
\ee
Relations between  generators  (\ref{DIHL2}) are encoded by   relations between the  generating series (\ref{DIHL1}). 

The generating series  (\ref{DIHL1}) are encoded  by  {\it    nonhomogeneous dihedral} words  in $\G$, where $g_i \in \G$: 
\be \la{DIHL}
\begin{split}
&\{g_0, g_1, \ldots,  g_m~ |~ t_0:\ldots :t_m\} \qquad \mbox{such that} \quad g_0 \cdot \ldots \cdot  g_m = 1,\\
&\{g_0, g_1, \ldots,  g_m ~|~ t_0: \ldots :t_m\} = \{g_0, g_1, \ldots ,  g_m | t + t_0:...: t+t_m\}.\\
\end{split}
\ee
We can  parametrize generators  (\ref{DIHL}) in a dual way, via  the   {\it    homogeneous dihedral} words in $\G$:
\be \la{F37}
\begin{split}
&\{g_0: g_1: \ldots :  g_m ~|~ t_0, \ldots, t_m\} \qquad \mbox{such that} \quad t_0 + \ldots + t_m = 0,\\
&\{g\cdot g_0:   \ldots : g\cdot g_m~|~t_0,\ldots,t_m\} = \{g_0:  \ldots : g_m~|~t_0,\ldots,t_m\}\quad \forall g \in G,\\
\end{split}
\ee
 The duality between the homogeneous and nonhomogeneous dihedral words   is given by  
\be \la{F38}
\begin{split} 
&\{g_0: g_1 : \ldots :  g_m~|~t_0,\ldots,t_m\} \longmapsto \{g_0^{-1} g_{ 1}, g_1^{-1} g_{ 2}, \ldots ,  g_m^{-1} g_{0}~|~t_0:t_0+t_1:\ldots:t_0+\ldots+t_m \},  \\
&\{g_0, g_1, \ldots ,  g_m~|~t_0:\ldots:t_m\} \longmapsto \{g_0: g_0  g_1 : \ldots :  g_0   \ldots  g_m~|~t_1-t_0,t_2-t_1,\ldots,t_0-t_m\}.\\
\end{split}
\ee

   \begin{definition}
  The space  ${\cal D}_{\bullet, m}(\G)$ is   generated by   elements (\ref{DIHL2}) subject to the following relations:  
  
 \begin{itemize}
 
\item The double shuffle relations: 
\be \la{F39}
\begin{split}
& \sum_{\sigma \in \Sigma_{k,m-k}}\{ g_{\sigma(1)} : \ldots :  g_{\sigma(n)}: g_{m+1}~|~ t_{\sigma(1)}, \ldots,  t_{\sigma(n)} , t_{m+1}\},\\
& \sum_{\sigma \in \Sigma_{k,m-k}}
\{x_0, x_{\sigma(1)}, \ldots ,  x_{\sigma(m)}~| ~t_0: t_{\sigma(1)}: \ldots:  t_{\sigma(m)} \}.\\
\end{split}
\ee
\item  The distribution relations, where    $l \in \Z$, and   $|l|$ divides  $|\G|$, if $\G$ is a finite group: 
\be
\{x_0^l, x_{ 1}^l, \ldots ,  x_{ m}^l ~|~ t_0: t_{ 1 }: \ldots:  t_{m} \} -
\sum_{y_i^l = x^l_i}\{y_0, y_{ 1}, \ldots ,  y_{m} ~|~ l\cdot t_0: \ldots:  l\cdot t_{m} \}.
\ee
\end{itemize}
\end{definition}
 
  The distribution relation for $l=-1$ is
$$
\{x_0^{-1}, x_{ 1}^{-1}, \ldots ,  x_{ m}^{-1}~|~ t_0: t_{1}: \ldots:  t_{m} \} =
 \{x_0, x_{ 1}, \ldots ,  x_{m} ~|~ - t_0: \ldots:  - t_{m} \}.
$$
We proved in \cite[Theorem 4.1]{G00a} that  the double shuffle relations imply the dihedral symmetry:
\be
\begin{split}
&\{g_0, \ldots , g_{m-1}, g_m~|~t_0:t_1: \ldots:t_m\} = \{g_1,  \ldots , g_m, g_0~|~t_1: \ldots :t_m: t_0\}, \\
& \{g_0,  \ldots , g_m~|~t_0: \ldots :t_m\} = (-1)^{n+1}\{g_n,  \ldots , g_0~|~t_m: \ldots :t_0\}.\\
\end{split}
\ee
 So we   picture   generators (\ref{DIHL}) as    pairs   $(g_0,t_0), \ldots , (g_m,t_m)$   located cyclically on an oriented circle:

\begin{center}
\hspace{0.0cm}
\epsffile{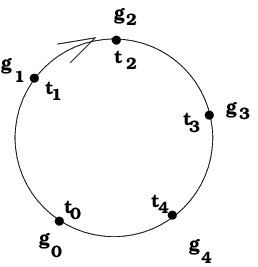}
\end{center}

\paragraph{The coproduct.}
We define a cobracket: 
\be \la{THECOP}
\begin{split}
&\delta:   \widehat {\cal D}_{\bullet, m}(\G)  \longrightarrow   \sum_{k+l=m}   \widehat {\cal D}_{\bullet, k}(\G)\wedge \widehat {\cal D}_{\bullet, l}(\G),\\
&\delta \{g_0,  \ldots ,  g_m~| ~t_0: \ldots :t_m\}:=\\
&\sum_{j \not = \{i-1, i\};~ i,j\in \Z/(m+1)\Z} \{g_{i}, \ldots ,g_{j-1}, y_{ij}, ~| ~t_{i}: \ldots : t_{j-1}:t_j    \}\\
& \wedge \{x_{ij}, g_{j+1},  \ldots ,  g_{i-1} ~|~ t_j: t_{j+1} \ldots : t_{i-1}  \}.\\
\end{split}
\ee
Here   the elements $x_{ij}$ and $y_{ij}$ are defined by 
\be
 g_{i}  ...  g_{j-1}y_{ij}= 1, ~~x_{ij}g_{j+1}  ...   g_{i-1} = 1.
\ee
 Each term of the formula corresponds to the following procedure. Choose 
an arc on the circle   between  two neighboring    points, and 
a   point different from the arc ends. 
Cut the circle 
in the   arc and in the   point,  make two   oriented circles out of it, and then make  the dihedral word on each of the circles out of the initial word.

\begin{figure}[ht]
\centerline{\epsfbox{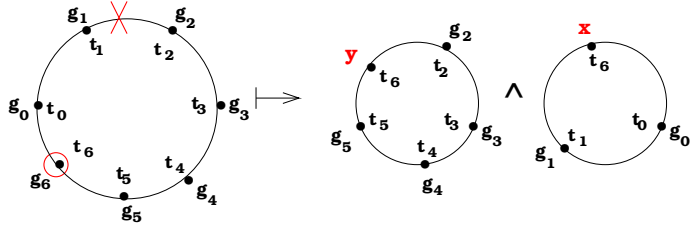}}
\caption{A term in the coproduct formula (\ref{THECOP}) corresponding to $i=2, j=6$.} 
\label{arb2}
\end{figure}


  \begin{theorem} \cite[Theorem 4.3]{G00a} \la{DDLCA}
  The map  $\delta$ provides a structure of   bigraded Lie coalgebras  on  $\widehat D_{\bullet, \bullet}(\G)$ and ${\cal D}_{\bullet, \bullet}(\G)$.
\end{theorem} 


\subsection{The  comparison theorem} \la{SECCTT2.6}



The following crucial result is \cite[Theorem 1.2]{G98}. We include the proof for completeness.

\begin{theorem} \label{4.26.01.3} Let us choose a primitive root of unity $\zeta_N$. 

a) There exists a canonical surjective map of dg-algebras
$$
\eta: {\rm M}^*_N\lra \Lambda^*\Bigl({\cal D}_{\bullet, \bullet}(\mu_N)\Bigr).
$$
 b) It gives rise to a collection of surjective homomorphisms o complexes 
\begin{equation} \label{4.26.01.2}
\eta^*_{w,m}: {\rm M}^{*}_{(m)}\otimes_{\Gamma_1(m;N)}{\rm Sym}^{w-m}(V_m) \lra 
\Lambda_{w,m}^*\Bigl({\cal D}_{\bullet, \bullet}(\mu_N)\Bigr).
\end{equation}
 The map (\ref{4.26.01.2}) is an isomorphism when $N=1$, or when $w=m$ and $N=p$ is a prime.
\end{theorem}

 In particular, there exists a canonical surjective map of $\Q$-vector spaces
$$
\eta^1_{w,m}: {\rm M}^1_{(m)}\otimes_{\Gamma_1(m;N)}{\rm Sym}^{w-m}(V_m) \lra {\cal D}_{w,m}(\mu_N).
$$

\begin{proof} 
 a) Recall that 
 \begin{equation} \label{4.26.01.1}
\Gamma_1(m;N)  \backslash {\rm GL}_m(\Z) = ~ \{ (\alpha_1, ..., \alpha_m)|
~ \alpha_i \in \Z/N\Z, 
\quad {\rm g.c.d.} (\alpha_1, ..., \alpha_m, N) =1\}.
 \end{equation}

Let $S'_{w,m}(N)$ be  the vector space of all functions 
$$
f(\alpha_0, ..., \alpha_{m}  
~|~ s_0: ... : s_{m}  ) 
$$
where $f$ is a polynomial function of the 
 degree $w-m$  in $s_0, ..., s_{m}$, and:
\be \la{KJIK}
\begin{split}
&\alpha_i \in \Z/N\Z, \quad \alpha_0 + ... +\alpha_{m} =0, \quad 
{\rm g.c.d.} (\alpha_1, ..., \alpha_m, N) =1,\\
&f(\alpha_0, ..., \alpha_{m}  
~|~ s_0: ... : s_{m}  )  =   f(\alpha_0, ..., \alpha_{m}  
~|~ s_0+s: ... : s_{m}+s ),~~~~  \forall s.\\
\end{split}
\ee
Using (\ref{4.26.01.1}) we get  canonical isomorphism
$$
{\cal F}\left(\Gamma_1(m;N)  \backslash {\rm GL}_m(\Z)\right) \otimes_{\Q}
  {\rm Sym}^{w-m}(V_m)   =   S'_{w,m}(N).
$$

The space $S'_{w,m}(N)$ can be identified with the  vector space  $S^{''}_{w,m}(N)$ of 
all functions
$$
g(\beta_0: ...: \beta_{m}  
~|~ t_0, ... ,t_{m}  ) 
$$
where $g$ is a degree $w-m$  polynomial function  
  in $t_0, ..., t_{m}$, where $t_0+ ... +t_{m}=0$, 
and
\be \la{F38a}
\begin{split}
&\beta_i \in \Z/N\Z, \quad 
{\rm g.c.d.} (\beta_0 - \beta_{m}, ..., \beta_{m-1} - \beta_{m}, N) =1,\\
&g(\beta_0: ...: \beta_{m}  
~|~ t_0, ... ,t_{m}  )   =  g(\beta_0 +\beta : ...: \beta_{m} +\beta 
~|~ t_0, ... ,t_{m}  )~~ \forall \beta \in \Z/N\Z.\\
\end{split}
\ee
There is a canonical isomorphism ${\rm I}: S'_{w,m}(N) \lra S''_{w,m}(N)$, given by
\be \la{F39a}
\begin{split}
&f(\alpha_0, ..., \alpha_{m}  
~|~ s_0: ... : s_{m}  ) \stackrel{{\rm I}}{\longrightarrow} {\rm I}(f)(\beta_0: ...: \beta_{m}  
~|~ t_0, ... ,t_{m}),\\
& \mbox{where} ~\alpha_i = \beta_{i+1} - \beta_{i}; ~~ t_i = s_{i+1} - s_{i}.\\
\end{split}
\ee

Let $S_{w,m}(N)$ be the quotient of the space $S'_{w,m}(N)$ by the 
double shuffle relations  
\be \la{F40a}
\begin{split}
&\sum_{\sigma \in \Sigma_{k,l}}f(\alpha_{0}, \alpha_{\sigma(1)}, \ldots , \alpha_{\sigma(m)}  
~ | ~s_0: s_{\sigma(1)}: ... : s_{\sigma(m)}  ) = 0, \\
&\sum_{\sigma \in \Sigma_{k,l}}{\rm I}(f)(\beta_0: \beta_{\sigma(1)}: \ldots : \beta_{\sigma(m)}
~  | ~t_0, t_{\sigma(1)}, ... , t_{\sigma(m)} ) = 0.\\
\end{split}
\ee
One should compare    (\ref{KJIK})  $\leftrightarrow$ (\ref{DIHL}),   (\ref{F38a}) $\leftrightarrow$ (\ref{F37}),   (\ref{F39a}) $\leftrightarrow$ (\ref{F38}),   (\ref{F40a}) $\leftrightarrow$ (\ref{F39}).
This leads to:

\begin{lemma} \label{4.26.01.5} There exists a canonical isomorphism of vector spaces
 \begin{equation} \label{4.26.01.4}
{\rm M}^1_{(m)}\otimes_{{\rm GL}_m(\Z)}\Bigl({\cal F}(\Gamma_1(m;N)  \backslash {\rm GL}_m(\Z))  \otimes_{\Q}
  {\rm Sym}^{w-m}(V_m)\Bigr)  =  S_{w,m}(N).
 \end{equation}
\end{lemma}

\begin{proof}
Let $(u_1, ..., u_m)$ be a basis of the lattice ${\rm L_m}$. Set 
$v_i:= u_{i+1} - u_i$. Then $(v_1, ..., v_m)$ is another basis of the lattice ${\rm L_m}$. 
A generator of the  space (\ref{4.26.01.4}) can be written in either of the following two   forms 
\be \nonumber
\begin{split}
&\langle v_0, ..., v_{m} \rangle  \otimes f(\alpha_0, ..., \alpha_{m}
 ~ | ~s_0: ... : s_{m});\\
&\langle u_0: ...: u_{m} \rangle  \otimes g(\beta_0: ... : \beta_{m}
 ~ |~ t_0, ... , t_{m}).\\
\end{split}
\ee
\end{proof} 

 Let $\alpha_0 + \ldots + \alpha_m =0$ where $\alpha_i \in \Z/N\Z$. 
Pick a basis $(v_1, \ldots , v_m)$ of the lattice ${\rm L_m}$. 
Set 
\begin{equation} \nonumber
\begin{split}
\eta_{m, w}^1: &\sum_{n_i\geq 0} (\alpha_1, \ldots , \alpha_m) \otimes t_1^{n_1} \ldots t_m^{n_m} \otimes [v_1, \ldots , v_m] \lms\\
&\{\zeta_N^{\alpha_0}, \zeta_N^{\alpha_1}, \ldots,  \zeta_N^{\alpha_m} ~|~0:t_1: \ldots : t_m\}= \sum_{n_i\geq 0}
\{\zeta_N^{\alpha_1}, \ldots,  \zeta_N^{\alpha_m}\}_{n_1, \ldots , n_m}t^{n_1}_1 \ldots t^{n_m}_m.\\ 
\end{split}\end{equation}
The double shuffle relations  (\ref{sshh1}) -  (\ref{sshh3}) in the modular complex  match the ones in the dihedral Lie algebra. The relation (\ref{sshh1})  match the second relation in (\ref{F39}), and the relation (\ref{sshh3}) match the first.

\vskip 2mm

Let us consider $\theta_i \in \Z/N\Z$ such that $\theta_i + \alpha_{n_{i-1}+1} + \ldots + \alpha_{n_{i}} = 0$ in $\Z/N\Z$. Set 
\begin{equation} \nonumber
\begin{split}
&\eta_{m, w}^l: \sum_{n_i\geq 0} (\alpha_1, \ldots , \alpha_m) \otimes t_1^{n_1} \ldots t_m^{n_m} \otimes [v_1, \ldots , v_{n_1}]  \wedge \ldots \wedge [v_{n_{l-1}+1}, \ldots , v_{n_l}] \lms\\
&\{\zeta_N^{\theta_1}, \zeta_N^{\alpha_1}, \ldots,  \zeta_N^{\alpha_{n_1}} ~|~0:t_1: \ldots : t_{n_1}\} \wedge \ldots
 \wedge  \{\zeta_N^{\theta_l}, \zeta_N^{\alpha_{n_{l-1}+1}}, \ldots , \zeta_N^{\alpha_{n_l}} ~|~0:t_{n_{l-1}+1}: \ldots : t_{n_l}\}. \\
\end{split}\end{equation}

Formula (\ref{FD}) for the differential in the modular complex match  formula (\ref{THECOP}) for the coproduct in the dihedral Lie coalgebra. 

\vskip 2mm

  b)   
  If $N=p$ is a prime and $w=m>1$, there is just one  
distribution relation $\{1, \ldots , 1\}_{0, ..., 0} = \sum_{x_i^p=1}\{x_1, \ldots , x_{m}\}_{0, ..., 0}$, which follows from the shuffle relation 
$\sum_{\sigma \in S_m}\{x_{\sigma(1)}, \ldots , x_{\sigma(m)}\}_{0, ..., 0}=0$.

If $N=1$ we have the distribution relations only for $l=-1$, and follows from the double shuffle relations  \cite[Theorem 4.1]{G98}.
\end{proof}

 \subsection{Proof of Theorem \ref{4.26.01.3a}}  \la{4.26.01.3aX}
 
\paragraph{1. An alternative   presentation of motivic correlators.} Given a collection of elements $b_0, ..., b_m\in \C^*$ defined up to a   rescaling  
$(b_0 ,  ..., b_m) \sim (tb_0, ..., tb_m)$, consider a new collection of elements $a_0, ..., a_m\in \C^*$ given by
\be \nonumber
 \begin{split}
 &a_i:= b_{i+1}/b_i, ~~~~i \in \Z/(m+1)\Z;\\
 &a_0 \cdot \ldots \cdot a_m=1.\\
 \end{split}
 \ee
 Then let us  organize the the data of a motivic correlator in the following: 
  \begin{equation} \label{nNna}
\begin{split}
&{\rm Cor}_{\rm Mot}^*(a_0|n_0, a_{1}|n_1, ..., a_{m}|n_m) : =\\
& {\rm Cor}_{\rm Mot}(b_0, \underbrace {0, ... , 0}_{n_0 ~\mbox{times}}, b_1, \underbrace {0, ... , 0}_{n_1 ~\mbox{times}}, \ldots , b_m, 
 \underbrace {0, ... , 0}_{n_m ~\mbox{times}}).\\
\end{split}
\ee 
Here the notation $a_i|n_i$ means a string  $(\ldots  a_i, 0, ..., 0, \ldots)$ given by an  element $a_i\in \C^*$ followed by  $n_i$ $0$'s. 
 Note that these $0$'s sit in the original correlator between $b_i$ and $b_{i+1}$.   We organize the elements (\ref{nNna}) into a generating series: 
 
 \be \la{GSMC}
 \begin{split}
 &{\rm Cor}_{\rm Mot}^*(a_0, a_{1}, ..., a_{m} ~|~ t_0: ... :t_m) := \\
 &\sum_{n_i\geq 0}{\rm Cor}_{\rm Mot}^*(a_0|n_0, a_{1}|n_1, ..., a_{m}|n_m)  \otimes t_0^{n_1} \ldots t_m^{n_m}.\\
\end{split}\end{equation} 

We refer to the elements (\ref{GSMC}) as motivic $\ast-$correlators. 
 
 \bt \la{CORHO} The generating series (\ref{GSMC}) are homogeneous with respect to the variables $\{t_i\}$:  
 \be
 {\rm Cor}_{\rm Mot}^*(a_0, a_{1}, ..., a_{m} ~|~ t_0+t: ... :t_m+t) ={\rm Cor}_{\rm Mot}^*(a_0, a_{1}, ..., a_{m} ~|~ t_0: ... :t_m).
\ee
\et

We deduce Theorem \ref{CORHO} from Theorem \ref{MSR} below. 
   
\paragraph{2. The $w=m$ case.} Let us  define    a map of complexes
 \begin{displaymath} 
    \xymatrix{
      \Bigl({\rm M}^{1}_{(m)}  \ar[r]^{~~\partial}  \ar[d]^{\gamma_{(m)}^{1}} &     \ldots \ar[r]^{\partial~~~~~~~} & 
  {\rm M}_{(m)}^{m-1}   \ar[r]^{\partial}  \ar[d]^{ \gamma_{(m)}^{m-1}} & {\rm M}_{(m)}^{m}\Bigr)_{\Gamma_1(m;N)}\ar[d]^{\gamma_{(m)}^{m}}  \\     
{\cal C}_m(\mu_N)\ar[r]^{~~~~\delta}&\ldots \ar[r]^{\delta~~~~~~~~~~~} &{\cal C}_{m-2}\otimes \Lambda^{m-2} \widehat {\cal C}_1(\mu_N)\ar[r]^{~~~~~~~\delta} & 
\Lambda^m \widehat {\cal C}_1(\mu_N) \\}
\end{displaymath} 

 Choose a primitive $N-$th root of unity $\zeta_N$. Pick  a basis $(v_1, \ldots , v_m)$ of the lattice ${\rm L_m}$. 
 
 Let us define the map $\gamma_{(m)}^{1}$  on the generator $[\alpha_1, ..., \alpha_m]$, where $\alpha_i \in \Z/N\Z$.

First of all set $\alpha_0:= -(\alpha_1 + ... + \alpha_m)$. 
Then we   set
\be \la{MG1m}
\begin{split}
&\gamma_{(m)}^{1}: {\rm M}_{(m)}^{1}\otimes_{\Gamma_1(m;N)}\Q  \lra {{\cal C}}_{m}(\mu_N),\\
& [\alpha_1, ..., \alpha_m]  \lra {\rm Cor}^*_{\rm Mot}(\zeta_N^{\alpha_0}, \ldots , \zeta_N^{\alpha_m}),\\
\end{split}
\ee

Note that the generators $[\alpha_1, ..., \alpha_m]$ are defined only when $(\alpha_1, ..., \alpha_m) \in {\rm M}_m(N)$. 
However the map (\ref{MG1m})  make  sense for any $(\alpha_1, ..., \alpha_m) \in (\Z/N\Z)^{m}$, unless $m=1$, $\alpha_0 = \alpha_1 = 1$. 
In this case we set ${\rm Cor}_{\rm Mot}(1,1):= \gamma^{\cal M}$. And this is why we need $\widehat {\cal C}_1(\mu_N)$. 

Therefore we can  extend 
the map $\gamma_{(m)}^{1}$ to all $(\alpha_1, ..., \alpha_m) \in (\Z/N\Z)^m$, and  
define   the other maps $\gamma_{(m)}^{k}$  by the condition that $\gamma$ is a map of graded commutative algebras.

 \paragraph{3. The general case.}  
 Let $\alpha_0 + \ldots + \alpha_m =0$ where $\alpha_i \in \Z/N\Z$.  
Set \begin{equation} \label{XXXb}
\begin{split}
\gamma_{m, w}^1: &\sum_{n_i\geq 0} (\alpha_1, \ldots , \alpha_m) \otimes t_1^{n_1} \ldots t_m^{n_m} \otimes [v_1, \ldots , v_m] \lms\\
&{\rm Cor}^*_{\rm Mot}(\zeta_N^{\alpha_0}, \zeta_N^{\alpha_1}, \ldots,  \zeta_N^{\alpha_m} ~|~0:t_1: \ldots : t_m).\\
\end{split}\end{equation}

In general, let us consider $\theta_i \in \Z/N\Z$ such that $\theta_i + \alpha_{n_{i-1}+1} + \ldots + \alpha_{n_{i}} = 0$. Set 
\begin{equation}\nonumber
\begin{split}
&\gamma_{m, w}^l: \sum_{n_i\geq 0} (\alpha_1, \ldots , \alpha_m) \otimes t_1^{n_1} \ldots t_m^{n_m} \otimes [v_1, \ldots , v_{n_1}]  \wedge \ldots \wedge [v_{n_{l-1}+1}, \ldots , v_{n_l}] \lms\\
&{\rm Cor}^*_{\rm Mot}(\zeta_N^{\theta_1}, \zeta_N^{\alpha_1}, \ldots,  \zeta_N^{\alpha_{n_1}} ~|~0:t_1: \ldots : t_{n_1}) \wedge \ldots \\
& \wedge  {\rm Cor}^*_{\rm Mot}(\zeta_N^{\theta_l}, \zeta_N^{\alpha_{n_{l-1}+1}}, \ldots , \zeta_N^{\alpha_{n_l}} ~|~0:t_{n_{l-1}+1}: \ldots : t_{n_l}). \\
\end{split}\end{equation}

\paragraph{4. Malkin's "difficult" shuffle relations for motivic correlators \cite{Ma19b}.} 
  \bt \la{MSR} 
  For any $p, q \geq 0$ we have:
\be \la{MSRq} 
\begin{split}
&0= \sum_{\sigma \in \Sigma_{p,q}} {\rm Cor}_{\rm Mot}^*(a_0|n_0, a_{\sigma(1)}|n_{\sigma(1)},, ..., a_{\sigma(p+q)}| n_{\sigma(p+q)})  + \mbox{lower depth terms}.\\
\end{split}
\ee
\et

 Here are two examples, where we elaborate all the terms.  We use a shorthand ${\rm Cor}^*(...)$.  
 
 For $p=q=1$, we have:
 \be \nonumber
 \begin{split}
& 0= {\rm Cor}^*(a_0|n_0, a_{1}|n_1,   a_{2}|n_2)  +   {\rm Cor}^*(a_0|n_0, a_{2}|n_2,   a_{1}|n_1)   \\
& - {\rm Cor}^*(a_0|n_0, a_{1}  \cdot a_{2}| n_{1}+n_{2}+1)    \\
& - {\rm Cor}^*(a_{1}|n_1,   a_0\cdot a_{2}| n_{0}+ n_{2}+1)  -  {\rm Cor}^*( a_{2}|n_2,   a_0\cdot a_{1}| n_0+n_{1}+1).\\
\end{split}
\ee

For $p=1, q=2$ we have: 
  \be \nonumber
  \begin{split}
&  0={\rm Cor}^*(a_0|n_0, a_{1}|n_1, a_{2}|n_2, a_3|n_3)  +    {\rm Cor}^*(a_0|n_0, a_{2}|n_2, a_{1}|n_1, a_3|n_3) \\
& +    {\rm Cor}^*(a_0|n_0, a_{2}|n_2, a_3|n_3, a_{1}|n_1)  \\
& - {\rm Cor}^* (a_0|n_0, a_{1} \cdot a_{2}|n_1+n_2+1, a_3|n_3)   - {\rm Cor}^*(a_0|n_0,   a_{2}|n_2, a_{1}\cdot a_3|n_1+n_3+1)  \\
& - {\rm Cor}^*(a_{1}|n_1,   a_{2}\cdot a_3\cdot a_0|n_{2}+ n_{3}+n_0+2)  -  {\rm Cor}^*( a_{2}|n_2,  a_{3}|n_3,  a_0\cdot a_{1}| n_0+n_{1}+1).\\
\end{split}
\ee

  The formula for the product of two motivic $\ast-$correlators is a sum of   terms of three   kinds: 
\begin{enumerate}

\item  The usual shuffle product, where we shuffle $\{a_1|n_1, ..., a_p|n_p\}$ and $\{a_{p+1}|n_{p+1}, ..., a_{p+q}|n_{p+q}\}$.

\item  The  terms  where, in the process of   moving $\{a_1|n_1, ..., a_p|n_p\}$ through $\{a_{p+1} | n_{p+1}, ..., a_{p+q}|n_{p+q}\}$, some pairs $a_i|n_i$ and $a_j|n_j$ where $i \in \{1, ..., p\}$ and 
 $j \in \{p+1, ..., p+q\}$ collided. Each such a collision results in  $a_i\cdot a_j | n_i+n_j+1$,  
and multiplies the correletor  by $-1$.
  
\item  The two extra terms.  One where we put $a_0\cdot a_1\cdot  ... \cdot a_p|n_0+n_1+ ... +n_p+p$   and keep the remaining terms intact. The other  
 is where we put $a_0\cdot a_{p+1}\cdot  ... \cdot a_{p+q}|n_0+n_{p+1}+ ... +n_{p+q}+q$,  and keep the remaining terms intact. These terms always come with the $-$ sign. 
 \end{enumerate}

The   type 1) + 2) terms form the generalized shuffle product, compare with \cite{GZ}. The   two   type 3) terms  are a new phenomenon. 
 
\paragraph{Proof of Theorem \ref{MSR}.} The claim is just equivalent to  the shuffle relation (\ref{MSRq}) of type $(0,q)$. 

\paragraph{End of the proof of Theorem \ref{4.26.01.3a}.} 
Recall the  double shuffle relations (\ref{sshh1}) and (\ref{sshh3}) in the modular complex. One immediately sees that the relation  (\ref{sshh3})  match  the "easy" shuffle relations motivic correlators ${\rm Cor}_{\rm Mot}$, which followed from (\ref{SRI}). The shuffle relation  (\ref{sshh1})  match 
the "difficult" ones for the motivic correlators ${\rm Cor}^*_{\rm Mot}$, considered modulo the depth filtration, and provided by Theorem \ref{MSR}.  This implies that the map $\gamma^1_{\bullet, \bullet}$ in (\ref{XXXb}) 
is well defined as a linear map. The maps $\gamma^*_{\bullet, \bullet}$ are defined so that they give rise to a map of commutative algebras. In particular, they are well defined as linear maps. 

The claim that the map $\gamma^*_{\bullet, \bullet}$ is a map of complexes follows from the  fact that the map $\gamma^1_{\bullet, \bullet}$  intertwines 
the differential in the modular complex with the coproduct   for cyclotomic motivic correlators, considered modulo the depth filtration. 
This is exactly what Lemma \ref{LCFMC}  says. 

\bl \la{LCFMC} The coproduct formula (\ref{wt2}) for the motivic correlators, considered at roots of unity, and modulo the depth filtration,  match the differential (\ref{FD}) in the modular complex, rewritten  for the homogeneous generators
$\langle u_0: u_1: \ldots : u_m\rangle$.  
\el

\begin{proof} Each term of the coproduct of a motivic correlator is determined by choosing a point $x$ and an arc $\alpha$ on the circle. If  $x=0$, 
the total depth of the coproduct is 1 less than the original depth. So this term disappears after taking the associate graded for the depth filtration. 

Next, if the arc containing the point $x$ is right next to the arc $\alpha$, one of the factors in the coproduct will be ${\rm Log}^{\cal M}(x)$, and thus it vanishes if $x^N=1$. 

The remaining terms  match  formula (\ref{THECOP}) for   the coproduct  in the dihedral Lie algebra for $\G = \mu_N$, and hence the differential in the modular complex with the coefficients in ${\rm Sym}^{w-m}(V_m)$. 
\end{proof}

Theorem \ref{4.26.01.3a} is proved.

\paragraph{Concluding comments.} If $m>2$, the  $(2m-2)-$chains $\psi[e_1, ..., e_m]$  do not satisfy the second shuffle relations.  
This   begs for a resolution of the 
dg-algebra of higher modular symbols. 
The boundary of the $(2m-2)-$chain representing second shuffle relation    is not   zero.  
It is  a linear combination of similar shuffle relations of the second kind. Therefore we  need to add extra terms to    the elements $\psi[e_1, ..., e_m]$, and  find 
$(2m-1)-$chains which bound the second shuffles of the modified elements. 
I leave realization of this program for $m>5$ for the future.



  \subsection{Modular complexes and motivic multiple polylogarithms at roots of unity} \la{Sec2}
  
  In this section we present another proof of Theorem \ref{4.26.01.3a}, based on the properties of motivic multiple polylogarithms at roots of unity. 

\paragraph{1. The multiple polylogarithms.} They  
     are defined by the power series \cite{G94}:
\begin{equation} \label{zhe5}
{\rm Li}_{n_{1},...,n_{m}}(x_{1},...,x_{m})   
  := 
\sum_{0 < k_{1} <  ... < k_{m} } \frac{x_{1}^{k_{1}}x_{2}^{k_{2}}
... x_{m}^{k_{m}}}{k_{1}^{n_{1}}k_{2}^{n_{2}}...k_{m}^{n_{m}}}.
\end{equation}
Here $w := n_1+...+n_m$ is the {\it  weight}  and $m$
is the {\it  depth}. 

Their special values when $x_i =1$ are  Euler's \cite{E}
multiple zeta values:
$$
\zeta(n_{1}, n_2, ...,n_{m}) : = \sum_{0 < k_{1} <  ... < k_{m} }
\frac{1}{k_{1}^{n_{1}}k_{2}^{n_{2}}...k_{m}^{n_{m}}} \qquad n_m >1.
$$

On the other hand,  consider the following iterated integrals:
\begin{equation} \label{3??}
\begin{split}
&{\rm I}_{n_{1},...,n_{m}}(a_{1}:...:a_{m}:a_{m+1}) : =\\
&\int_{0}^{a_{m+1}} \underbrace {\frac{dt}{t-a_{1}} \circ \frac{dt}{t} \circ ... \circ 
\frac{dt}{t}}_{n_{1} \quad  \mbox {times}} \circ 
\quad ... \quad \circ 
\underbrace {\frac{dt}{t-a_{m}} \circ \frac{dt}{t} \circ ... \frac{dt}{t}}_
{n_{m} \quad  \mbox {times}}.\\
\end{split}
\ee
We also use the notation
$$
{\rm I}_{n_{1},...,n_{m}}(a_{1}, ..., a_{m}) = {\rm I}_{n_{1},...,n_{m}}(a_{1}:...:a_{m}:1).
$$
If $|x_{i}| < 1$,     power series (\ref{zhe5}) and   iterated integrals (\ref{3??}) are related  by  \cite[Theorem 2.2]{G01}:
\begin{equation} \nonumber
\begin{split}
& {\rm Li}_{n_{1}, \ldots ,n_{m}}(x_{1},...,x_{m} )  =  (-1)^m \cdot 
{\rm I}_{n_{1},\ldots ,n_{m}}
( (x_1\ldots x_m)^{-1},(x_2 \ldots x_m)^{-1},\ldots ,x_m^{-1}).  \\
& {\rm I}_{n_{1},\ldots ,n_{m}}(a_{1}: \ldots :a_{m}:a_{m+1})   =   (-1)^m \cdot {\rm Li}_{n_{1},\ldots ,n_{m}}
\left(\frac{a_2}{a_1},\frac{a_3}{a_2}, \ldots ,\frac{a_m}{a_{m-1}},\frac{a_{m+1}}{a_m}\right).  \\
\end{split}
\end{equation}

\paragraph{2. The double shuffle relations.} Consider the formal generating series
$$
{\rm Li} (x_{1},...,x_{m} ~|~ t_1,...,t_m ) := \sum_{n_i \geq 1}{\rm Li}_{n_{1},...,n_{m}}(x_{1},...,x_{m} )t_1^{n_1-1}... t_m^{n_m-1}.
$$
Then one has  the product formula
\begin{equation} \label{11.21.0.1}
\begin{split}
&{\rm Li}(x_1,...,x_m~|~ t_1,...,t_m)\cdot {\rm Li}(x_{m+1},...,x_{m+n}~|~ t_{m+1},...,t_{m+n})  =\\
&\sum_{\sigma \in \Sigma_{m,n}}{\rm Li}(x_{\sigma (1)},...,x_{\sigma (m+n) }~|~t_{\sigma ( 1) }, 
...,t_{\sigma (m+n)})  + \mbox{the lower depth terms}.\\
\end{split}
\ee

On the other hand, let us set
\begin{equation} \nonumber
\begin{split}
&{\rm I}[a_1,...,a_m~|~t_1,...,t_m]:= \sum_{n_1>0, ..., n_m>0}{\rm I}_{n_1, ..., n_m}(a_1,...,a_m) t^{n_1-1}_1 \ldots t^{n_m-1}_m.\\
&{\rm I}^![a_1,...,a_m~|~t_1,...,t_m]:= {\rm I}[a_1,...,a_m~|~t_1, t_1+t_2,t_1+t_2+t_3,...,t_1+...+t_m].\\
\end{split}
\end{equation}

Then according to    \cite[Lemma 2.10]{G01} one has  
$$
{\rm I}^![x_1,...,x_m~|~t_1,...,t_m] = \int_0^1\frac{ s^{-t_1}}{x_1 -
  s}ds \circ ... \circ \frac{s^{-t_m}}{x_m - s}ds.
$$
Therefore the classical formula for the product of the iterated integrals  immediately implies:
\begin{theorem} \label{shuffle}
\begin{equation} \label{shuffle1}
\begin{split}
&{\rm I}^![a_1,...,a_{k}~|~t_1,...,t_{k}]\cdot
{\rm I}^![a_{k+1},...,a_{k+l}~|~t_{k+1},...,t_{k +l}] =\\
&\sum_{\sigma \in \Sigma_{k,l}}{\rm I}^![a_{\sigma (1)},...,a_{\sigma (k+l) }~|~
t_{\sigma ( 1) },...,t_{\sigma (k+l)}].\\
\end{split}
\ee
\end{theorem}

 \paragraph{3. Motivic multiple polylogarithms.}  Let $\F$ be a number field, and ${\cal O}_{\F, S}$ the ring of $S-$integers in $\F$.  
 Let us denote by ${\cal L}_\bullet({\cal O}_{\F, S})$ the graded dual to the fundamental Lie algebra   ${\rm L}_\bullet({\cal O}_{\F, S})$. Then, if $x_i \in \F$,  combining \cite{G02b} with \cite{DG} we get the motivic multiple polylogarithm elements
 \be \la{MMPa}
 {\rm Li}^{\cal M}_{n_{1},...,n_{m}}(x_{1},...,x_{m})  \in {\cal L}_\bullet({\F}).
 \ee

Let us assume that $x_i \in \mu_N$. Recall the cyclotomic scheme ${\rm S_N}$, see (\ref{SN}). Then we have
  \be \la{MMPb}
 {\rm Li}^{\cal M}_{n_{1},...,n_{m}}(x_{1},...,x_{m})  \in {\cal L}_\bullet({\rm S}_N), ~~~~x_i \in \mu_N.
 \ee
 
 The following basic result was proved in \cite{G02a}. 
 \bt \la{THH3.2} Motivic multiple polylogarithms (\ref{MMPa}) satisfy   regularized double shuffle relations (\ref{11.21.0.1}) $\&$ (\ref{shuffle1}). 
 \et
 
 \paragraph{4. The cyclotomic Lie coalgebra.} Let ${\cal G}_w(\mu_N)$ be the subspace of ${\cal L}_w({\rm S}_N)$ generated by the elements (\ref{MMPb}) of weight $w$. Let us set
\be \la{CLC}
  {\cal G}_\bullet(\mu_N) :=  \bigoplus_{w\geq 1}  {\cal G}_w(\mu_N).
\ee

 There is a depth filtration  on the space ${\cal G}_w(\mu_N)$ such that the depth $\leq m$ subspace is generated by the elements (\ref{MMPb}) of the weight $w$ and depth $\leq m$.
 We denote by $\widetilde {\cal C}_{w,m}(\mu_N)$ the associate graded of $ {\cal G}_{w}(\mu_N)$ for the depth filtration, and set 
\be \la{CLC1}
\widetilde {\cal C}_{\bullet, \bullet}(\mu_N) :=  \bigoplus_{w, m \geq 1}\widetilde {\cal C}_{w, m}(\mu_N).
\ee

\bt  \begin{enumerate}

\item  The coproduct provides (\ref{CLC}) with a structure of a graded Lie coalgebra. 

\item  The depth filtration is compatible with the coproduct,   inducing a filtration on   Lie coalgebra (\ref{CLC}). 
Therefore (\ref{CLC1}) is a bigraded Lie coalgebra.

\item The cyclotomic Lie coalgebra $\widetilde {\cal C}_{\bullet, \bullet}(\mu_N)$  is canonically identified with the Lie coalgebra (\ref{CLC1}):
$$
\widetilde {\cal C}_{\bullet, \bullet}(\mu_N) =  {\cal C}_{\bullet, \bullet}(\mu_N).
$$
\end{enumerate}
\et

\begin{proof} 1)-2). It follows from the coproduct formula for the motivic iterated integrals \cite{G02b} that (\ref{CLC1})
 is a Lie subcoalgebra of ${\cal L}_\bullet({\rm S}_N)$, and that the depth filtration is respected by the coproduct. 
 
 3). Follows from \cite{G01}, where the Hodge version of the claim is proved, combined with \cite{DG} which allows to transform it to the motivic set-up. 
\end{proof}



\bt \la{SMCO} There is a canonical surjective homomorphism of bigraded Lie coalgebras
\be \la{DC}
\xi: {\cal D}_{\bullet, \bullet}(\mu_N) \lra \widetilde {\cal C}_{\bullet, \bullet}(\mu_N).
\ee
\et

\begin{proof}

Assume that $x_i^N=1$. We define map (\ref{DC}) on the generators by setting (cf \cite[(67)]{G02b}):
$$
\xi: \{x_1,...,x_m\}_{n_1,...,n_m}   \longmapsto  {\rm Li}^{\cal M}_{n_1,...,n_m}(x_1,...,x_m).
$$
It is a well defined map of bigraded spaces by Theorem \ref{THH3.2}.  Comparing the formulas for the coproduct in the dihedral Lie coalgebra with the one for the motivic multiple polylogarithms 
considered modulo products and projected to the associated graded for the depth filtration, see \cite[Theorem 5.4]{G02b} as well as Proposition 6.1 there,  we see that this map commutes with the coproduct. It is surjective by the definition. 
 \end{proof}

   \section{Feynman graphs  and geometry of  symmetric spaces} 
\subsection{Voronoi complexes  }

The {\it Voronoi cell complex} is a  
${\rm GL}_m(\Z)$-invariant cell decomposition of the principal 
symmetric space for ${\rm GL}_m(\R)$. The chain complex of this 
cell decomposition is called the {\it Voronoi complex}. 
One can compactify  open simplices of the Voronoi 
decomposition by adding some simplices at infinity.
  For example for ${\rm GL}_2$ this amounts to 
adding the cusps parametrized by  $P^1(\Q)$ 
to the hyperbolic plane. The corresponding 
chain complex is called the {\it extended Voronoi complex}. 

\paragraph{The Voronoi  cell decomposition of the symmetric space 
for ${\rm {\rm GL}}_m(\R)$.} Pick a rank $m$ lattice 
${\rm L_m}$. 
Recall the   vector space $V_m$ dual to ${\rm L_m} \otimes \R$, and the cone ${\cal P}(V_m)$ (respectively $\overline {\cal P}(V_m)$)   of positive (respectively non-zero, non-negative)  definite quadratic forms in    $V_m$. 
Recall the  symmetric space ${\Bbb H}_{\rm L_m}$,  
 compactified   by the space $\overline {\Bbb H}_{\rm L_m}$.  The space of  quadratic forms  in $V_m$ is denoted by $Q(V_m)$. 
Any vector $l \in {\rm L_m}$ defines a degenerate    quadratic form 
$$
\varphi(l):= (l,x)^2 \in \overline {\cal P}(V_m).
$$
Take  the convex hull ${\cal C}({\rm L_m})$ of  the vectors $\varphi(l)$ in 
the cone $\overline {\cal P}(V_m)$ when $l$ runs through all non-zero primitive 
vectors of the lattice ${\rm L_m}$. Then ${\cal C}({\rm L_m})$ is of codimension $1$ 
  and has a structure of 
an infinite polyhedra.  Its faces are convex polyhedra $\varphi(l_1,...,l_n)$ with vertices $\varphi(l_1), ... , \varphi(l_n)$, $l_i \in {\rm L_m}$.  

Then the  faces of ${\cal C}({\rm L_m})$ are polyhedra $\varphi(l_1,...,l_n)$ for appropriately chosen vectors $l_1,...,l_n \in {\rm L_m}$. The projection of ${\cal C}({\rm L_m})$ defines a ${\rm {\rm Aut}}({\rm L_m})$-invariant polyhedral decomposition of a certain part of $\overline {\cal P}(V_m)/\R_+^*$, denoted 
${\Bbb H}^*_{{\rm L_m}}$. It is  the union of ${\Bbb H}_{{\rm L}}$  over all sublattices   ${\rm L}$: 
$$
{\Bbb H}^*_{{\rm L_m}} =   {\Bbb H}_{{\rm L_m}} \bigcup  \bigcup_{{\rm L} \subset {\rm L_m}}{\Bbb H}_{{\rm L}}. 
$$

\paragraph{Example.} ${\cal P}(V_2)$ is a cone $x^2 +y^2 < z^2$, $z \geq 0$ in $\R^3$, and  
${\Bbb H}_{{\rm L_2}}$ is the hyperbolic plane ${\cal H}$. 
It as of the interior of the unit disc, and 
 $\overline {\cal P}(V_2)/\R_+^*$ is the unit disc and 
${\Bbb H}^*_{{\rm L_2}} = {\cal H} \cup P^1(\Q)$.

\vskip 2mm
The vectors of  ${\rm L_m} -\{ 0\}$ minimizing the values of a quadratic 
form $F$ on ${\rm L_m} -\{ 0\}$ are called 
{\it  minimal vectors}
of $F$.

\bd A quadratic form $F$ in $V_m$ is  {\em perfect} if the number of minimal vectors of $F$ is at least $\frac{m(m+1)}{2} = {\rm dim}Q({\rm L_m})$.
\ed

Let $s$ be a codimension $1$ face of ${\cal C}({\rm L_m})$. 
Let $h(s)$ be the codimension $1$ subspace in the space of the quadratic forms $Q(V_m)$, parallel to the face $s$.

\paragraph{Voronoi's lemma.} {\it A quadratic form $F$ on $V_m$  is orthogonal to the subspace $h(s)$ 
if and only if the form $F$ is  perfect. In this case $\pm l_1,...,\pm l_n$ are precisely the set of minimal vectors for $F$.}

\begin{proof} One has $(F, \varphi(l)) = F(l)$. 
Let $(F, x) =c$ be the equation of the hyperplane $h(s)$. Since ${\cal C}({\rm L_m})$ 
is a convex hull it is located in just one of the subspaces $(F, x) < c$ or $(F, x) > c$. Since $(F, \varphi(l)) = F(l)$ could be arbitrary big, we conclude that the domain $\{x| (F, x) < c\}$ 
does not intersect ${\cal C}({\rm L_m})$. Further, 
$(F, \varphi(l)) =c$ for any vertex $\varphi(l)$ of the face $\varphi$, so such $l$'s are minimal vectors for $F$. Since the face $\varphi$ is of codimension $1$, 
the number of its vertices is at least ${\rm dim} {Q}(V_m)$. So the form $F$ is perfect. \end{proof}

{\it Voronoi cells of  type $A_m$}. Let $v_1,...,v_{m+1}$ be an extended basis of ${\rm L_m}$. So $v_1 + ... + v_{m+1} = 0$. 
Set 
\begin{equation} \label{1!}
v_{i,j}:= v_i + v_{i+1} + ... +  v_{j-1}+ v_j, \qquad 1 \leq i, j \leq m+1, \quad i \not = j-1
\end{equation}
Then the configuration of vectors $v_{i,j}$ in (\ref{1!})
is linearly equivalent to the configuration of the roots of the root system $A_n$. 
See Figure \ref{mp2}  
for the configuration  points on the projective plane 
corresponding to the root system $A_3$.  
The convex hull of $\varphi(v_{i,j})$ is a  top dimensional cell of the Voronoi decomposition 
 and  
the corresponding perfect form is the quadratic form of 
the root system $A_m$ with the the set of minimal vectors given by the roots. 

The quadratic form 
of the root system $D_m$ provides another top dimensional Voronoi cell.

It is known that the number of different types of perfect forms, and thus equivalence classes 
of Voronoi sells 
of top dimension, grow very fast with $m$, and so unlikely to be completely understood. However in small dimensions ($m \leq 4$) we have a very clear picture of the Voronoi's cell decomposition, which is extensively used in this paper. Here is one of the main results.

\paragraph{Voronoi's theorem \cite{V}, \cite{M}.} {\it a) For $m=2,3$ any cell of top dimension in the Voronoi decomposition of ${\Bbb H}_{\rm L_m}$ is ${\rm GL}_m(\Z)$-equivalent to a cell of type $A_m$.

b) Any cell of top dimension in the Voronoi decomposition of ${\Bbb H}_{\rm {\rm L_4}}$ is ${\rm GL}_4(\Z)$-equivalent to a cell of type $A_4$ or $D_4$}.

\paragraph{Voronoi complexes.} 
 Let 
$$
({\Bbb V}_{\bullet}^{(m)}, d) =({\Bbb V}_{\bullet}({\rm L_m}), d) :=   {\Bbb V}^{(m)}_{\frac{m(m+1)}{2} -1} \stackrel{d}{\lra} {\Bbb V}^{(m)}_{\frac{m(m+1)}{2} -2} 
\stackrel{d}{\lra} ... \stackrel{d}{\lra}  {\Bbb V}^{(m)}_{0}
$$ 
be the chain complex of 
the Voronoi decomposition of 
${\Bbb H}^*_{{\rm L_m}}$. We  call it the extended Voronoi complex of the lattice ${\rm L_m}$, or simply  {\it the extended Voronoi complex for ${\rm GL}_m(\Z)$}.

The part of the Voronoi complex generated by the polyhedrons whose interior  is inside of the symmetric space is called {\it the Voronoi complex}. 
We denote it by ${V}^{(m)}_{\bullet}$ or ${V}_{\bullet}({\rm L_m})$. It has the following shape:
$$
({V}_{\bullet}^{(m)}, d) =({V}_{\bullet}({\rm L_m}), d) :=   {V}^{(m)}_{\frac{m(m+1)}{2} -1} \stackrel{d}{\lra} {V}^{(m)}_{\frac{m(m+1)}{2} -2} 
\stackrel{d}{\lra} ... \stackrel{d}{\lra}  {V}^{(m)}_{m-1}.
$$ 

Say that a polyhedron of the Voronoi decomposition has rank $r$ if it is a convex hull of the forms $\varphi(l_i)$ such that the vectors $l_i$ span a sublattice of rank $r+1$. 
The interior part of the Voronoi complex consists of all polyhedrons of rank $m-1$. 
We get a decreasing filtration of the Voronoi  complex given by the rank of cells. 

The extended Voronoi complex ${\Bbb V}_{\bullet}({\rm L_m})$ 
is the total complex of the following bicomplex:
$$
{V}_{\bullet, \bullet}({\rm L_m}) = \oplus_r {V}_{\bullet, r}({\rm L_m}), \qquad {V}_{\bullet, r}({\rm L_m}) = \quad \bigoplus_{{\rm L} \subset {\rm L_m}}{V}_{\bullet}({\rm L_m})
$$
where the sum is over all rank $r$ sublattices of ${\rm L_m}$. 
The differentials are 
$$
\partial: {V}^{(m)}_{i,j} \lra {V}^{(m)}_{i-1,j}, \quad \partial': {V}^{(m)}_{i,j} \lra {V}^{(m)}_{i-1,j-1}.
$$
The numeration is different from the standard one in a bicomplex. To get the standard one  use 
${V}^{(m)}_{i-j,j}$ instead of ${V}^{(m)}_{i,j}$. In our case $i$ is the degree in the complex and $j$ is the rank. 
 For example for ${\rm {\rm GL}_2}$ we get
\begin{displaymath} 
    \xymatrix{
        {V}^{(2)}_{2,1} \ar[r]^{\partial}  & 
{V}^{(2)}_{1,1}
\ar[d]^{\partial'} \\     
& {V}^{(2)}_{0,0}\\}
\end{displaymath}

For ${\rm {\rm GL}_3}$ the Voronoi bicomplex looks as follows:
\begin{displaymath} 
    \xymatrix{
        {V}^{(3)}_{5,2}\ar[r]^{\partial}  & 
{V}^{(3)}_{4,2} \ar[r]^{\partial}  & {V}^{(3)}_{3,2}\ar[r]^{\partial}\ar[d]^{\partial'}& {V}^{(3)}_{2,2}\ar[d]^{\partial'} \\     
& &{V}^{(3)}_{2,1}\ar[r]^{\partial}  &{V}^{(3)}_{1,1}\ar[d]^{\partial'}\\
&&&{V}^{(3)}_{0,0}}
\end{displaymath}
Compare with the shape of the modular bicomplex. It makes sense to consider the Voronoi bicomplex for the infinite rank lattice. 

\paragraph{The Voronoi complex and cohomology of subgroups of ${\rm {\rm GL}}_m(\Z)$.} 
Let  $W$ be  
finite dimensional  ${\rm {\rm GL}}_m$-module. 
For a   torsion free subgroup  $\Gamma \subset {\rm {\rm GL}}_m(\Z)$  
the group cohomology 
${\rm H}^*(\Gamma, W)$ is isomorphic to the cohomology of $\Gamma \backslash 
{\Bbb H}_m$ with coefficients in the local system ${\cal L}_W$ assigned to $W$. 

For any finite index subgroup $\Gamma \hookrightarrow {\rm GL}_m(\Z)$ there is a normal 
torsion free finite index subgroup $\widetilde \Gamma \hookrightarrow  \Gamma$.
So if $W$ is a $\Q$-rational ${\rm GL}_m$-module, by the 
Hochshild-Serre spectral sequence  
 \begin{equation} \nonumber
{\rm H}^*(\Gamma, W) \quad = \quad {\rm H}^*(\widetilde \Gamma, W)^{\Gamma/\widetilde \Gamma}. 
 \end{equation}
  Using this  we reduce the study of cohomology  
${\rm H}^*(\widetilde \Gamma, W)$ to the case when $\Gamma$ is torsion free.

\begin{theorem} \label{3.29.01.1} 
Let $\Gamma$ be a  finite index subgroup   of ${\rm GL}_m(\Z)$. Then for any 
${\rm GL}_m$-module $W$ 
\begin{equation} \label{3.29.01.2}
 {\rm H}_i\Bigl(V^{(m)}_{\bullet} \otimes_{\Gamma} W \Bigr)  = \left\{ \begin{array}{ll}
 {\rm H}^{d_m-i}(\Gamma, W\otimes \varepsilon)&: 
\quad \mbox{if $m$ is even}, \\ 
{\rm H}^{d_m-i}(\Gamma, W)&: \quad \mbox{if $m$ is odd}. \\
 \end{array}\right.
\end{equation}
\end{theorem}

\begin{proof}
The complex $V^{(m)}_{\bullet} \otimes_{\Gamma} W$   
is the complex of chains with 
coefficients in the local system ${\cal L}_W$; 
it is relative to the Voronoi triangulation of 
${\Gamma \backslash {\Bbb H}_{\rm L_m}}$, which has finitely many cells, 
 hence giving the Borel-Moore homology. 
By the Poincar\'e duality we get cohomology of 
${\Gamma \backslash {\Bbb H}_{\rm L_m}}$  with coefficients in the local 
system ${\cal L}_{W}$ twisted by the orientation class, i.e. ${\cal L}_{W \otimes 
\varepsilon_m}$, for even $m$, and ${\cal L}_{W}$ for odd $m$.  \end{proof}

\subsection{Modular complexes, graphs,   and  symmetric spaces} \la{Sec4.2}

The {relaxed modular complex} $\widehat {\rm M}_{(m)}^{\bullet}$ from Section \ref{Sec2aa}  
 has  a canonical realization   in the 
symmetric space for ${\rm GL}_m(\R)$, given by the higher modular symbol complex.  We use it in the next chapters to 
compare the modular and Voronoi complexes   for 
${\rm GL}_2(\Z)$, ${\rm GL}_3(\Z)$  and ${\rm GL}_4(\Z)$. Similarly we define the extended 
complex $\widehat {\Bbb M}_{(m)}^{\bullet}$ placed in degrees $[1,2m-1]$, and its geometric realization.

Our   construction has an important generalization. 
We introduce the rank $m$  colored forest complex $F_{(m)}^{\bullet}$ and 
define its {\it canonical} geometric realization in 
 the  symmetric space for ${\rm GL}_m$. 
The colored forests complex contains 
the relaxed modular complex, and it is quasiisomorphic to it.

\paragraph{\bf 1. The geometric realization of the relaxed modular complex.}   
Let  $v_1,...,v_{n_1}$ and $v_{n_1+1},...,v_{n_1+n_2}$ be two sets of 
vectors of the lattice such that the lattices generated by 
these sets   has zero intersection. Then we set
$$
\varphi(v_1,...,v_{n_1}) \ast \varphi(v_{n_1+1},...,v_{n_1+n_2}) := 
  \varphi(v_1,..., v_{n_1+n_2}).
$$
We extend the join $\ast$ by linearity.

\begin{theorem} \label{15} 
There exists a canonical  morphism of complexes: 
\begin{displaymath} \la{FFa}
    \xymatrix{
        {\widehat {\rm M}}_{(m)}^{1} \ar[r]^{ } \ar[d]^{ \widehat \psi_{(m)}^\bullet} & 
\ldots\ar[r]^{}    & {\widehat {\rm M}}_{(m)}^{m} \ar[d]^{\widehat \psi_{(m)}^\bullet} \\     
V_{(m)}^{1}\ar[r]^{} &\ldots \ar[r]^{ } & {V}_{(m)}^{m}\\}
\end{displaymath}
such that, setting $a_i := |A_i|$, we have:
\begin{equation} \label{hh}
\widehat 
\psi^k_{(m)}: [A_1] \wedge ... \wedge [A_k]   \lms
\widehat \psi^1_{(a_1)} [A_1] \ast ... \ast  \widehat 
\psi^1_{(a_k)} [A_k]. 
\end{equation}
 The map
$\widehat \psi^{\bullet}_{(m)}:{\widehat {\rm M}}^{(m)}_{\bullet} \lra 
{V}^{(m)}_{\bullet} $  
extends to a  map of bicomplexes and   the associated complexes
$$
\widehat \Psi^{\bullet, \bullet}_{(m)}: {\widehat {\rm M}}_{(m)}^{\bullet, \bullet} \lra 
{V}_{(m)}^{\bullet, \bullet}, \qquad 
\widehat \Psi^{\bullet}_{(m)}: {\widehat {\Bbb M}}_{(m)}^{\bullet} \lra 
{\Bbb V}_{(m)}^{\bullet}. 
$$ 
\end{theorem}

\paragraph{\bf Example.} The very right component of the homomorphism
  from Theorem  \ref{15} is given by 
$$
\widehat \psi^{m-1}_{(m)}\Bigl([v_1] \wedge ... \wedge [v_m] \Bigr) = \varphi(v_1, ... , v_m).
$$  
\vskip 2mm

To define such a map $\widehat \psi^{\bullet}_{(m)}$ one needs  to define 
$\widehat \psi^{1}_{(m)}[v_1,...,v_m]$ for vectors $v_1,...,v_m$ forming a 
basis of the lattice ${\rm L_m}$ so that the dihedral and the 
first shuffle relations go to zero 
and 
$$
d \widehat \psi^{1}_{(m)} [v_1,...,v_m] = \widehat  \psi^{2}_{(m)}\circ \partial [v_1,...,v_m].
$$
Here the right hand side is  
computed by (\ref{hh}) and the formula for $\partial $.  

The proof of Theorem \ref{15} 
is given below.

\paragraph{2. The map $\widehat \psi^{1}_{(m)}$.}   
   Recall that each edge $E$ of a tree ${\rm T}$ colored by an extended basis $(e_0, ..., e_m)$ of  ${\rm L}_m$ provides a vector $f_E \in {\rm L_m}$ 
defined up to a sign, see Figure \ref{mp1ba}. 
Following (\ref{FFFa}), we set 
\be \la{FFF}
\begin{split}
&\widehat \psi^{1}_{(m)}: \langle e_0,e_1,...,e_m \rangle \lms 
 \sum_{\mbox{T}}{\rm sgn} (E_1 \wedge ... \wedge E_{2m-1})\cdot 
\varphi(f_{E_1}, ... , f_{E_{2m-1}}).\\
\end{split}
\ee
Here the sum is over all plane trivalent trees ${\rm T}$ colored by $e_0,...,e_m$. 

To establish the properties of the map $\widehat \psi^{1}_{(m)}$ we use 
the fact that the tree complex considered by 
 Boardman and Kontsevich is a resolution of the cyclic
 Lie (co)operad.

\paragraph{\bf 3. The graph complex resolution of the cyclic Lie co-operad \cite{GK}.} 
Let ${\cal C}{\cal A}ss_{\bullet}$ be the cyclic associative operad \cite{GK}. 
This means, in particular, the following. 
Let ${\cal A}ss_m \{e_1,...,e_m\}$ 
be the space of all associative words formed by the letters $e_1,...,e_m$, each used once. 
Then ${\cal C}{\cal A}ss_{m+1}$ is a vector space generated by the expressions 
$$
(A, e_0), \quad \mbox{where} \quad A \in {\cal A}ss_m \{e_1,...,e_m\}; \quad 
(xy, z) = (x, yz), \quad (x,y) =(y,x)
$$ 
i.e. $(*,*)$ is an invariant scalar product on an associative  algebra. 
So as a vector space ${\cal C}{\cal A}ss_{m+1}$ is isomorphic to the space 
of $m$-ary operations ${\cal A}ss_{m}$ in the associative operad, or simply speaking to 
${\cal A}ss_m \{e_1,...,e_m\}$. 

Similarly let ${\cal C}{\cal L}ie_{\bullet}$ be the cyclic Lie operad. Let 
 ${\cal L}ie_m\{e_1,...,e_m\}$ be the space of Lie words formed by the letters $e_1,...,e_m$, each used once. 
As a vector space, ${\cal C}{\cal L}ie_{m+1}$ is isomorphic to the space of 
$m$-ary operations in the Lie operad. We think about it as of 
the space generated by the expressions 
$$
(L, e_0), ~L \in {\cal L}ie_m \{e_1,...,e_m\}; \quad 
([x,y],z) = (x, [y,z]), \quad (x,y) =(y,x),
$$ 
i.e. $(*,*)$ is an invariant scalar product on a Lie algebra. 

\paragraph{\bf Example.} The space ${\cal C}{\cal L}ie_{3}$ is one dimensional. It is generated by $([e_1, e_2], e_0)$. 
The space ${\cal C}{\cal A}ss_{3}$ is two dimensional, generated  
by $(e_1e_2,e_0)$ and $(e_2e_1, e_0)$. 

\vskip 3mm

Denote by 
${\cal C}{\cal A}ss^*_{m+1}$ (${\cal C}{\cal L}ie^*_{m+1}$) 
 the $\Q-$vector space dual to ${\cal C}{\cal A}ss_{m+1}$ 
(respectively ${\cal C}{\cal L}ie_{m+1}$).  
We claim  that there is canonical isomorphism
\begin{equation}  \label{MB}
{\cal C}{\cal L}ie^*_{m+1} = \frac{{\cal C}{\cal A}ss^*_{m+1}}
{\mbox{Shuffle relations}},
\end{equation}
where the denominator is defined as follows. Let $(e_0 e_1 ... e_m)^* \in 
{\cal C}{\cal A}ss^*_{m+1}$ be a functional whose value 
on $(e_1 ... e_m, e_0 )$ is $1$ and on $(e_{i_1} ... e_{i_{m}}, e_0 )$ is 
zero if $\{i_1, ..., i_m\} \not = 
\{1,2, ..., m\}$ as ordered sets. 
The  denominator  is a  subspace generated by the expressions 
$$
\sum_{\sigma \in \Sigma_{k,m-k}} (e_0 e_{\sigma (1)} ... e_{\sigma (m)})^*; 
\qquad 1 \leq k \leq m-1.
$$

To prove (\ref{MB}) notice that 
${\cal L}ie_m\{e_1,...,e_m\} \subset {\cal A}ss_m\{e_1,...,e_m\}$ is a 
subspace of primitive elements with respect to the 
coproduct $\Delta$, $\Delta(e_i) = e_i \otimes 1 + 1 \otimes e_i$. This implies that 
the dual vector space  ${\cal L}ie^*_m\{e_1,...,e_m\}$ is the quotient of 
${\cal A}ss^*_m\{e_1,...,e_m\}$ by the shuffle relations. 

\paragraph{4. The tree complex \cite{K1}.} Denote by ${\rm T}_{(m)}^i$ the abelian group generated 
by the isomorphism classes of the pairs  $(T, {\rm Or}_T)$ where $T$ is a 
tree (not  a plane tree!) with $2m-i$ edges and 
whith $m+1$  ends colored by the set $\{e_0, ..., e_m\}$. 
Here ${\rm Or}_T$ is an orientation of he tree $T$. The only relation is 
that changing the orientation of the tree $T$ amounts changing the sign of 
the generator. 
There is a differential defined by shrinking of internal edges of a tree $T$: 
\be \nonumber
\begin{split}
&d: {\rm T}_{(m)}^i \lra {\rm T}_{(m)}^{i+1},\\
&(T, {\rm Or}_{T}) \lms \sum_{\mbox{internal edges $E$ of $T$}} (T/E, {\rm Or}_{T/E}).\\
\end{split}
\ee
Here if ${\rm Or}_{T} = E \wedge E_1 \wedge ... $, then  
${\rm Or}_{T/E}:= E_1 \wedge ...$. 

The group  ${\rm T}_{(m)}^i$ is a free abelian group with a basis; the basis vectors 
are defined up to a sign. Thus there is a perfect pairing 
${\rm T}_{(m)}^i \otimes {\rm T}_{(m)}^i \lra \Z$, and we 
may identify the group ${\rm T}_{(m)}^i$ with its dual $({\rm T}_{(m)}^i)^*:= 
{\rm Hom}({\rm T}_{(m)}^i, \Z)$. 
Consider the map $w$
$$
 \Bigl({\rm T}_{(m)}^1\Bigr)^*   \stackrel{w}{\lra}  
{\cal C}{\cal L}ie_{m+1}
  \stackrel{i}{\hookrightarrow}  {\cal C}{\cal A}ss_{m+1}
$$
defined as follows. Take a trivalent tree $T$ colored by 
$e_0, e_1, ..., e_m$ 
(not a plane tree). Make a tree rooted at $e_0$ out of it. 
Define a Lie word 
in $e_1, ..., e_m$ using this tree. This is the element $w(T) \subset {\cal L}{\cal L}ie_{m+1}$. Then consider it as an associative word. We get $i \circ w (T)$. 

There are the dual maps
$$
{\cal C}{\cal A}ss^*_{m+1} \stackrel{i^*}{\lra}  
{\cal C}{\cal L}ie^*_{m+1}   \stackrel{w^*}{\lra}  
{\rm T}_{(m)}^1.
$$

We get a complex
\begin{equation} \label {GKap1}
{\cal C}{\cal L}ie^*_{m+1} \quad \stackrel{w^*}{\lra} \quad 
{\rm T}_{(m)}^1 \stackrel{d}{\lra}{\rm T}_{(m)}^2 \stackrel{d}{\lra}
 ... \stackrel{d}{\lra}{\rm T}_{(m)}^{m-2}.
\end{equation}

\begin{theorem} \label {GKap}
The complex (\ref{GKap1}) is exact.
\end{theorem}

\begin{proof} This is a reformulation of one of the main result of \cite{GK}. 
 Notice that the tree complex in \cite{GK}  
is dual to ours. 

\begin{definition} \label{3.28.01.1}
An unordered extended basis $\{e_0, ..., e_m\}$ is a set of $m+1$ 
vectors of the lattice ${\rm L_m}$  whose sum is zero and  
any $m$ of them provide a basis of ${\rm L_m}$.
\end{definition}

 Notice that an extended basis $(e_0, ..., e_m)$ is ordered. 
One has a natural isomorphism
\begin{equation}  \label{3.28.01.2}
\widehat {\rm M}_{(m)}^1 =  \bigoplus_{\mbox{unordered extended bases}} 
{\cal C}{\cal L}ie_{m+1}^*\{e_0, ..., e_m\}.
\end{equation}

\paragraph{\bf 5. Proof of Theorem \ref{15}.} Let us 
show that the map $\widehat \psi^{1}_{(m)}$ defined on  generators by (\ref{FFF})
 gives rise to a group homomorphism 
$$
\widehat \psi^{1}_{(m)}: \widehat {\rm M}^{1}_{(m)} \lra  V^{1}_{(m)}.
$$

\begin{lemma} \label{15AQ}
a) The homomorphism 
$\widehat \psi^{1}_{(m)}$ can be written as a composition: 
$$
\widehat \psi^{1}_{(m)}([e_1,...,e_m]) =  (w^* \circ i^*) (e_0 e_1 ... e_m)^*.
$$

b) The map $\widehat \psi^{1}_{(m)}$ sends the first shuffle and dihedral symmetry relations to zero. 
\end{lemma} 

\begin{proof} a) The image of the $(e_0 e_1 ... e_m)^* $ under the map $w^* \circ i^* $ is a functional on the colored oriented 
$3$-valent trees. It is $+1$ if the tree is isomorphic to a plane 
trivalent tree colored by $ e_0, ..., e_m$ with the standard 
orientation induced by the orientation of the circle, and it 
is zero otherwise. 

b) Since $i^* $ kills the shuffle relations, 
the first shuffle relations go to zero under 
the homomorphism 
$\psi^{1}_{(m)}$. 
It is obvious that the map $\widehat \psi^{1}_{(m)}$ sends the cyclic symmetry relations to zero. 
 To check that it sends to zero the dihedral symmetry relations notice that 
if we reverse the orientation of the circle 
then the orientation of the colored plane trivalent tree is multiplied by $(-1)^{m-1}$.  This can be seen by the induction on $m$: chopping off an internal vertex with two legs we change the sign of the permutation of the edges induced by changing the plane orientation. 
\end{proof}

Let us compute the boundary of the chain $\widehat \psi^{1}_{(m)}([e_1,...,e_m])$.  
Consider first the contribution of 
 legs. A parametrized by $e_i$  
leg of a  plane $3$-valent tree   
provides a partition 
\begin{equation} \label{UYU}
\{e_0, ..., e_m\} \backslash e_i = \{e_{i+1}, ..., e_{i+j}\} \cup \{e_{i+j+1}, ..., e_{i-1}\}.
\end{equation}
\begin{figure}[ht]
\centerline{\epsfbox{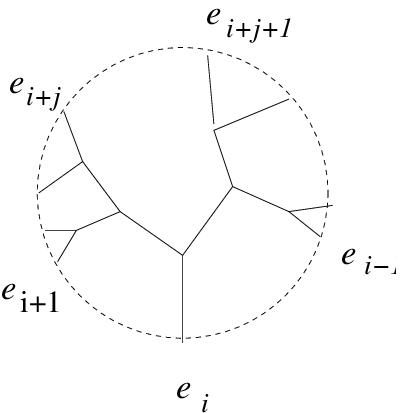}}
\caption{} 
\label{mp1f}
\end{figure}
It is easy to see that the sum over all colored plane $3$-valent trees which provide 
 this partition for the given leg $e_i$ is  
\begin{equation} \label{UYU.}
\widehat \psi^{2}_{(m)}\Bigl([e_{i+1}, ..., e_{i+j}] \wedge [e_{i+j+1}, ..., e_{i-1}]\Bigr).
\end{equation}
Taking the sum over all legs and all possible types of partitions (\ref{UYU}) is the same as taking the sum over all $i,j$ in (\ref{UYU.}). So we get 
$\widehat \psi^{2}_{(m)} \partial \langle e_{0}, ..., e_{m}\rangle$.

The contribution of the internal edges to the boundary of $\widehat \psi^{1}_{(m)}([e_1,...,e_m])$ is zero because
 $$
(w^* \circ i^*) (e_0 e_1 ... e_m)^* \in {\rm Ker} \Bigl( {\rm T}_{(m)}^1 \stackrel{d}{\lra} {\rm T}_{(m)}^2\Bigr).
$$ 
\end{proof} 
 
\paragraph{\bf 6. The colored forest complex.}  

A  forest $F$ is a graph without loops. It is a union 
if its connected components, i.e. trees: $F = T_1 \cup ... \cup T_k$.  
The edges of a forest 
consist of legs, i.e. external edges,  
and internal edges. A forest may have no internal vertices. 

Let us choose   a lattice ${\rm L_m}$. 
A {\it colored forest} is a forest  $F$ equipped with the following data, see Figure \ref{mp1Aa}. 
Let $F = T_1 \cup ... \cup T_k$ be the decomposition of a forest $F$ into trees, and  
$$
{\rm L}^{(1)} \oplus ... \oplus {\rm L}^{(k)} \subset {\rm L_m}
$$
 be  a 
sublattice of ${\rm L_m}$ given by a direct sum of $k$ nonzero sublattices 
 such that the number of legs of the tree $T_i$ is 
equal to 
${\rm rk} ~{\rm L}^{(i)}+1$.  
Choose an extended basis in each of the sublattices  ${\rm L}^{(i)}$ and 
label the legs of the tree $T_i$ by 
the vectors of this extended basis. 
Then we say that the forest $F$ is colored - by the vectors of the extended  basis
 of the lattices ${\rm L}^{(i)}$.

\begin{figure}[ht]
\centerline{\epsfbox{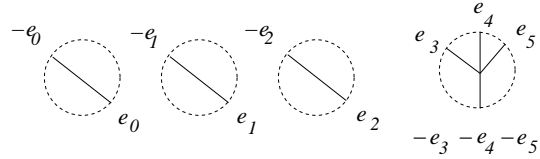}}
\caption{ A colored forest.}
\label{mp1Aa}
\end{figure}

Let us  define the rank of a colored forest as 
${\rm rk}~ ({\rm L}^{(1)} \oplus ... \oplus {\rm L}^{(k)})$. So
$$
\mbox{number of legs of $F$} \quad = \quad\mbox{number of 
connected components of $F$} \quad +  
\quad\mbox{$\Sigma$ ranks of ${\rm L}^{(i)}$}.
$$
The orientation $\Z/2\Z-$torsor  ${\cal O}_F$ of a forest $F$ is a tensor product  
${\cal O}_F = \otimes_{i=1}^k {\cal O}_{T_i}$. 

Let ${\Bbb F}_{(m)}^{i}$ (resp. $F_{(m)}^{i}$) be the abelian group generated 
by the isomorphism classes of pairs $(F, {\rm Or}_F)$ where $F$ is a 
colored forest (resp. rank $m$ colored forest ) 
with 
 $2m - i$ edges,  and ${\rm Or}_F$ is an orientation of the forest. The only relation is that 
if we alter the orientation of the forest, the generator is multiplied by $-1$. 

\paragraph{\bf Example.} The group $F_{(m)}^{1} = {\Bbb F}_{(m)}^{1} $ is generated by 
colored $3$-valent trees with $m+1$ legs. 
\vskip 2mm

Let us define a differential 
$d: {\Bbb F}_{(m)}^{i} \lra {\Bbb F}_{(m)}^{i+1}$. One has $d = d_I + d_L$. 
 
The map $d_I$ is defined as follows. Let $(F, {\rm Or}_F)$ be a generator of 
${\Bbb F}_{(m)}^{i}$ and $E$ an {\it internal} edge of $F$. Let $F/E$ be the forest
 obtained by contracting  the edge $E$. 
So it has one less edge, and one less vertex. The orientation ${\rm Or}_F$ 
of the graph $F$ induces a natural orientation ${\rm Or}_{F/E}$ 
of the graph $F/E$. Namely, if ${\rm Or}_F= E \wedge E_1 \wedge E_{2 }\wedge  ...$ then ${\rm Or}_{F/E} = E_1 \wedge E_{2 }\wedge  ... $. Contructing an internal 
edge we do not touch the coloring. So $d_I$ is  the differential 
used by Boardman and Kontsevich.

Let us define the map $d_L$. Let $L$ be a leg of $F$. 
Let us remove this leg together with a little neighborhood of its vertices. 
If both of the vertices of $L$ has valency $1$ this means that 
$L$ is a connected component of $F$, so we just removed this connected component. Otherwise 
one of the vertices has the valency $v \geq 3$, and the connected component of the tree 
containing the leg $L$ is replaced by 
$v-1$ connected trees. Namely, each of   internal edges sharing the vertex with $L$ produces a new tree. The coloring of the forest $F$ provides a 
  coloring of the new forest, see Figure \ref{mp1Bb}. 
\begin{figure}[ht]
\centerline{\epsfbox{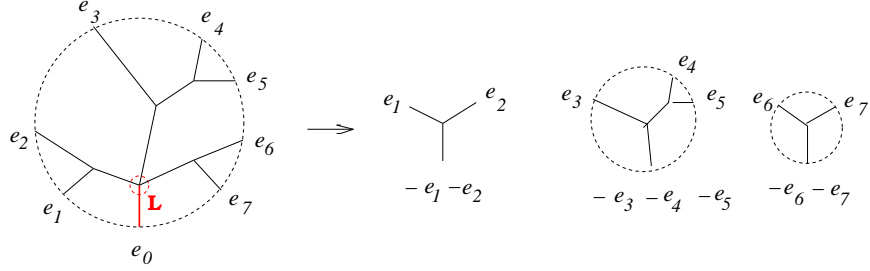}}
\caption{Cutting out a leg $L$. }
\label{mp1Bb}
\end{figure} 
Precisely, the coloring of the legs of $T$ different from $L$ remains unchanged; each of the new trees has one new leg, and the vector attached to it is the minus sum of the vectors atached to the rest of the legs of this tree.

It is easy to see that we get a complex, denoted 
${\Bbb F}_{(m)}^{\bullet}$ and called the colored forests complex.

The rank provides a filtration on the   colored forests complex. 
The $m$-th associated graded   
with respect to this filtration is called the {\it rank $m$ colored forests complex}  
and denoted by ${\rm F}_{(m)}^{\bullet}$. 
It is generated by the oriented colored forests of rank $m$, and the differential is defined the same way as before 
except one detail: by definition the component of the differential 
corresponding to a leg $L$ which forms a connected component of the graph $F$ is zero. 

The rank $m$  colored forests complex 
is  concentrated in degrees $[1, m]$. It differs from the  graph complexes 
 studied by  Boardman and Kontsevich since in those complexes there were
 no $d_L$-component of the differential. 

The relaxed modular complex is a subcomplex of 
colored forests complex: there is a natural inclusion of complexes
$$
w_{(m)}^{\bullet}: \widehat {\rm M}^{\bullet}_{(m)} \hookrightarrow F^{\bullet}_{(m)}.
$$ 
The map $w_{(m)}^1$ is given by the map $w^*$  and (\ref{GKap1})-(\ref{3.28.01.2}). 
To define the maps 
$w_{(m)}^k$ let us define an operation $\circ$ on colored forests by
$$
(F_1, {\rm Or}_{F_1}) \circ (F_2, {\rm Or}_{F_2}) := (F_1 \cup F_2, 
{\rm Or}_{F_1} \wedge {\rm Or}_{F_2}). 
$$
Here the forests $F_i$ are colored by vectors of the lattices ${\Bbb L}^i$, and 
$F_1 \cup F_2$ is colored by the vectors of the lattice ${\Bbb L}^1 \oplus {\Bbb L}^2$.
The coloreing of $F_1 \cup F_2$ is defined in a natural way using the coloring of 
$F_1$ and $F_2$. We extend the operation $\circ$ by linearity and set
$$
w_{(m)}^k ([A_1] \wedge ... \wedge [A_k]) := w_{(a_1)}^1 [A_1] \circ ...\circ 
 w_{(a_1)}^1 [A_k], \qquad a_i = |A_i|. 
$$

\begin{theorem} \label{TTCCqq}
a) The map $w_{(m)}^{\bullet}$ is a quasiisomorphism. So the 
complex $F^{\bullet}_{(m)}$ 
is a resolution of the  complex $\widehat {\rm M}^{\bullet}_{(m)}$. 

b) Similarly the complex ${\Bbb F}^{\bullet}_{(m)}$  is a resolution of the 
 complex $\widehat {\Bbb M}^{\bullet}_{(m)}$. 
\end{theorem}

\begin{proof} a) 
The complex $F^{\bullet}_{(m)}$ has a filtration by the 
number of connected components of the forest. The $E_1$-term of the 
corresponding spectral sequence is a sum over all unordered extended basis of 
the  tree complexes.  
By   Ginzburg-Kapranov's theorem \cite{GK} the tree complex with $m+1$ legs 
is a resolution of the space of $m$-ar operations in the cyclic Lie 
cooperad ${\cal C}{\cal L}ie_{m+1}^*$. 
It remains to recall   isomorphism (\ref{3.28.01.2}).

b) Completely similar. \end{proof}

\paragraph{\bf 7. A geometric realization of the complex of colored forests.}

\begin{theorem-construction} \label{TTCC} 
a) There exist canonical homomorphisms of complexes
$$
\tau_{(m)}^{\bullet}: {F}_{(m)}^{\bullet} \lra {V}^{\bullet }_{(m)}; \qquad
\tau_{(m)}^{\bullet}: {\Bbb F}_{(m)}^{\bullet} \lra {\Bbb V}^{\bullet }_{(m)}
$$

b) Restricting these homomorphisms to the subcomplexes $ {\rm M}_{(m)}^{\bullet} 
\hookrightarrow {F}_{(m)}^{\bullet}$ and $ {\Bbb M}_{(m)}^{\bullet} 
\hookrightarrow {\Bbb F}_{(m)}^{\bullet}$ we get 
the homomorphisms $\widehat \psi^{\bullet}_{(m)}$ and 
$\widehat \Psi^{\bullet}_{(m)}$.
\end{theorem-construction}

\begin{proof} a) The construction of the map $\tau_{(m)}^{\bullet}$ 
 generalizes  the construction of the map 
$\widehat \psi^{\bullet}_{(m)}$. 

Let $(F, {\rm Or}_F)$ be a generator of 
${\Bbb F}_{(m)}^i$. Each edge $E$ of the forest $F$ provides a vector $f_E \in {\rm L_m}$ 
defined up to a sign. Namely, let $T$ be the connected component of   $F$ containing $E$. Then $T$ is a colored tree, and the   construction   on Figure \ref{mp1ba}
 provides a vector $f_E$   defined up to a sign.

Let $E_1, ... , E_{2m-2-i}$ be the set of all edges of $F$. 
Let ${\rm sgn}(E_1 \wedge ... \wedge E_{2m-2-i}) = +1$ if ${\rm Or}_F = E_1 \wedge ... \wedge E_{2m-2-i}$, 
and $-1$ otherwise. Set 
$$
\tau^{(m)}_i: (F, {\rm Or}_F) \lms {\rm sgn} (E_1 \wedge ... \wedge E_{2m-2-i}) \cdot 
\varphi(f_{E_1}, ... , f_{E_{2m-2-i}}) \in {\Bbb V}^{i}_{(m)}. 
$$
Let us show that we get a morphism of complexes. The cell $\varphi(f_{E_1}, ... , f_{E_{2m-2-i}})$ is a simplex, so its boundary is a union of 
simplices corresponding to the edges of the tree. The contribution of an {\it internal}
edge $E$ is precisely the cell corresponding to  $(F/E, {\rm Or}_{F/E})$. 
Similarly the contribution to the boundary of a leg $L$ of the tree $F$ is given by the $d_I$ component of the differential. 

b) It follows immediately from the construction of 
the homomorphism $w^{\bullet}_{(m)}$. 
\end{proof}

\section{Modular and   Voronoi complexes for ${\rm GL}_m$} \la{SEC5}

The extended modular complex ${\Bbb M}_{(m)}^{\bullet}$ is a 
cohomological complex placed in degrees $[1,2m-1]$. Let us cook up out of him a homological complex sitting in degrees $[2m-2, 0]$ by setting
$$
{\Bbb M}^{(m)}_{\bullet}:= {\Bbb M}_{(m)}^{2m-1 - \bullet}.
$$
Using the same formula to the modular complex we get a 
complex ${\rm M}^{(m)}_{\bullet}$  in degrees $[2m-2, m-1]$.

\subsection{Relating modular and Voronoi complexes for  ${\rm GL}_2$}

The extended Voronoi complex for ${\rm GL}_2$ looks as follows:
$$
{\Bbb V}^{(2)}_{2} \stackrel{d}{\lra} {\Bbb V}^{(2)}_{ 1 } 
\stackrel{d}{\lra}   {\Bbb V}^{(2)}_{0}.
$$
It is the chain complex of the classical triangulation of ${\cal H}^*:= {\cal H}\cup {\Bbb P}^1(\Q)$, see Figure \ref{mod}.


The Voronoi complex is the chain complex of the triangulation of ${\cal H}$:  
$$
{V}_{\bullet}^{(2)}:= 
{V}^{(2)}_{2} \stackrel{d}{\lra} {V}^{(2)}_{ 1 }.
$$

 According to Theorem-Construction \ref{15} there is a homomorphism of complexes   
$$
\Psi_{\bullet}^{(2)}: {\Bbb M}^{(2)}_{\bullet} \lra {\Bbb V}_{\bullet}^{(2)} 
$$
 defined as follows. Let $e_1,e_2, e_3$ be an extended  basis of ${\rm L_2}$, and $e \in {\rm L_2}$ a primitive vector.  Set
\begin{equation} \nonumber
[e] \lms \varphi(e), \quad [e_1] \wedge [e_2] \lms \varphi(e_1,e_2),
\qquad 
 [e_1,e_2] \lms \varphi(e_1,e_2,e_3).  
\end{equation}

The rank $2$ colored forest complex is canonically isomorphic to 
 the relaxed modular complex. 

\vskip 2mm 
Let us recall 
few facts about the cohomology of a subgroup $\Gamma \subset {\rm GL}_2(\Z)$ 
following  \cite[Section 2.3]{G00b}. 
Let $V$ be a   
${\rm GL}_2$-module.  For a torsion free subgroup $\Gamma \subset {\rm GL}_2(\Z)$ 
the group cohomology 
${\rm H}^*(\Gamma, V)$ is isomorphic to the cohomology of $\Gamma \backslash 
{\Bbb H}_2$ with coefficients in the local system ${\cal L}_V$ corresponding to $V$.
Notice that 
${\rm H}^*(\Gamma \backslash {\Bbb H}_2, {\cal L}_V ) = 
{\rm H}^*(\overline {\Gamma \backslash {\Bbb H}_2}, 
Rj_*{\cal L}_V )$ where 
$j: \Gamma \backslash {\Bbb H}_2 \hookrightarrow \overline {\Gamma \backslash 
{\Bbb H}_2}$.  

For a torsion free finite index subgroup 
$\Gamma$ of ${\rm GL}_2(\Z)$
one defines the cuspidal cohomology ${\rm H}^*_{{\rm cusp}}(\Gamma, V)$ as the cohomology 
of $\overline {\Gamma \backslash {\Bbb H}_2}$ with coefficients in a middle 
extension of the local system 
${\cal L}_V$. In our case the middle extension means 
the sheaf $j_*{\cal L}_V$. 
We define the cuspidal cohomology for arbitrary finite index subgroup $\Gamma$ 
by reducing it to the  torsion free case by formula similar to  (\ref{4-11.1d}).  
The following results is \cite[Lemma 2.3]{G97}. 
 
\bt \label{thc2} 
  a) The map $\Psi_{\bullet}^{(2)}$ is an isomorphism of complexes of ${\rm GL}_2(\Z)$-
modules. Restricting it
 to ${\rm M}^{(2)}_{\bullet}$ we get an isomorphism $\psi_{\bullet}^{(2)}: {\rm M}^{(2)}_{\bullet} 
\stackrel{}{\lra} {V}_{\bullet}^{(2)}$. 

b) Let $\Gamma$ be a  finite index subgroup   of $GL_2(\Z)$. Then for any $\Q$-rational 
${\rm GL}_2$-module $V$ there are  canonical isomorphisms
\be \nonumber
\begin{split}
&{\rm MH}_{(2)}^i(\Gamma, V)   =  
{\rm H}^{i-1}(\Gamma, V\otimes \varepsilon_2),  \\
& {\Bbb M}{\rm H}_{(2)}^i(\Gamma, V)  =  
{\rm H}_{\rm cusp}^{i-1}(\Gamma, V\otimes \varepsilon_2).\\
\end{split}
\ee
\et

\begin{proof} a) When $(e_1,e_2, e_3)$ run  through all extended basis of the lattice
$L_2$,  the   triangles $\varphi(e_1, e_2, e_3)$  
are Voronoi's cells of type $A_2$, and so by Voronoi's theorem produce 
  all the $2$-cells of  Voronoi's complex for ${\rm GL}_2$. 
  So we 
 get an isomorphism of complexes or bicomplexes. 

b) The first equality follows from Theorem \ref{3.29.01.1}.  
The dual to the complex
$
{\Bbb M}_{(2)}^* \otimes_{\Gamma} V[1]
$  
is a  cochain complex computing 
the cohomology of $\overline {\Gamma \backslash {\Bbb H}_2}$ with coefficients 
the   middle extension sheaf 
$j_*{\cal L}_{V^{\vee}}$.  
For the complex itself by Poincar\'e duality we get cohomology of 
 $\overline {\Gamma \backslash {\Bbb H}_2}$ with value in the middle extension of the local 
system ${\cal L}_{V}$ twisted by the orientation class, i.e. $j_*{\cal L}_{V \otimes 
\varepsilon_2}$. \end{proof}

\subsection{Relating   modular and Voronoi complexes for ${\rm GL}_3$}

\paragraph{1. Rank $3$  colored forest complex and the 
truncated Voronoi complex for ${\rm GL}_3$.}

The extended Voronoi complex for ${\rm GL}_3$ looks as follows: 
$$
({\Bbb V}^{(3)}_{\bullet}, d) :=   
{\Bbb V}^{(3)}_{5} \stackrel{d}{\lra} {\Bbb V}^{(3)}_{4} 
\stackrel{d}{\lra} ... \stackrel{d}{\lra}  {\Bbb V}^{(3)}_{0}.
$$ 
  The Voronoi complex is:
$$
({V}^{(3)}_{\bullet}, d) :=  
{V}^{(3)}_{5} \stackrel{d}{\lra} {V}^{(3)}_{4} 
\stackrel{d}{\lra}   {V}^{(3)}_{3} \stackrel{d}{\lra}   {V}^{(3)}_{2}.
$$
Notice that ${V}^{(3)}_{i} = {\Bbb V}^{(3)}_{i}$ for $3 \leq i \leq 5$, but ${V}^{(3)}_{2} = {\Bbb V}^{(3)}_{2,0} \not = {\Bbb V}^{(3)}_{2}$.
Let us cook up a cohomological complex out of the extended Voronoi complex 
for ${\rm GL}_3$ by setting
$
{\Bbb V}_{(3)}^{\bullet}:= {\Bbb V}^{(3)}_{5-\bullet}
$. 
Let $s^{\geq p}C^{\bullet}:= C^p \lra C^{p+1} \lra ... $ be the stupid truncation of a complex $C^{\bullet}$. 
By Theorem-Construction \ref{TTCC} there is a   map of complexes 
\begin{equation}  \label{3mmorr} 
\tau^{\bullet}_{(3)}: {\Bbb F}_{(3)}^{\bullet} \lra s^{\geq 1}{\Bbb V}_{(3)}^{\bullet}.
\end{equation}
\begin{theorem}  \label{thc3a.} 
The morphism (\ref{3mmorr}) is an isomorphism. 
\end{theorem}

\begin{proof}     
In this chapter   we suppose that $e_1,e_2,e_3, e_4$ is  an  extended
   basis of ${\rm L_3}$ and 
$$
f_{12}:= e_1 + e_2, \quad f_{23}:= e_2 + e_3, \quad ... 
$$

By Voronoi's theorem the ${\rm GL}_3(\Z)$-orbits of the $5$-simplex $\varphi( e_1,e_2, e_3, e_4,   
f_{12}, f_{23})$ and its faces provide all cells of the Voronoi decomposition for ${\rm GL}_3$. 
The following  facts about the   Voronoi 
cell decomposition   are  easy to see using this fact from the figure below:


A cell of the Voronoi complex which is not contained 
in the boundary is called an {\it interior cell}. 
Any interior Voronoi cell is ${\rm GL}_3(\Z)$-equivalent to just one cell of the following list:

\begin{enumerate}

\item 
$2$-dimensional cell:
$$
\varphi(e_1, e_2, e_3).
$$

\item $3$-dimensional cells:
\be \nonumber
\begin{split}
&\mbox{a $3$-cell $\varphi(e_1, e_2, -f_{12}, e_3)$, called {\it special $3$-cell}, or}\\
& \mbox{a $3$-cell $\varphi(e_1, e_2, e_3, e_4)$, called 
{\it generic $3$-cell}}.  \\
\end{split}
\ee

\item $4$-dimensional cell:
$$
\varphi(e_1, e_2, e_3, e_4, f_{12}).
$$

\item $5$-dimensional cell:
$$
\varphi(e_1, e_2, e_3, e_4, f_{12}, f_{23}).
$$
\end{enumerate}

Therefore generic  $3$-cells  are parametrized by the set  $E({\rm L_3})$ 
of unordered extended basis of ${\rm L_3}$. 

\vskip 2mm
{\it Verification}. We will show how to deduce the statement for $3$-dimensional cells. 
$3$-dimensional faces of the simplex 
$\varphi(e_1,e_2, e_3, e_4, f_{12},f_{23})$ may have corank $0$ or $1$. 
It is easy to see that there are only three corank zero $3$-faces of this simplex:
\begin{equation} \label{SSREL}
\varphi(e_1,e_2, e_3, e_4), \quad \varphi(f_{12}, e_3, -f_{23}, -e_1), \quad 
\varphi(f_{12}, e_4, f_{23}, -e_2).
\end{equation} 
 All of them are ${\rm GL}_3(\Z)$-equivalent since
we have
$$
e_1 + e_2 + e_3 + e_4 = 0, \quad f_{12} + e_3 - f_{23} - e_1  = 0, 
\quad f_{12} + e_4 +  f_{23} - e_2  = 0.
$$
A similar argument shows that all corank one $3$-faces of the simplex $\varphi(e_1,e_2, e_3, e_4, f_{12},f_{23})$ are equivalent. Since by Voronoi's theorem all top dimensional cells 
are ${\rm GL}_3(\Z)$-equivalent, the statement is  proved.

The generators of the rank $3$  colored 
forests complex are given by the following graphs

\begin{figure}[ht]
\centerline{\epsfbox{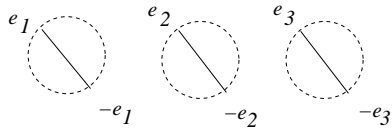}}
\caption{A colored forest for   a 2-cell.}
\label{mp1Cc1}
\end{figure}

 \begin{figure}[ht]
\centerline{\epsfbox{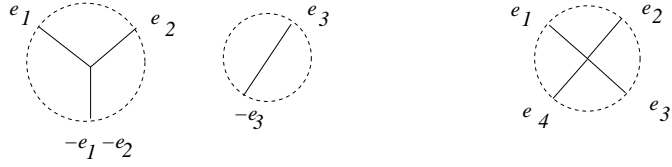}}
\caption{Colored forests for  special (on the left) and generic (on the right) 3-cells.}
\label{mp1Cc2}
\end{figure}

\begin{figure}[ht]
\centerline{\epsfbox{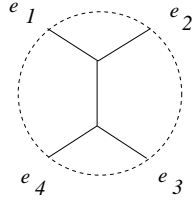}}
\caption{A colored tree for a 4-cell.}
\label{mp1Cc3}
\end{figure}

Since they match with the list of the Voronoi cells given above we get the theorem. \end{proof}

\paragraph{\bf 2. A quasiisomorphism between the extended modular and 
(truncated) Voronoi complexes for ${\rm GL}_3$.} Restricting the map 
$\tau^{\bullet}_{(3)}$ to the subcomplex 
$\widehat {\rm M}^{(3)}_{\bullet}$ we get an injective morphism of complexes

\begin{displaymath} 
    \xymatrix{
         \widehat {\rm M}^{(3)}_4 \ar[r]^{\partial} \ar[d]^{\widehat \psi^{(3)}_4} & 
 \widehat {\rm M}^{(3)}_{3} \ar[r]^{\partial}  \ar[d]^{\widehat \psi^{(3)}_3} & \widehat {\rm M}^{(3)}_{2} \ar[d]^{\widehat \psi^{(3)}_2} \\     
V^{(3)}_{4}\ar[r]^{d} &V^{(3)}_{3}\ar[r]^{d} & {V}^{(2)}_{0,0}\\}
\end{displaymath}
Recall that it is defined on the generators as follows:
\be  \nonumber
\begin{split}
&[e_1] \wedge [e_2]\wedge [e_3] \lms \varphi(e_1,e_2,e_3),\\
 &[e_1,e_2] \wedge [e_3] \lms \varphi(e_1,e_2,-e_1-e_2, e_3),\\ 
 \end{split}
\end{equation}
\be   \label{*} 
\begin{split}
&[e_1,e_2,e_3] \lms \varphi(e_1,e_2, e_3, e_4, f_{12}) - \varphi(e_1,e_2, e_3, e_4,  f_{23}).\\
\end{split}
\end{equation}
We are going to show that it provides a morphism of complexes
\be
\psi_{\bullet}^{(3)}: {\rm M}^{(3)}_{\bullet} \lra
{V}_{\bullet}^{(3)}/ d {V}_5^{(3)}.
\ee 
Then we extend it to a morphism of the bicomplexes
$$
\psi_{\bullet, \bullet}^{(3)}: {\rm M}^{(3)}_{\bullet, \bullet} \lra
{V}_{\bullet, \bullet}^{(3)}/ d {V}_5^{(3)}
$$
using the formulas 
for the map $\psi_{\bullet}^{(2)}$:
$$
[e_1,e_2] \lms \varphi(e_1,e_2,-e_1-e_2), \quad [e_1] \wedge [e_2] 
\lms \varphi(e_1,e_2),\quad [e_1] \lms \varphi(e_1).
$$

\begin{theorem}  \label{thc3} 
a) The map $\psi_{\bullet}^{(3)}$ is an injective morphism of  complexes 
of ${\rm GL}_3(\Z)$-modules. 
It is a  quasiisomorphisms.

b) Let $\Gamma$ be a subgroup   of ${\rm GL}_3(\Z)$. Then for any ${\rm GL}_3$-module $V$ there are  
canonical isomorphisms
$$ 
{\rm H}{\rm M}_{(3)}^i(\Gamma, V) \otimes \Q  =   {\rm H}^i(\Gamma, V)\otimes \Q, \qquad i \geq 1
$$ 

c) Similarly $\psi_{\bullet, \bullet}^{(3)}$ 
is an injective morphism of bicomplexes and it is a quasiisomorphism. 
\end{theorem}

\begin{proof}  
a)  One has $\widehat {\rm M}^{(3)}_{i} = {\rm M}^{(3)}_{i}$ for $ i \not = 4$, and $\widehat {\rm M}^{(3)}_{4}$ 
is a quotient of ${\rm M}^{(3)}_{4}$ by the 
second shuffle relations. So to show that $\psi_{\bullet}^{(3)}$ is a well defined morphism of complexes we need to check only that the 
map $\widehat \psi_{4}^{(3)}$ sends the second shuffle relations to boundary of certain $5$-chain.

The second shuffle relation looks as follows:
\be  \la{eqa}
\begin{split}
&s[u_1|u_2:u_3] := [u_1:u_2:u_3] + [u_2:u_1:u_3] +[u_2:u_3:u_1] =\\
&[u_2-u_1,u_3-u_2,-u_3] + [u_1-u_2,u_3-u_1,-u_3] +[u_3-u_2,u_1-u_3,-u_1].\\
\end{split}
\ee
Recall the notation 
$$
\langle e_1,e_2,e_3,e_4\rangle := [e_1,e_2,e_3] \quad \mbox{provided} \quad e_1+e_2+e_3+e_4=0.
$$
Changing  the variables $e_1:= u_2-u_1, e_2:= u_3-u_2, e_3:= -u_3$ we write (\ref{eqa}) as  
$$
 \langle e_1,e_2,e_3,e_4\rangle  + \langle -e_1,f_{12},e_3,f_{41}\rangle + \langle e_2,-f_{12},-e_4,f_{41}\rangle.
$$
The 
map $\widehat \psi_{4}^{(3)}$ sends it to the boundary of  Voronoi's  $5$-simplex $\varphi(e_1,e_2, e_3, e_4, f_{12},f_{23})$:
\be \nonumber
\begin{split}
& \varphi(e_1,e_2, e_3, e_4, f_{12}) - \varphi(e_1,e_2, e_3, e_4,  f_{23}) + 
\varphi(e_1, f_{12}, e_3, f_{23}, e_2) \\
&- \varphi(e_1, f_{12}, e_3, f_{23}, e_4) +
\varphi(e_2,f_{12},e_4,f_{23},e_1) - 
\varphi(e_2,f_{12},e_4,f_{23},e_3) = \\
&d \varphi(e_1,e_2, e_3, e_4, f_{12},f_{23}).\\
\end{split}
\ee
It is easy to check that the map $\widehat \psi_{\bullet}^{(3)}$  is injective. 
  Since the cells $\varphi(e_1,e_2, e_3, e_4, f_{12},f_{23})$ are all top 
dimensional cells of the Voronoi complex for ${\rm GL}_3$ by Voronoi's Theorem, this 
implies that 
the map $\psi_{\bullet}^{(3)}$ is injective. 
It is an isomorphism in all the degrees except $3$ and $4$. 

The part a) follows from this and Theorem \ref{thc3a.}. 

The implication a) $=>$ b) follows from  Theorem \ref{3.29.01.1}.  
The part c) is straitforward.  
\end{proof}

\paragraph{3. A resolution for the rank $3$ modular complex.}
Let ${\cal M}^{(3)}_{5}$ be a group generated by the symbols $\{e_1,e_2,e_3,e_4\}_5$ where 
$(e_1,e_2,e_3, e_4)$ run through all extended basis of the lattice ${\rm L_3}$; the only 
relations between them are the dihedral symmetry relations 
$$
\{e_1,e_2,e_3,e_4\}_5 = \{e_2,e_3,e_4,e_1\}_5 = \{e_4,e_3,e_2,e_1\}_5 = 
\{-e_1, -e_2, -e_3, -e_4\}_5. 
$$
There is a group homomorphism 
$$
{\cal M}^{(3)}_{5} \stackrel{\partial}{\lra} 
\widehat {M}^{(3)}_{4}, \qquad \{e_1,e_2,e_3,e_4\}_5 \lms 
[e_1,e_2,e_3] + [ -e_1,f_{12},e_3]  + 
[e_2,-f_{12},-e_4].
$$
It sends the generator $\{e_1,e_2,e_3,e_4\}_5 $ to the second shuffle relation. 
So 
$$
{\rm M}^{(3)}_{4}= \quad {\rm Coker} \Bigl( {\cal M}^{(3)}_{5} \stackrel{\partial}{\lra} 
\widehat {\rm M}^{(3)}_{4} \Bigr).
$$
Let 
$$
{\cal M}^{(3)}_{4} = \widehat {\rm M}^{(3)}_{4}, \quad {\cal M}^{(3)}_{3} = {\rm M}^{(3)}_{3},  \quad 
{\cal M}^{(3)}_{2} = 
{\rm M}^{(3)}_{2}.
$$
We get a complex
$$
{\cal M}^{(3)}_{5} \stackrel{\partial}{\lra}  {\cal M}^{(3)}_{4} \stackrel{\partial}{\lra}
 {\cal M}^{(3)}_{3}  \stackrel{\partial}{\lra}   
{\cal M}^{(3)}_{2}.
$$

\begin{theorem} \label{RESA}
a) There is an injective  morphism of complexes $\widetilde   \psi^{(3)}_{\bullet}: 
{\cal M}^{(3)}_{\bullet} \lra V^{(3)}_{\bullet}$:  
 
\begin{displaymath} 
    \xymatrix{
      {\cal M}^{(3)}_{5}\ar[r]^{\partial}  \ar[d]^{\widetilde  \psi^{(3)}_5} &     {\cal M}^{(3)}_4 \ar[r]^{\partial} \ar[d]^{\widetilde  \psi^{(3)}_4} & 
  {\cal M}^{(3)}_{3} \ar[r]^{\partial}  \ar[d]^{\widetilde  \psi^{(3)}_3} &  {\cal M}^{(3)}_{2} \ar[d]^{\widetilde  \psi^{(3)}_2} \\     
V^{(3)}_{5}\ar[r]^{d}&V^{(3)}_{4}\ar[r]^{d} &V^{(3)}_{3}\ar[r]^{d} & V^{(3)}_{2}\\}
\end{displaymath}

 b) It is a quasiisomorphism.
\end{theorem}

\begin{proof} Setting $\widetilde  \psi^{(3)}_5: \{e_1,e_2,e_3,e_4\}_5 \lra 
\varphi(e_1,e_2,e_3, e_4, f_{12}, f_{23})$ we get a well defined group homomorphism 
$\widetilde  \psi^{(3)}_5: {\cal M}^{(3)}_{5} \lra V^{(3)}_{5}$. It is surjective by Voronoi's
 theorem, and it is very easy to see that it is injective. The rest follows from the part a) of 
\ref{thc2}.  \end{proof}

\begin{corollary} \label{3.21.01.3}
The map $\widetilde \partial: {\cal M}^{(3)}_5 \lra {\cal M}^{(3)}_4$ is injective.
\end{corollary}

\begin{proof}  Follows from  Theorem \ref{RESA} since the 
similar map $\widetilde \partial: V^{(3)}_5 \lra V^{(3)}_4$ in the Voronoi complex is injective. \end{proof}

\paragraph{4. Miscellenious remarks.} Below we collect some 
  geometric information which should help to visualize the 
relationship between the double shuffle relations for the lattice ${\rm L_3}$ and the 
geometry of the Voronoi complex for ${\rm GL}_3$. 

i) A   generic  $3$-cell   is contained in  three $5$-simplices. 
For the $3$-cell $\varphi(e_1,e_2, e_3, e_4)$ these are 
\begin{equation} \label{three}
\varphi(e_1,e_2, e_3, e_4, f_{12},f_{23}), \quad \varphi(e_2, e_1, e_3, e_4, f_{21},f_{13}), \quad  \varphi(e_2, e_3, e_1, e_4, f_{23},f_{31}). 
\end{equation}
A $5$-simplex containing  a generic  $3$-cell $\varphi(e_1, e_2, e_3, e_4)$ 
is determined by the dihedral order of the vectors 
$e_1, e_2, e_3, e_4$.

ii) Elements (\ref{*})
 are in bijection with the pairs 
\be \la{00}
\{ \mbox{generic $3$-cell, a $5$-cell containing it}\}.
\ee   
Given a pair $(C_3, C_5)$ as in (\ref{00}), the right hand side of (\ref{*}) is the sum of the $4$-cells $C_4$ such that   $C_3 \subset C_4 \subset C_5$. 

The first shuffle relation means that the sum of the elements  of type (\ref{*})  corresponding to  $5$-cells
 containing a given generic $3$-cell is zero.

The sum of  elements (\ref{*}) corresponding to the generic 
$3$-cells (\ref{SSREL}) of a given $5$-simplex is the second shuffle relation. 
It is the boundary of that $5$-simplex. 

iii) Recall  the set  $E({\rm L_m})$ of all unordered 
extended basis  of the lattice ${\rm L_m}$. 

Let $T^{\bullet}_{(m)}\{e_0,e_1, ... , e_m\}$ be the complex of trees with 
$m+1$ legs colored by an extended basis
$e_0,e_1, ... , e_m$. Let $\overline T^{\bullet}_{(3)}\{e_0,e_1, ... , e_m\}$ 
be the quotient of this complex by the subgroup 
${\cal C}{\cal L}ss_{m+1}^* \subset T^{0}_{(m)}\{e_0,e_1, ... , e_m\}$. 
So it is acyclic by Theorem \ref{GKap}.

\begin{corollary} \label{kker}
One has
\begin{equation} \label{cocker}
{\rm Coker }(\psi^{\bullet}_{(3)}) =  \prod_{\{e_0, ..., e_3\} \in E({\rm L_3})} \overline 
T^{\bullet}_{(3)}\{e_0, ..., e_3\}. 
\end{equation} 
\end{corollary}

iv) {\it The double shuffle relations and projective duality}. Consider the configuration 
formed by $6$ points corresponding to the vectors $e_1,e_2,e_3,e_4,f_{12}, f_{23}$ and $6$ lines through them, see Figure \ref{mp2}. It is selfdual with respect to the projective duality. 
\begin{figure}[ht]
\centerline{\epsfbox{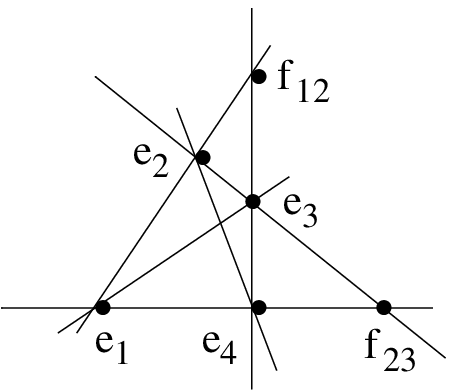}}
\caption{    } 
\label{mp2}
\end{figure}

The first shuffle relation corresponds to the three different $4$-gons with the given 
$4$ vertices, say $e_1, e_2,e_3,e_4$, which can be obtained using these lines, see Figure \ref{mp3}: 
\begin{figure}[ht]
\centerline{\epsfbox{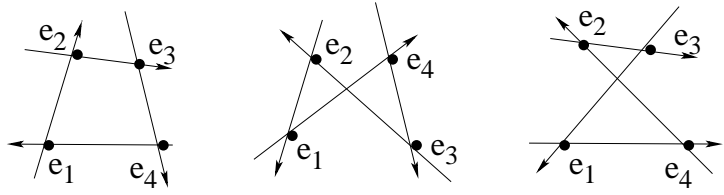}}
\caption{The first shuffle relation.} 
\label{mp3}
\end{figure}
$$
\langle e_1,e_2,e_3,e_4 \rangle  + \langle e_2,e_1,e_3,e_4 \rangle  + \langle e_2,e_3,e_1,e_4 \rangle.  
$$
   The second shuffle relation corresponds to the  
three different $4$-gons one can get 
using some of the six vertices and four given lines, see Figure \ref{mp4}.
\begin{figure}[ht]
\centerline{\epsfbox{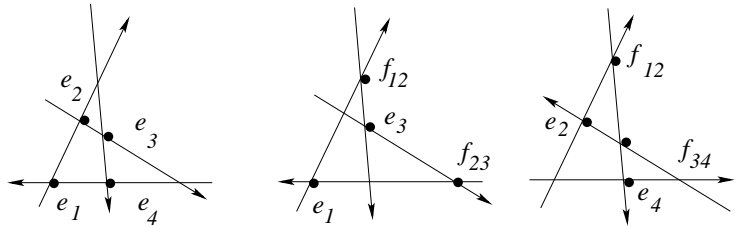}}
\caption{The second shuffle relation.} 
\label{mp4}
\end{figure}
$$
\langle e_1,e_2,e_3,e_4 \rangle  + \langle -e_1,f_{12},e_3, f_{14} \rangle  + \langle e_2,-f_{12},-e_4, f_{14} \rangle  
$$
So in a sense the projective duality interchanges the shuffle relations. 

v) {\it The dual Voronoi complex}. Consider the cell complex $\widehat {\cal H}_{{\rm L}_m}$ dual to the Voronoi cell decomposition. It's dimension is 
$d_m - (m-1)$; the cells of dimension $i$ are dual to the cells of codimension $i$ in the Voronoi complex. The symmetric space space ${\Bbb H}_{{\rm L}_m}$ can be retracted onto it.

\paragraph{Example.} For ${\rm GL}_2$ we have the famous plane $3$-valent tree, dual to the modular triangulation. 

For ${\rm GL}_3$ we get a $3$-dimensional complex, see \cite{S}.

The corank zero $3$-cells of the Voronoi complex correspond to the triangles in the dual Voronoi cell complex $\widehat {\cal H}_{{\rm L}_m}$. An element (\ref{*}) corresponds to the $1$-chain given by the
 difference of the two edges of the triangle. 
There are three such $1$-chains attached to a given triangle. Their sum is zero:
this corresponds precisely to the first shuffle relation.

\subsection{Modular and Voronoi complexes for ${\rm GL}_4$} \la{SEC6}

\paragraph{1. The homomorphism $\widehat \psi^{(4)}_6$.}
 The configuration of  points in $P^3$ 
corresponding to the roots of the root system $A_4$ is shown on Figure \ref{mls}:
\begin{figure}[ht]
\centerline{\epsfbox{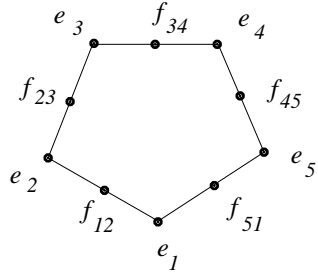}}
\caption{The projectivised configuration of  the roots of the root system $A_4$.} 
\label{mls}
\end{figure}
Here $e_1,...,e_5$ is an extended basis of the lattice ${\rm {\rm L_4}}$, so  $e_1+...+e_5 = 0$ 
 and $e_1,...,e_4$ is a basis.

Recall the map $\widehat \psi_6^{(4)}$, see (\ref{FFF}) and (\ref{Mappsi}),  defined by 
$$
\widehat \psi_6^{(4)}:   \langle  e_1,...,e_5\rangle  := [e_1,e_2,e_3,e_4]  \lms   
{\rm Cycle}_5    \varphi(e_1,...,e_5, f_{12}, f_{34}). 
$$

We will employ the following notations (compare with Section  \ref{Sec2}):
\be  \nonumber
\begin{split} &[u_1:u_2:u_3:u_4]:=   \langle u_1: u_2:u_3:u_4:0\rangle, \\
&s[u_1: ... : u_p|u_{p+1}: ... :u_m]:=  
s(u_1: ... : u_p|u_{p+1}:... :u_m:0).\\
\end{split}
\ee
The  second  shuffle relations for ${\rm GL}_4$ are given by
\be \label{simpler1}
\begin{split}
&s[u_1|u_2:u_3:u_4]:= [u_1:u_2:u_3:u_4] + [u_2:u_1:u_3:u_4] +
  [u_2:u_3:u_1:u_4] + [u_2:u_3:u_4:u_1].\\
  \end{split}
\end{equation}
and
\be \label{simpler}
\begin{split}
&s[u_1:u_2|u_3:u_4]:= \\
&[u_1:u_2:u_3:u_4] + [u_1:u_3:u_2:u_4] + [u_1:u_3:u_4:u_2] +\\
&[u_3:u_1:u_2:u_4]  + [u_3:u_1:u_4:u_2] + [u_3:u_4:u_1:u_2]. \\
\end{split}
\ee

We call the relation $[u_1: u_2: u_3: u_4] + [u_4: u_3: u_2: u_1] =0 $ 
the {\it reflection symmetry}.

\begin{lemma} \label{4more} One has 
\be \label{4m}
\begin{split}
&s[u_1:u_2|u_3:u_4] =s[u_3|u_1: u_2: u_4]  + s[u_1|u_3: u_4: u_2] - [u_1: u_2: u_4: u_3] - [u_3: u_4: u_2: u_1].\\
\end{split}
\end{equation}
Therefore in the case $n=4$ the shuffle relation of type $(2,2)$ 
is equivalent to the reflection symmetry modulo the shuffle relation of type $(1,3)$. 
\end{lemma}

\begin{proof} Straithforward. \end{proof}

According to Theorem-Construction  \ref{15} there exists 
a morphism of complexes
\begin{displaymath} 
    \xymatrix{
      \widehat {\rm M}^{(4)}_{6}\ar[r]^{\partial}  \ar[d]^{\widehat  \psi^{(4)}_6} &     \widehat {\rm M}^{(4)}_{5} \ar[r]^{\partial} \ar[d]^{\widehat  \psi^{(4)}_5} & 
  \widehat {\rm M}^{(4)}_{4} \ar[r]^{\partial}  \ar[d]^{\widehat \psi^{(4)}_4} & \widehat {\rm M}^{(4)}_{3} \ar[d]^{\widehat  \psi^{(4)}_3} \\     
V^{(4)}_{6}\ar[r]^{d}&V^{(4)}_{5}\ar[r]^{d} &V^{(4)}_{4}\ar[r]^{d} & V^{(4)}_{3}\\}
\end{displaymath}
Let us set   in (\ref{simpler1}) 
$$
e_1:= u_2-u_1, \quad e_2:= u_3-u_2, \quad e_3:= u_4-u_3, \quad e_4:= -u_4.
$$
 Then we write it as  
\begin{equation} \label{bvbva}
\begin{split}
&\langle e_1, e_2, e_3, e_4, e_5\rangle + \langle -e_1, f_{12}, e_3, e_4, f_{51}\rangle +  
 \langle e_2, -f_{12}, -f_{45}, e_4, f_{15}\rangle 
+ \langle e_2, e_3, f_{45}, -e_5, f_{15}\rangle.\\
\end{split} 
\ee

\paragraph{2. The complex ${\cal M}^{(4)}_{\bullet}$.}
To relate   modular and Voronoi complexes of the rank $4$ we  introduce the following  complex, which, as we prove below, 
  is quasiisomorphic to both the truncated rank $4$ Voronoi complex $\tau_{[6,3]}V_{\bullet}^{(4)}$ and the rank $4$ modular complex:
\begin{equation}  \la{bvbv}
{\cal M}^{(4)}_{6}/d{\cal M}^{(4)}_{7} \stackrel{\widetilde \partial}{\lra}  
{\cal M}^{(4)}_{5} \stackrel{\widetilde \partial}{\lra}
{\cal M}^{(4)}_{4} \stackrel{\widetilde \partial}{\lra}  {\cal M}^{(4)}_{3}. 
\end{equation}
Namely, we set   
\begin{equation}  \label{bvb}
\begin{split}
&{\cal M}^{(4)}_{i}:= \widehat {\rm M}^{(4)}_i, ~~3 \leq i \leq 5,\\
& {\cal M}^{(4)}_{6}:= \widehat {\rm M}^{(4)}_6 \bigoplus 
\bigoplus_{(L_3, w)} {\cal M}^{(3)}_{5}(L_3) \wedge [w].\\
\end{split}
\end{equation}
The sum in (\ref{bvb}) is over all decompositions of the lattice ${\rm {\rm L_4}}$ 
into a direct sum $L_3 \oplus \langle w\rangle $ of 
lattices of ranks $3$ and $1$; $w$ is a generator of the rank $1$ lattice $\langle w\rangle$.

Let us define a differential
$
\widetilde \partial: {\cal M}^{(4)}_{6} \lra {\cal M}^{(4)}_{5}
$.  
Its restriction to $\widehat {\rm M}^{(4)}_6 $ is the map 
$\partial : \widehat {\rm M}^{(4)}_6  \lra \widehat M^{(4)}_5 $ 
defined above, and the restriction to 
$ {\cal M}^{(3)}_{5}(L_3) \wedge [w]$ is $\partial \wedge {\rm Id}$.

\bd The subgroup  $d{\cal M}^{(4)}_{7} \subset {\cal M}^{(4)}_6 $ is generated by the following 
  elements (\ref{4.12.01.0})+(\ref{4.12.01.1})  when $(e_1, e_2, e_3, e_4)$ 
run through all  bases of the lattice ${\rm {\rm L_4}}$:
\be \label{4.12.01.0}
\begin{split}
&\langle e_1, e_2, e_3, e_4, e_5 \rangle  + \langle -e_1, f_{12}, e_3, e_4, f_{15}\rangle  \\
&+\langle e_2, -f_{12}, -f_{45}, e_4, f_{15}\rangle   + \langle e_2, e_3, f_{45}, -e_5, f_{15}\rangle \\
\end{split}
\end{equation}
\begin{equation} \label{4.12.01.1}
\begin{split}
&-\langle e_1, e_2, e_3, f_{45}\rangle  \wedge [e_4]
-\langle e_4, e_5, e_1, f_{23}\rangle  \wedge [e_2]\\
&-\langle e_5, e_1, e_2, f_{34}\rangle  \wedge [e_3]
-\langle e_3, e_4, e_5, f_{12}\rangle  \wedge [f_{51}]\\
&+\langle e_5, e_1, e_2,  f_{34}\rangle  \wedge [e_4]
+\langle e_3, e_4, e_5, f_{12}\rangle  \wedge [e_2]\\
&+\langle e_4, e_5, e_1, f_{23}\rangle  \wedge [e_3]
+\langle e_1, e_2, e_3,  f_{45}\rangle  \wedge [f_{51}].\\
\end{split}
\end{equation}
\ed
An element (\ref{4.12.01.0})+(\ref{4.12.01.1})  lifts the second 
shuffle relations of type $(1,3)$, see (\ref{bvbva}),  to ${\rm Ker} \widetilde \partial \subset {\cal M}^{(4)}_6$.

\begin{lemma} \label{sus1} $\tilde \partial (d{\cal M}^{(4)}_{7}) =0$.  
\end{lemma}

\begin{proof}  Direct calculation. \end{proof}

So we get a  map $\widetilde \partial: {\cal M}^{(4)}_{6}/d{\cal M}^{(4)}_{7} 
\to {\cal M}^{(4)}_{5}$. Other differentials are inherited from the complex  
$\widehat M^{(4)}_{\bullet}$. 

\begin{lemma} \label{il}
There is a quasiisomorphism of complexes
\begin{displaymath} 
    \xymatrix{
      {\cal M}^{(4)}_{6}/d{\cal M}^{(4)}_{7}\ar[r]^{~~~~\widetilde \partial}  \ar[d]^{p_6} &     {\cal M}^{(4)}_{5} \ar[r]^{\widetilde\partial} \ar[d]^{p_5} & 
{\cal M}^{(4)}_{4} \ar[r]^{\widetilde\partial}  \ar[d]^{p_4} &{\cal M}^{(4)}_{3} \ar[d]^{p_3} \\     
{\rm M}^{(4)}_{6}\ar[r]^{\partial}&{\rm M}^{(4)}_{5}\ar[r]^{\partial} &{\rm M}^{(4)}_{4}\ar[r]^{\partial} & {\rm M}^{(4)}_{3}\\}
\end{displaymath}
\end{lemma}

{\bf Proof}. Complex (\ref{bvbv}) contains a subcomplex
$$
\oplus_{(L_3, w)}\Bigl( {\cal M}^{(3)}_{5}(L_3) \wedge [w] \stackrel{\partial}{\lra }
\partial {\cal M}^{(3)}_{5}(L_3) \wedge [w]\Bigr).
$$
This subcomplex is acyclic since $\partial $ is injective on $ {\cal M}^{(3)}_{5}(L_3)$ 
by Corollary  \ref{3.21.01.3}.  The quotient of (\ref{bvbv}) by this subcomplex is isomorphic to 
the rank $4$ modular complex. 

 \paragraph{3. The main result.} 

\begin{theorem} \label{1more} a) There is an injective  morphism of complexes
\begin{displaymath}  
    \xymatrix{
      {\cal M}^{(4)}_{6}/d{\cal M}^{(4)}_{7}\ar[r]^{~~~~\widetilde \partial}  \ar[d]^{\widetilde  \psi^{(4)}_6} &     {\cal M}^{(4)}_{5} \ar[r]^{\widetilde\partial} \ar[d]^{\widetilde  \psi^{(4)}_5} & 
{\cal M}^{(4)}_{4} \ar[r]^{\widetilde\partial}  \ar[d]^{\widetilde\psi^{(4)}_4} &{\cal M}^{(4)}_{3} \ar[d]^{\widetilde  \psi^{(4)}_3} \\     
V^{(4)}_{6}/dV^{(4)}_{7}\ar[r]^{~~~d}&V^{(4)}_{5}\ar[r]^{d} &V^{(4)}_{4}\ar[r]^{d} & V^{(4)}_{3}\\}
\end{displaymath}

b) It  is a quasiisomorphism.
\end{theorem} 

\begin{corollary} \label{2more} 
The modular complex ${\rm M}^{(4)}_{\bullet}$ is quasiisomorphic to the truncated 
Voronoi complex $\tau_{[6,3]}{V}^{(4)}_{\bullet}$. 
Thus for a finite index subgroup $\Gamma \in {\rm GL}_4(\Z)$ and a 
$\Q$-rational ${\rm GL}_4$-module  $V$ we have
\begin{equation} \label{3more} 
{{\rm H}{\rm M}}^i(\Gamma, V) = {\rm H}^{i+2}(\Gamma, V \otimes \varepsilon) \quad \mbox{for} \quad 1 \leq i \leq 4.
\end{equation}
\end{corollary}

\begin{proof}  Combining Theorem \ref{1more} with Lemma \ref{il}  
we get the first statement. Passing to cohomological complexes we arrive to (\ref{3more}). \end{proof}

To prove Theorem \ref{1more} we need a classification of Voronoi's cells for ${\rm GL}_4$.

\paragraph{\bf 4.  Voronoi cells and colored forests for ${\rm GL}_4$.}  
We   use the following result, which can be easily 
deduced from the part b) of Voronoi's theorem, see [Mar].

\begin{lemma} \label{102}
All cells of dimension $\leq 7$  of the Voronoi complex for ${\rm GL}_4$ 
are ${\rm GL}_4(\Z)$-equivalent to the  faces of the Voronoi simplex corresponding to the quadratic 
form of the root system $A_4$. 
\end{lemma}

\begin{figure}[ht]
\centerline{\epsfbox{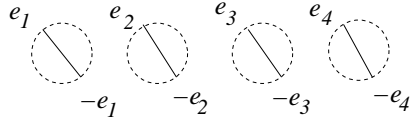}}
\caption{Colored forests corresponding to internal 3-cells.}
\label{mp1Dd3.eps}
\end{figure}

\begin{figure}[ht]
\centerline{\epsfbox{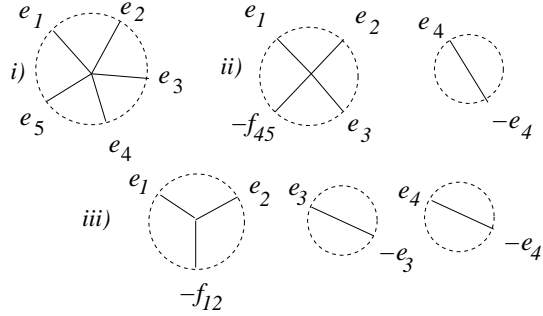}}
\caption{Colored forests corresponding to internal 4-cells.}
\label{mp1Dd4.eps}
\end{figure}

It was proved in    \cite[Lemma 2.3]{LS}, and can be checked using Lemma \ref{102},  that any 
internal Voronoi cell whose 
dimension is from $3$ to $7$  is ${\rm GL}_4(\Z)$-equivalent to just one cell from the following list:

\begin{enumerate}

\item $3$-dimensional cells:

\be
\varphi(e_1, e_2, e_3, e_4).
\ee

\item $4$-dimensional cells:
\be
{\rm i)} \quad \varphi(e_1, e_2, e_3, e_4, e_5),
\qquad
{\rm ii)} \quad  \varphi(e_1, e_2, e_3, e_4, f_{45}),
\qquad
{\rm iii)} \quad  \varphi(e_1, e_2, e_3, e_4, f_{12}).
\ee

\item $5$-dimensional cells:
\be
\begin{split}
&{\rm i)} \quad \varphi(e_1, e_2, e_3, e_4, e_5, f_{12}),\qquad
{\rm ii)} \quad \varphi(e_1, e_2, e_3, e_4, f_{45}, f_{51}),\\
&{\rm iii)} \quad  \varphi(e_1, e_2, e_3, e_4, f_{12}, f_{45}),\quad
{\rm iv)} \quad \varphi(e_1, e_2, e_3, e_4, f_{12}, f_{34}).\\
\end{split}
\ee

\item $6$-dimensional cells:
\be
\begin{split}
&{\rm i)} \quad  \varphi(e_1, e_2, e_3, e_4, e_5, f_{34}, f_{45}),
\qquad
{\rm ii)} \quad \varphi(e_1, e_2, e_3, e_4, e_5, f_{12}, f_{34}),\\
&{\rm iii)} \quad  \varphi(e_1, e_2, e_3, e_4, f_{12}, f_{13}, f_{45}),
\quad
{\rm iv)} \quad \varphi(e_1, e_2, e_3, e_4, f_{12}, f_{35}, f_{45}).\\
\end{split}
\ee

\item $7$-dimensional cells:
\be
{\rm i)} \quad \varphi(e_1, e_2, e_3, e_4, e_5, f_{12}, f_{34}, f_{45}),
\qquad
{\rm ii)} \quad  \varphi(e_1, e_2, e_3, e_4, e_5, f_{12}, f_{13}, f_{34}).
\ee
\end{enumerate}

To check this use Lemma \ref{102}. To simplify   considerations use the dihedral symmetry and 
assume, without loss of generality since the $e$'s and $f$'s appeared in the root system $A_4$ 
in a symmetric way, that at least half of the vectors are $e$'s.

\begin{figure}[ht]
\centerline{\epsfbox{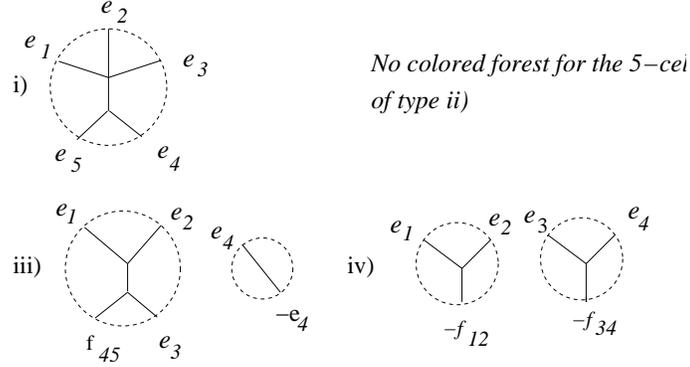}}
\caption{Colored forests corresponding to internal 5-cells. }
\label{mp1Dd5.eps}
\end{figure}

 
\begin{figure}[ht]
\centerline{\epsfbox{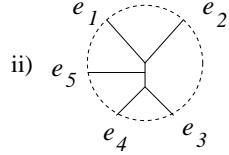}}
\caption{A colored forest corresponding to an internal 6-cell (type ii).}
\label{mp1Dd6.eps}
\end{figure}

For a $6$-cell of type iii) one has 
\begin{equation} \label{1form}
\begin{split}
&\varphi(e_1, e_2, e_3, e_4,  f_{45}, f_{12}, f_{13}) = \\
&\varphi(e_1, e_2, e_3, f_{45}, f_{12}, f_{13}) \ast \varphi(e_4) \in \widetilde \psi_5^{(4)}
\Bigl( {\cal M}^{(3)}_5(L_3) \wedge [e_4]\Bigr), \quad L_3:= 
\langle e_1,e_2,e_3\rangle.\\
\end{split}
\ee
No colored forests correspond to  $6$-cells of types i) and iv).

\paragraph{\bf 5. Proof of Theorem \ref{1more}a).} 
i) {\it The map $\widetilde \psi^{(4)}_6$ is well defined}. 
The maps $\widetilde \psi^{(4)}_{i}$ are defined on the generators by the same formulas as  the maps $\widehat \psi^{(4)}_{i}$. 
We need to  show that $\widetilde \psi^{(4)}_6$ 
is a well defined map. For this it remains to 
show that 
it  kills the subgroup $d{\cal M}^{(4)}_7 $. 
Since the Voronoi complex is acyclic in 
the degree $6$,  Lemma \ref{sus1} implies
  that one has 
$\widetilde \psi^{(4)}_6(d {\cal M}^{(4)}_7) \subset dV_7^{(4)}$. The statement 
is proved. 

In fact we can do a much better job. By Lemma \ref{4more} 
the second shuffle relation of 
type $(2,2)$ are reduced to the second shuffle relation of 
type $(1,3)$ plus reflection symmetry. Reflection symmetry relations die under the map 
$\widehat \psi^{(4)}_6$. 
The image  of the second shuffle relation of 
type $(1,3)$ under the  
map $\widehat \psi^{(4)}_6$ is computed by the following formula: 
\be \label{m6a}
\begin{split}
&\widehat \psi^{(4)}_6\Bigl( \langle e_1, e_2, e_3, e_4, e_5 \rangle  + \langle -e_1, f_{12}, e_3, e_4, f_{15} \rangle  + \\
&\langle e_2, -f_{12}, -f_{45}, e_4, f_{15} \rangle  + \langle e_2, e_3, f_{45}, -e_5, f_{15} \rangle  \Bigr) = \\
\end{split}
\ee
\be \label{m6}
\begin{split}&d\Bigl(\varphi(e_1, e_2, e_3, e_4, e_5, f_{12}, f_{45}, f_{51}) - 
\varphi(e_1, e_2, e_3, e_4, e_5, f_{23}, f_{45}, f_{51}) \\
&-\varphi(e_1, e_2, e_3, e_4, e_5, f_{12}, f_{34}, f_{51}) 
+ \varphi(e_1, e_2, e_3, e_4, f_{12}, f_{23}, f_{45}, f_{51}) \\
&+ \varphi(e_2, e_3, e_4, e_5, f_{12}, f_{34}, f_{45}, f_{51}) \Bigl)\\
\end{split}
\ee
\be \label{m6aa}
\begin{split}
&- \varphi(e_1, e_2, e_3, e_4, f_{12}, f_{23}, f_{45}) 
- \varphi(e_1, e_2, e_4, e_5, f_{23}, f_{45}, f_{51})\\
&+\varphi(e_1, e_2, e_3, e_5, f_{12}, f_{34}, f_{51}) 
+\varphi(e_3, e_4, e_5, f_{12}, f_{34}, f_{45}, f_{51}) \\
&- \varphi(e_1, e_2, e_4, e_5, f_{12}, f_{34}, f_{51}) 
- \varphi(e_2, e_3, e_4, e_5, f_{12}, f_{34}, f_{45}) \\
&+\varphi(e_1, e_3, e_4, e_5, f_{23}, f_{45}, f_{51}) 
- \varphi(e_1, e_2, e_3,  f_{12}, f_{23}, f_{45}, f_{51}).\\
\end{split}
\ee

Elements (\ref{m6aa}) are precisely the images of  elements 
(\ref{4.12.01.1})  under the map 
$\widetilde \psi^{(4)}_6$. This not only shows that the generators of the subgroup 
$d{\cal M}_7^{(4)}$ are mapped to the boundaries of some $7$-cells in the Voronoi complex, 
but describes these $7$-cells explicitly as the five cells (\ref{m6}). 

ii) {\it The maps $\widetilde \psi^{(4)}_i$ are injective for $i=3,4,5$}. The right two squares in the map $\widetilde \psi_\bullet^{(4)}$ in Theorem \ref{1more} are identified with  
a map of complexes
\begin{displaymath} 
    \xymatrix{
            \widehat {\rm M}^{(4)}_{5} \ar[r]^{\partial} \ar[d]^{\widehat  \psi^{(4)}_5} & 
  \widehat {\rm M}^{(4)}_{4} \ar[r]^{\partial}  \ar[d]^{\widehat \psi^{(4)}_4} & \widehat {\rm M}^{(4)}_{3} \ar[d]^{\widehat  \psi^{(4)}_3} \\     
 V^{(4)}_{5}\ar[r]^{d} &V^{(4)}_{4}\ar[r]^{d} & V^{(4)}_{3}\\}
\end{displaymath} 

The first two maps are given by 
\be  \label{a*} 
\begin{split}
&\widehat  \psi^{(4)}_3: [e_1] \wedge [e_2]\wedge [e_3] \wedge [e_4]\lms \varphi(e_1,e_2,e_3, e_4),\\
 &\widehat  \psi^{(4)}_4: [e_1,e_2] \wedge [e_3] \wedge [e_4]\lms \varphi(-e_1-e_2, e_1,e_2, e_3, e_4).\\ 
 \end{split}
\end{equation}

 The maps $\widehat  \psi^{(4)}_\bullet$ commute with the action of the group ${\rm GL}_4(\Z)$. The element $[e_1] \wedge [e_2]\wedge [e_3] \wedge [e_4]$ is stabilized by a finite subgroup 
 $C_4$ vof ${\rm GL}_4(\Z)$, generated by the diagonal matrices ${\rm diag}(\pm 1, \pm 1, \pm 1, \pm 1)$, and the subgroup of permutations $S_4$ - the 
 symmetry group of a 4-dimensional cube.   The same subgroup stabilizes   the cell $g\cdot \varphi(e_1,e_2,e_3, e_4)$ $g \in {\rm GL}_4(\Z)$ either coincide, or their interiors are disjoint. 
 The injectivity of the map  $\widehat  \psi^{(4)}_3$ is clear from this.
 
  The injectivity of the map  $\widehat  \psi^{(4)}_4$ is proved by a similar argument: there is a finite subgroup $C_2 \times (\Z/2\Z)^3$ stabilizing the element $[e_1,e_2] \wedge [e_3] \wedge [e_4]$ and the same subgroup stabilizes   the cell $\varphi(-e_1-e_2, e_1,e_2, e_3, e_4)$.

  The map $\widehat  \psi^{(4)}_5$ has two components:  
\be  \label{ac*} 
\begin{split}
 &\widehat  \psi^{(4)}_5: [e_1,e_2] \wedge [e_3,e_4]\lms \varphi(-e_1-e_2, e_1,e_2, -e_3-e_4, e_3, e_4),\\ 
&\widehat  \psi^{(4)}_5:  [e_1,e_2,e_3] \wedge [e_4] \lms \varphi(e_1,e_2, e_3, e_4, f_{45}, f_{12}) - \varphi(e_1,e_2, e_3, e_4, f_{45},  f_{23}).\\
 \end{split}
\end{equation}

The injectivity of its first component is proved by a similar analysis.  The subgroup stabilizing an element 
$[e_1,e_2] \wedge [e_3,e_4]$ and the cell $\varphi(-e_1-e_2, e_1,e_2, -e_3-e_4, e_3, e_4)$ 
  coincide, and the cells $g\cdot \varphi(-e_1-e_2, e_1,e_2, -e_3-e_4, e_3, e_4)$, $g \in {\rm GL}_4(\Z)$ either coincide, or their interiers are disjoint. 

The chain $\widehat  \psi^{(4)}_5([e_1,e_2,e_3] \wedge [e_4])$ is the join of $\varphi(e_4)$ and the cell $\widehat  \psi^{(3)}_4([e_1,e_2,e_3] )$. So the injectivity follows from 
the ${\rm GL_3}$-case since a linear combination of joins  $\varphi(v_i)\ast X_i$, where $X_i$ is a ${\rm GL}_3$-cell, can be zero only if each of the joins $\varphi(v_i)\ast X_i$ is zero.

iii) {\it The map $\widetilde \psi^{(4)}_6$ is injective}.  A proof by N. Malkin is given in the Appendix.

\paragraph{\bf 6. Proof of   Theorem \ref{1more}b).} The part b) is the difficult part of Theorem \ref{1more}.  Our  goal is to show that 
{\it the complex ${\rm Coker} ~\widetilde \psi_{\bullet}^{(4)}$ is acyclic}.

Suppose that ${\rm L_4} = {\rm L_3} \oplus \langle w\rangle$.  
Denote by $C({\rm L_3}, w)$ the subcomplex of   
${\rm Coker}~\widetilde \psi_{\bullet}^{(4)}$ 
generated by the Voronoi cells 
\begin{equation} \label{quq}
\varphi(v_1,...,v_k, w) \quad \mbox{where ${\rm span}_{\Z}(v_1,...,v_k) = {\rm L_3}$}. 
\end{equation}
\begin{lemma} \label{101}
The subcomplex $C({\rm L_3}, w)$ is acyclic.
\end{lemma}

\begin{proof} If the lattice spanned by the  vectors $v_1,...,v_k$ has rank $\leq 3$ 
then $\varphi(v_1,...,v_k)$ is not an interior cell, so 
 its contribution to the boundary of $\varphi(v_1,...,v_k, w)$ is zero.  
 One checks using the list of equivalence classes of the Voronoi cells that if 
$\varphi(v_1,...,v_k, w)$ is a Voronoi cell
for ${\rm L_4}$ and ${\rm L_3}:=  {\rm span}_{\Z}(v_1,...,v_k)$ is of rank $3$ then 
$\varphi(v_1,...,v_k)$ is a Voronoi cell for ${\rm L_3}$. 
Therefore the subcomplex $C'({\rm L_3},w)$ of the Voronoi complex generated by the cells 
$\varphi(v_1,...,v_k, w)$ where ${\rm span}_{\Z}(v_1,...,v_k) = {\rm L_3}$ is canonically 
isomorphic to the Voronoi 
complex for the lattice ${\rm L_3}$. The image of this complex in 
${\rm Coker}~\widetilde \psi_{\bullet}^{(4)}$ coincides, by the very definitions, 
with the complex ${\rm Coker}~\widetilde \psi_{\bullet}^{(3)}$ for the lattice ${\rm L_3}$, 
which is acyclic by Corollary \ref{kker}. \end{proof} 

\bl \label{c-acyca}
 The subcomplex of ${\rm coker}~ \widetilde \psi_\bullet^{(4)}$ generated by cells of type (\ref{c-cells}) is isomorphic to 
 \be \la{sc}
 \bigoplus_{({\rm L_3},w)}C({\rm L_3},w).
 \ee
\el

Lemma \ref{c-acyca} is proved in the Appendix as Lemma \ref{c-acyc}. 

 \vskip 2mm
Denote by $\overline {\rm Coker}~\widetilde \psi_{\bullet}^{(4)}$ 
the quotient of the complex ${\rm Coker}\widetilde \psi_{\bullet}^{(4)}$ 
by  subcomplexes (\ref{sc}). 

We denote by $\overline \varphi (\ast)$ the image of the cell 
$\varphi (\ast)$ 
in $\overline {\rm Coker}~\widetilde \psi_{\bullet}^{(4)}$, and by 
$\overline d$ the differential.

\begin{lemma} \label{105}
$$
\overline d: \overline \varphi(e_1,e_2,e_3,e_4, f_{12},f_{35}, f_{45}) \lms 
\overline \varphi(e_1, e_2, e_3, e_4, f_{35}, f_{45}).
$$
\end{lemma}

\begin{proof} The differential $\overline d \overline \varphi(e_1,e_2,e_3,e_4, f_{12},f_{35}, f_{45})$ 
is equal to
\be
\begin{split}
&\overline \varphi(e_2, e_3, e_4, f_{12},f_{35}, f_{45}) - 
\overline \varphi(e_1, e_3, e_4, f_{12},f_{35}, f_{45}) + 
\overline \varphi(e_1, e_2, e_4, f_{12},f_{35}, f_{45}) - \\
&\overline \varphi(e_1, e_2, e_3, f_{12},f_{35}, f_{45}) + 
\overline \varphi(e_1,e_2,e_3,e_4, f_{35}, f_{45}) -\\
&\overline \varphi(e_1,e_2,e_3,e_4, f_{12}, f_{45}) +
\overline \varphi(e_1,e_2,e_3,e_4, f_{12},f_{35}). \\
\end{split}
\ee

All  summands except $\overline \varphi(e_1,e_2,e_3,e_4, f_{35}, f_{45})$ are zero 
in $\overline {\rm Coker}~ \widetilde \psi_{\bullet}^{(4)}$ 
since all of them are ${\rm GL}_4(\Z)$-equivalent  to the cell 
$\varphi(e_1, e_2, e_3, f_{12}, f_{45}, e_4)$. Indeed, the vectors 
$e_1, e_2, e_3, f_{12}, f_{45}$ span the lattice ${\rm span}_{\Z}(e_1, e_2, e_3)$ and 
$e_4$ does not belong to this lattice. On the other hand:

For the first two terms: $e_3, e_4, f_{12},f_{35}, f_{45}$ span the lattice 
${\rm span}_{\Z}(e_3, e_4, e_5)$. 

For the third term:  $e_1, e_2, e_4, f_{12},f_{35}, $ span the lattice 
${\rm span}_{\Z}(e_1, e_2,  e_4)$.

For the fourth and sixth terms: $e_1, e_2, e_3, f_{12},f_{45}, $ span the lattice 
${\rm span}_{\Z}(e_1, e_2, e_3)$. 

For the seventh term: $e_1, e_2, e_3, f_{12},f_{35}, $ span the lattice 
${\rm span}_{\Z}(e_1, e_2, e_3)$. \end{proof}

Notice that 
$$
\overline \varphi(e_1,e_2,e_3,e_4, f_{35}, f_{45}) \sim 
\overline \varphi(e_1, e_2, e_3, e_4, f_{45}, f_{51}).
$$
So Lemma \ref{105} tells us that the   differential $\overline d $ maps a 6-cell of type iv) to a 5-cell of type ii). 
\vskip 3mm

Let us summarize what has been already done. Denote by 
$\overline {\rm Cok} ~\widetilde \psi^{(4)}_{\bullet}$ the quotient of the complex 
$\overline {\rm Coker} ~\widetilde \psi^{(4)}_{\bullet}$ by the subcomplex 
generated by the cells ${\rm GL}_4(\Z)$-equivalent to 
\begin{equation} \label{106a}
\overline \varphi(e_1, e_2, e_3, e_4, f_{35}, f_{45}), \quad 
\overline \varphi(e_1, e_2, e_3, e_4, f_{12}, f_{35}, f_{45}).
\end{equation} 
This subcomplex is acyclic by Lemma \ref{105}, so 
$\overline {\rm Cok}~ \widetilde \psi^{(4)}_{\bullet}$ is 
quasiisomorphic to $\overline {\rm Coker} ~\widetilde \psi^{(4)}_{\bullet}$. 

\vskip 2mm
The complex $\overline {\rm Cok}~ \widetilde \psi^{(4)}_{\bullet}$ 
sits in the degrees $[4,6]$. We  
describe it below  by generators and relations, 
listing   representatives for each ${\rm GL}_4(\Z)$-equivalence 
class.

{\it degree $4$}: The group $\overline {\rm Cok} ~\widetilde \psi^{(4)}_{4}$ is 
generated by the elements 
\begin{equation} \label{3.25.01.10}
\overline \varphi(e_1, e_2, e_3, e_4, e_5).
\end{equation} 

{\it degree $5$}: The group $\overline {\rm Cok} ~\widetilde \psi^{(4)}_{5}$ 
is 
generated by  the elements 
\begin{equation} \label{3.25.01.11}
\overline \varphi(e_1, e_2, e_3, e_4, e_5, f_{12}).
\end{equation} 

{\it degree $6$}: The group $\overline {\rm Cok} ~\widetilde \psi^{(4)}_{6}$ 
is generated by  the elements:
\begin{equation} \label{106c}
\overline \varphi(e_1, e_2, e_3, e_4, e_5, f_{12}, f_{34}), \quad 
\end{equation} 
\begin{equation} \label{106b}
\overline \varphi(e_1, e_2, e_3, e_4, e_5, f_{34}, f_{45}).
\end{equation} 
The relations are: 
\begin{equation} \label{1116b}
{\rm Cycle}_5 \overline \varphi(e_1, e_2, e_3, e_4, e_5, f_{12}, f_{34}), \quad 
\end{equation} 
\begin{equation} \label{116b}
\overline d\overline \varphi(e_1, e_2, e_3, e_4, e_5, f_{12}, f_{34}, f_{45}), \quad 
\overline d\overline \varphi(e_1, e_2, e_3, e_4, e_5, f_{12}, f_{13}, f_{34}).
\end{equation} 
Indeed,  element (\ref{1116b}) is $\widehat \psi^{(4)}_6(\langle e_1, ..., e_5\rangle)$,
and  elements (\ref{116b}) are boundaries of   $7$-cells i) and ii). 

Observe that the subset $\{e_1, e_2, e_3, e_4, e_5\}$ of the set of 
vectors $\{e_1, e_2, e_3, e_4, e_5, f_{12}\}$ is 
determined uniquely as the subset of $5$ vectors whose sum is zero. 

\begin{proposition} \label{3.25.01.14} The complex 
$\overline {\rm Cok} ~\widetilde \psi^{(4)}_{\bullet}$ is acyclic in the degrees $4$ and $5$.
\end{proposition}

\begin{proof} In the degree $4$ it follows from

\begin{lemma} \label{3.25.01.13} In the complex 
$\overline {\rm Cok} ~\widetilde \psi^{(4)}_{\bullet}$ one has 
$$
\overline d: 
\overline  \varphi(e_1, e_2, e_3, e_4, e_5, f_{12}) \lms 
\overline \varphi(e_1, e_2, e_3, e_4, e_5).
$$
\end{lemma}

\begin{proof} One has $\overline 
 \varphi(e_1, ... \widehat e_i, ...,  e_5, f_{12})=0$. Indeed, if $i=1$ or $i=2$ then it belongs to $C({\rm L_3}, w)$ where 
${\rm L_3} = {\rm span}_{\Z}(e_3, e_4, e_5, f_{12})$. Observe that $f_{12} = - f_{3 4 5}$. 
If $i=3, 4, 5$ then it is equivalent to a 4-cell of type iii), and so it is 
zero already in 
$\overline {\rm Coker} ~\widetilde \psi^{(4)}_{\bullet}$. \end{proof}

  To handle 
the degree $5$ case of Proposition \ref{3.25.01.14} we need the following Lemma.

\begin{lemma} \label{3.25.01.17} The subgroup of $5$-cycles in 
$\overline {\rm Cok} ~\widetilde \psi^{(4)}_{6}$ 
is generated by the cycles ${\rm GL}_4(\Z)$-equivalent to 
 \begin{equation} \label{3.25.01.1}
\overline  \varphi(e_1, e_2, e_3, e_4, e_5, f_{12}) - 
\overline  \varphi(e_1, e_2, e_3, e_4, e_5, f_{34}).
\end{equation}
\end{lemma}

\begin{proof} Any $5$-cycle in the complex 
$\overline {\rm Cok} ~\widetilde \psi^{(4)}_{\bullet}$ 
is ${\rm GL}_4(\Z)$-equivalent to   cycle (\ref{3.25.01.1}) or 
$$
\overline  \varphi(e_1, e_2, e_3, e_4, e_5, f_{12}) - 
\overline  \varphi(e_1, e_2, e_3, e_4, e_5, f_{23}),  
$$
which is a sum of   cycle (\ref{3.25.01.1}) and the two cycles obtained from it 
by action of the  cyclic shift $i \lms i+2$ mod $5$, and its iteration. 
\end{proof}

 Returning to the proof of Proposition \ref{3.25.01.14}, we note that one has
$$
\overline d:  \overline  \varphi(e_1, e_2, e_3, e_4, e_5, f_{12}, f_{34}) 
\lms 
\overline  \varphi(e_1, e_2, e_3, e_4, e_5, f_{12}) - 
\overline  \varphi(e_1, e_2, e_3, e_4, e_5, f_{34}).
$$
Indeed, each of the terms 
$\overline  \varphi(e_1, ... \widehat e_i, ... , e_5, f_{12}, f_{34})$ 
is zero in $\overline {\rm Cok} ~\widetilde \psi^{(4)}_{\bullet}$ since  
it is a $5$-cell of type iv) when $i=5$, and  of type iii) when $i=1,2,3,4$.
\end{proof}

\paragraph{7. End of the proof.} A   short  proof by N. Malkin can be found in the Appendix. 
We present below the original argument, since it gives some extra information on the story. 

\begin{proposition} \label{a106} a) The group $\overline {\rm Cok} ~\widetilde \psi^{(4)}_{6}$ 
is generated by the elements (\ref{106c}). 

b) The subgroup of relations is generated by the elements (\ref{1116b}) and 
\be \label{m9}
\begin{split}  
&\overline d \Bigl(\overline \varphi(e_1, e_2, e_3, e_4, e_5, f_{12}, f_{13}, f_{34}) -
\overline \varphi(e_2, e_5, e_1, e_3, e_4, f_{25}, f_{13}, f_{34}) \\
&+ \overline \varphi(e_4, e_5, e_2, e_1, e_3, f_{45}, f_{12}, f_{13})  
+ \overline \varphi(e_2, f_{13}, - e_3, f_{34}, e_5, -f_{45}, e_4, - f_{12})  \\
&- \overline \varphi(e_4, f_{13}, - e_1, f_{12}, e_5, -f_{25}, e_2, -f_{34})\Bigr).\\
\end{split}
\end{equation}
\end{proposition}

\begin{proof} We need the following lemma.

\begin{lemma} \label{106} In $\overline {\rm Cok} ~\widetilde \psi^{(4)}_{6}$ one has
\begin{equation} \label{2106}
\overline d \overline \varphi(e_1, e_2, e_3, e_4, e_5, f_{12}, f_{34}, f_{45}) =
\end{equation}
\begin{equation} \label{106d}
 - \overline \varphi(e_1, e_2, e_3, e_4, e_5, f_{34}, f_{45}) + 
\end{equation}
\begin{equation} \label{106e}
\overline \varphi(e_1, e_2, e_4,  e_5, f_{12}, f_{34}, f_{45}) +  
\overline \varphi(e_1, e_2, e_3, e_4, f_{12}, f_{34}, f_{45}) 
\end{equation}
\begin{equation} \label{106f}
- \overline \varphi(e_1, e_2, e_3, e_4, e_5, f_{45}, f_{12}) 
- \overline \varphi(e_1, e_2, e_3, e_4, e_5, f_{12}, f_{34}).  
\end{equation}
\end{lemma}

\begin{proof} The five terms above enter to the formula for $\overline d$. The remaining three terms in the formula for $\overline d$ are 
\begin{equation} \label{2107}
\overline \varphi(e_2, e_3, e_4, e_5, f_{12},f_{34}, f_{45}) - 
\overline \varphi(e_1, e_3, e_4, e_5, f_{12},f_{34}, f_{45}) + 
\overline \varphi(e_1, e_2, e_3, e_5, f_{12},f_{34}, f_{45}).  
\end{equation}
The first two of them are ${\rm GL}_4(\Z)$-equivalent to the element 
$\overline \varphi(e_1, e_2, e_3, e_4, f_{12}, f_{13}, f_{45})$, 
and thus they are zero already in $\overline {\rm Coker} ~\widetilde \psi^{(4)}_{\bullet}$ 
by (\ref{1form}). Indeed,  for the first one 
 use the transformation
$(e_1, e_2, e_3, e_4, e_5) \lms (e_4, e_3, e_5, e_2, e_1) $. The second  is 
  similar. 
Flipping $e_4 \leftrightarrow e_5$ we see that the last term in (\ref{2107}) 
 is ${\rm GL}_4(\Z)$-equivalent to $\varphi(e_1, e_2, e_3, e_4, f_{12}, 
f_{35}, f_{45})$, which is in (\ref{106a}), and so has been killed. \end{proof}

\paragraph{\bf Proof of Proposition \ref{a106}.}  a) Let us say that a generator of $\overline {\rm Cok} ~\widetilde \psi^{(4)}_{6}$ is of type (\ref{106c}) or (\ref{106b})
 if it is ${\rm GL}_4(\Z)$-equivalent to element (\ref{106c}) or (\ref{106b}). 

Let us study  element (\ref{2106}). The term 
(\ref{106d}) is obviously of type (\ref{106b}). The terms 
(\ref{106e})-(\ref{106f}) are of type (\ref{106c}). This is clear for (\ref{106f}); for the 
terms (\ref{106e}) use the transformation
\be \nonumber
\begin{split}
&(e_1, e_2, e_3, e_4, e_5, f_{12}, f_{34}) \lms (e_1, e_2, -e_4, f_{45}, f_{34}, f_{12},  e_5),\\
&(e_1, e_2, e_3, e_4, e_5, f_{12}, f_{34}) \lms (e_1, e_2, -e_4, f_{34}, f_{45}, f_{12},  e_3).\\
\end{split}
\ee
The part a)   is proved.

b) Now let us investigate the second element in (\ref{116b}). We start from the following   observation

\begin{lemma} \label{3.26.01.1}
Elements of type (\ref{106d}) are in 1-1 correspondence with the  
elements of type (\ref{2106}). Namely, an element  (\ref{106d}) determines
{\it uniquely} an element of type (\ref{2106})
which contains  (\ref{106d}). 
\end{lemma}

\begin{proof} The element (\ref{106d}) is the unique element of 
type (\ref{106b}) cotained in 
(\ref{2106}). 

To prove the converse statement we need to show that one can recover 
vector $f_{12}$, up to a sign,  from the unordered configuration of the vectors 
\begin{equation} \label{3.21.01.8}
\{\pm e_1, ..., \pm e_5, \pm f_{34}, \pm f_{45}\}.
\end{equation} 
Among these vectors 
 there are exactly four $3$-tuples of vectors which sum to zero: 
$(e_3, e_4, f_{34})$, $(e_4, e_5, f_{45})$, and 
another two obtained by reversing the signes. 
The two vectors which has not been used, considered up to a sign,  
are $\pm e_1$ and $\pm e_2$. Notice that $e_1$ and $-e_2$ can 
not be part of an extended basis formed by the $5$ of the vectors (\ref{3.21.01.8}). 
\end{proof}

One has 
$$
\overline d \overline \varphi(e_1, e_2, e_3, e_4, e_5, f_{12}, f_{13}, f_{34}) 
= \quad 
$$
\be \label{126a}
\begin{split}  
&- \overline \varphi(e_1, e_2, e_3, e_4, e_5, f_{13}, f_{34}) 
-\overline \varphi(e_1, e_2, e_3,  e_4, e_5, f_{12}, f_{13}),  \\
&+\overline \varphi(e_2, e_3, e_4, e_5, f_{12}, f_{13}, f_{34}) 
+\overline \varphi(e_1, e_2, e_4, e_5, f_{12}, f_{13}, f_{34}).\\
\end{split}
\end{equation}

\be \label{126c}
\begin{split} 
&- \overline \varphi(e_1, e_3, e_4, e_5, f_{12}, f_{13}, f_{34})
 - \overline \varphi(e_1, e_2, e_3, e_5, f_{12}, f_{13}, f_{34}) \\
&+\overline \varphi(e_1, e_2, e_3, e_4, f_{12}, f_{13}, f_{34}) +
 \overline \varphi(e_1, e_2, e_3, e_4, e_5, f_{12}, f_{34}).\\
\end{split}\end{equation}

We claim that the first four  terms (\ref{126a})  in this formula are of type 
(\ref{106b}). Indeed, for the first two 
terms in (\ref{126a}) this is clear. For the second two terms in (\ref{126a}) we   use   transformations
\be \nonumber
\begin{split}
&(e_1, e_2, e_3, e_4, e_5, f_{34}, f_{45}) \lms (e_2, f_{13},-e_3, f_{34},   e_5, e_4, -f_{12}),\\
&(e_1, e_2, e_3, e_4, e_5, f_{34}, f_{45}) \lms (e_4, f_{13},-e_1, f_{12},   e_5, e_2, -f_{34}).\\
\end{split}
\ee
The rest of the terms (\ref{126c})  are of type  (\ref{106c}). This is clear for the last of them. For the others use the transformations
\be \nonumber
\begin{split}
&(e_1, e_2, e_3, e_4, e_5, f_{12}, f_{34}) \lms (-e_1, f_{13},e_5, f_{12},   e_4, e_3, -f_{34}),\\
&(e_1, e_2, e_3, e_4, e_5, f_{12}, f_{34}) \lms (e_5, f_{34}, f_{13},   -e_3, e_2, -f_{12}, e_1),\\
&(e_1, e_2, e_3, e_4, e_5, f_{12}, f_{34}) \lms (-e_2, f_{12},   -e_4, f_{34}, -f_{13}, e_1,  e_3).\\
\end{split}
\ee

It follows from this that 
  element (\ref{m9})  is a sum of terms of type (\ref{106c}) 
only. 

It follows from this and Lemma \ref{3.26.01.1} that if 
a linear combination of elements (\ref{116b}) consists of elements of type 
(\ref{106c}) only, then it is a sum of elements (\ref{m9})   only. 
The part b)   is proved. 
\end{proof}

The last $8$ elements in the formula in Section 6.5 vanish after projection to 
$\overline {\rm Coker} \widetilde \varphi^{(4)}_{\bullet}$. 

Under the transformation $(e_1, e_2, e_3, e_4, e_5) \lms (e_1, e_2, e_5, e_4, e_3)$  
elements (\ref{m6})  go to (\ref{m9}). 

It follows from the definition of the subgroup $d{\cal M}^{(4)}_7$  
that the image of $\widetilde \psi_6^{(4)}(d{\cal M}^{(4)}_7)$ in 
$\overline {\rm Cok} ~\widetilde \psi_6^{(4)}$ coincides with the 
subgroup spanned by elements (\ref{m9}).  On the 
other hand we just proved that this subgroup is precisely the image 
of the  second shuffle relations of type $(1,3)$. 
Applying Lemma \ref{4more} we conclude that in fact it coincides 
with the image 
of   all second shuffle relations.

 \section{Applications}

 \paragraph{The   reduced  Lie coalgebra ${\cal D}'_{\bullet, \bullet}(\G)$.}  Recall the element $\{1,1\}_{1,1} \in {\cal D}_{1,1}(\G)$. Let us write
$$
  {\cal D}_{\bullet, \bullet}(\G) =  {\cal D}'_{\bullet, \bullet}(\G) \oplus \{1,1\}_{1,1}\cdot \Q.
$$
  It was proved in \cite{G00a} that 
$ 
  \delta {\cal D}_{\bullet, \bullet}(\G) \subset \Lambda^2   {\cal D}'_{\bullet, \bullet}(\G). 
$ 
So ${\cal D}'_{\bullet, \bullet}(\G)$ is a bigraded Lie coalgebra.

 \paragraph{The diagonal  Lie coalgebra.}    The following subspace of  ${\cal D}_{\bullet, \bullet}(\G)$ is evidently closed under the cobracket, and thus is a graded Lie subcoalgebra: 
 $$
 {\cal D}_\bullet(\G) := \bigoplus_{w=1}^\infty  {\cal D}_w(\G), ~~~~{\cal D}_w(\G):= {\cal D}_{w,w}(\G).
 $$
 Therefore we get a graded     Lie coalgebra, called the diagonal (dihedral)  Lie coalgebra.
 
  \vskip 2mm 
  There is a codimension one Lie subcoalgebra ${\cal D}'_\bullet(\G) \subset  {\cal D}_\bullet(\G) $. 
  
  In particular, for $\G = \mu_p$ we have
$$
{\cal D}_\bullet(\mu_p) =  {\cal D}'_\bullet(\mu_p) \oplus \{1,1\}_{1,1}\cdot \Q, ~~~~\{1,1\}_{1,1}\in {\cal D}_1(\mu_p).
$$
The element $\{1,1\}_{1,1}\in {\cal D}_1(\mu_p)$ corresponds, in a sense,  to the exponential motive provided by the Euler $\gamma-$constant, see \cite[Section 7.5]{G00a}. 

 \paragraph{The  unramified Lie coalgebra ${\cal D}^\circ_\bullet(\mu_p)$.}   There is a  map 
\be \nonumber
\begin{split}
&v: {\cal D}'_1(\mu_p) \lra \Q\\
&\{1, \zeta\}_{1,1} \lms 1~~ \forall\zeta \in \mu_p - \{1\}.
\end{split}
\ee
 Its kernel  is denoted by ${\cal D}^\circ_1(\mu_p)$.  So we have a non-canonical splitting
 $$
  {\cal D}_1(\mu_p) =  {\cal D}^{\circ}_1(\mu_p) \oplus \Q^2. 
  $$
Denote by  $\Q[-1]$  a one-dimensional $\Q-$vector space in the degree $1$. So we arrived at
\bl \la{L3.5}
 There is a non-canonical splitting of graded spaces:
$$
 {\cal D}_\bullet(\mu_p) =  {\cal D}^{\circ}_\bullet(\mu_p) \oplus \Q[-1]^2.
$$
 \el

Lemma \ref{L3.5} immediately implies that
\be \label{BR4}
\Lambda^*_{(w)}( {\cal D}_\bullet(\mu_p)) = \Lambda^*_{(w)}( {\cal D}^\circ_\bullet(\mu_p)) \bigoplus  \Lambda^*_{(w-1)}({\cal D}^\circ_\bullet(\mu_p))\otimes \Q^2[-1]  \bigoplus  \Lambda^*_{(w-2)}( {\cal D}^\circ_\bullet(\mu_p))[-2].   
\ee

 \paragraph{Cohomology of the unramified  Lie coalgebra ${\cal D}^\circ_{\bullet}(\mu_p)$.}  We start with   the following    result.
\bt \la{CMR}
 \begin{equation} \label{BR3}
\begin{split}
& {\rm H}^i_{(4)}({{\cal D}}_{\bullet}(\mu_p)) =  
 {\rm H}^{i+2}  (\Gamma_1(4;p), \varepsilon).\\
& {\rm H}^i_{(3)}({{\cal D}}_{\bullet}(\mu_p)) =  
 {\rm H}^{i}  (\Gamma_1(3;p), \Q) ,  ~~~~ i>0.\\
& {\rm H}^i_{(2)}({{\cal D}}_{\bullet}(\mu_p)) =  
 {\rm H}^{i-1}  (\Gamma_1(2;p), \varepsilon).\\
\end{split}
 \end{equation}
 \et
 
 \begin{proof} Note that $ {\rm H}^i_{(w)}({{\cal D}}_{\bullet}(\mu_p)) =0$  unless $1 \leq i \leq w$.  
 If $1 \leq i \leq w$, where $w=2,3,4$, the claim (\ref{BR3}) follows from Theorem \ref{4.26.01.3}c) and Theorem \ref{MTTHHa}.  
 It follows from the second line in (\ref{EST}) that    
 $$
  {\rm H}^{> 6}  (\Gamma_1(4;p), \varepsilon)=0, ~~~~ {\rm H}^{>3}  (\Gamma_1(3;p), \Q) =0, ~~~~ {\rm H}^{>1}  (\Gamma_1(2;p), \varepsilon)=0.
 $$
 Finally, $ {\rm H}^{i}  (\Gamma_1(4;p), \varepsilon) =0$ for $i<3$. For $i=0$ this is clear. For $i=1$, see  \cite[Theorem 16]{BMS}. 
 For $i=2$ this was proved  in \cite{H19}. 
  \end{proof}

\bt \la{CMR1} The weight $2,3$ cohomology of the   unramified diagonal Lie coalgebra ${\cal D}^{{\circ}}_{\bullet}(\mu_p)$ are:
\begin{equation} \nonumber
\begin{split}
 & {\rm H}^i_{(2)}({{\cal D}}^{{\circ}}_{\bullet}(\mu_p)) =  
 {\rm H}^{i-1}_{\rm cusp}  (\Gamma_1(2;p), \varepsilon);\\
& {\rm H}^i_{(3)}({{\cal D}}^{{\circ}}_{\bullet}(\mu_p)) =  
 {\rm H}^{i}_{\rm cusp}  (\Gamma_1(3;p), \Q).\\
  \end{split}
  \end{equation}
 \et
 
 
 \begin{proof} 
Follows from  
(\ref{BR4}).
\end{proof}

  

\paragraph{Applications to   cyclotomic Lie algebras.}  Recall   the subspace of  cyclotomic units 
$$
{\cal C}^\circ_{1 }(\mu_p):= {\cal O}^*(\Z[\zeta_p])\otimes\Q\subset {\cal C}_{1 }(\mu_p) = {\cal O}^*({\rm S}_p)\otimes\Q.
$$  
There is a canonical isomorphism
\be \nonumber
\begin{split}
& \xi_{1}: {\cal D}^{\circ}_1(\mu_p) \stackrel{\sim}{\lra} {\cal C}_{1}^\circ(\mu_p),\\
&\{1, \zeta_1\}_{1,1} - \{1, \zeta_2\}_{1,1} \lms \frac{1 - \zeta_1}{1 - \zeta_2}.\\
\end{split}
\ee
   
 We use the notation ${\cal C}^\circ_{w }(\mu_p) = {\cal C}_{w }(\mu_p)$ for $w>1$.
 
  \bt \la{THw23} Let $p$ be a prime number. Then 

i) The map (\ref{DC}) induces isomorphisms of vector spaces
\be \nonumber
\begin{split}
&\xi_{2}:  {\cal D}_{2}(\mu_p) \lra  {\cal C}_{2}(\mu_p),\\
 &\xi_{3}:   {\cal D}_{3}(\mu_p) \lra  {\cal C}_{3}(\mu_p).\\
\end{split}
\ee

ii) Furthermore, it  induces   a canonical isomorphism of bigraded Lie coalgebras
$$
\bigoplus_{w=1}^3{\cal D}^\circ_{w}(\mu_p) \lra \bigoplus_{w=1}^3{\cal C}^\circ_{w }(\mu_p).
$$
\et
 
 \begin{proof} i) There is a map of complexes of amplitude $[1,2]$: 
 \begin{displaymath} 
    \xymatrix{
         {\cal D}_{2 }(\mu_p) \ar[r]^{\delta~~~} \ar[d]^{\xi_{2}} &    \bigwedge^2 {\cal D}^\circ_{1 }(\mu_p)\ar[d]_{=}^{\wedge^2\xi_{1 } }\\     
     {\cal C}_{2 }(\mu_p) \ar[r]^{\delta~~~}  &   \bigwedge^2 {\cal C}^\circ_{1 }(\mu_p) \\}
\end{displaymath}
The degree $2$ component of this map is  an isomorphism. The $ {\rm H}^1$  of the top complex 
is equal to $ {\rm H}^1(\Gamma_1(2;p); \varepsilon_2)$ by Theorem \ref{4.26.01.3}c) and Theorem \ref{MTTHHa}. Since this group is zero,  the map $\xi_{2}$ is injective. 
On the other hand, the map $\xi_{2}$ is surjective by Theorem \ref{SMCO}. So it is an isomorphism. 

\vskip 2mm

There is a map of complexes of amplitude $[1,3]$:
 \begin{displaymath} 
    \xymatrix{
         {\cal D}_{3 }(\mu_p) \ar[r]^{\delta~~~~~~~~} \ar[d]^{\xi_{3 }} 
         &  {\cal D}_{2 }\otimes  {\cal D}^\circ_{1 }(\mu_p) \ar[d]_=^{\xi_{2 }\otimes\xi_{1 }}\ar[r]^{~~~~\delta}&   \bigwedge^3 {\cal D}^\circ_{1 }(\mu_p)\ar[d]_=^{\wedge^3\xi _{1 }}\\     
     {\cal C}_{3 }(\mu_p) \ar[r]^{\delta~~~~~~~~}  &  {\cal C}_{2 } \otimes  {\cal C}^\circ_{1 }(\mu_p)  \ar[r]^{~~~~\delta}&  \bigwedge^3 {\cal C}^\circ_{1 }(\mu_p) \\}
\end{displaymath}
Its degree  $2$ and $3$ components  are isomorphisms. The $ {\rm H}^1$  of the top complex 
is equal to $ {\rm H}^1(\Gamma_1(3;p); \Q)$ by Theorem \ref{4.26.01.3}c) and Theorem \ref{MTTHHa}. Since 
${\rm H}^1(\Gamma_1(3;p), \Q)=0$ by Kazhdan's theorem,  the map $\xi_{3}$ is injective. 
It is surjective by Theorem \ref{SMCO}. So it is an isomorphism. 

ii) Follows from i) and Theorem \ref{SMCO}.\end{proof}

\bc All relations between the motivic multiple polylogarithms of weight = depth $\leq 3$ at $p-$th roots of unity follow from  the double shuffle relations.
\ec


\bt \la{CMR1x} The weight $2, 3$ cohomology of the   unramified cyclotomic Lie coalgebra ${\cal C}^{{\circ}}_{\bullet}(\mu_p)$ are:
\begin{equation} \nonumber
\begin{split}
 & {\rm H}^i_{(2)}({{\cal C}}^{{\circ}}_{\bullet}(\mu_p)) =  
 {\rm H}^{i-1}_{\rm cusp}  (\Gamma_1(2;p), \varepsilon),\\
& {\rm H}^i_{(3)}({{\cal C}}^{{\circ}}_{\bullet}(\mu_p)) =  
 {\rm H}^{i}_{\rm cusp}  (\Gamma_1(3;p), \Q).\\
  \end{split}
  \end{equation}
 \et
 
 \begin{proof} It follows from Theorems \ref{CMR1} and \ref{THw23}.
 \end{proof}

\paragraph{In the depth $\leq 4$, the double shuffle relations are all relations.} According to Horozov, \cite[Theorem 1.1]{H06},  we have: 
\begin{equation} \label{4.8.01.4} 
  {\rm H}^i({\rm GL}_4(\Z), S^{w-4}V_4 \otimes \varepsilon) = \quad \left\{ \begin{array}{ll}
  {\rm H}_{\rm cusp}^1({\rm GL}_2(\Z), S^{w-2}V_2 \otimes \varepsilon) \bigoplus \Q& i=3,\\ 
0 & \mbox{\rm otherwise}.\\
 \end{array}\right.
\end{equation}
 
 I   suspect that the extra summand $\Q$ in $ {\rm H}^3$ is an error. As shown in the Remark below, the main result of this paper  for $w=4$ is compatible with the following result:
 \begin{equation} \label{4.8.01.4A} 
  {\rm H}^i({\rm GL}_4(\Z), S^{w-4}V_4 \otimes \varepsilon) = \quad \left\{ \begin{array}{ll}
  {\rm H}_{\rm cusp}^1({\rm GL}_2(\Z), S^{w-2}V_2 \otimes \varepsilon)& i=3,\\ 
0 & \mbox{\rm otherwise}.\\
 \end{array}\right.
\end{equation}

  \bt \la{TTHH5.7} i) All relations between  motivic multiple $\zeta-$values  in the depth $\leq 3$ follow from the regularised motivic double shuffle relations
  
  ii)  Assuming Horozov's theorem in the corrected form (\ref{4.8.01.4A}), all relations between  motivic multiple $\zeta-$values  in the depth $4$ follow from the regularised motivic double shuffle relations. 
 \et

 \begin{proof} i) Recall the natural surjective map of Lie coalgebras $ {\cal D}_{\bullet, \bullet}  \lra {\cal C}_{\bullet, \bullet}$.   Consider the  depth $m$  part of the map of the standard complexes  induced by this map.  It is a surjective map. 
For example,  
 for the depth $4$ we get a map of complexes of amplitude $[1,4]$: 
 \begin{displaymath} 
     \xymatrix{
          {\cal D}_{\bullet, 4}  \ar[r]^{\delta~~~~~~~~~~ ~} \ar[d]  
        &  {\cal D}_{\bullet, 3} \otimes  {\cal D}_{\bullet, 1 }  \bigoplus \bigwedge^2 {\cal D}_{\bullet,  2 }  \ar[d] \ar[r]^{~~~~\delta}
          &  {\cal D}_{\bullet, 2 }  \otimes \bigwedge^2 {\cal D}_{\bullet, 1} \ar[d] \ar[r]^{~~~~~~\delta}&  
           \bigwedge^4 {\cal D}_{\bullet, 1}\ar[d] \\     
      {\cal C}_{\bullet, 4}  \ar[r]^{\delta~~~~~~~ ~~~~~}  &  {\cal C}_{\bullet, 3}  \otimes  {\cal C}_{\bullet,  1}   \bigoplus \bigwedge^2 {\cal C}_{\bullet, 2 } \ar[r]^{~~~~\delta}
      & {\cal C}_{\bullet, 2 }  \otimes \bigwedge^2 {\cal C}_{\bullet, 1 } \ar[r]^{~~~~~~\delta}&  \bigwedge^4 {\cal C}_{\bullet, 1 }  \\}
 \end{displaymath}
 
   Arguing by the induction on the depth, just like in  Theorem \ref{THw23},    the  degree $\geq 2$ components of the depth $m$ map are isomorphisms. 
 In the depth $\leq 3$ the ${\rm H}^1$  of the  top complex is zero. Therefore the  left map must be an isomorphism. 
 
 ii) It follows from i) that the degree $\geq 2$ components of the depth $4$ map above are isomorphisms. 
 The ${\rm H}^1$    of the top complex 
 is equal to $ {\rm H}_{\rm cusp}^1({\rm GL}_2(\Z), S^{\bullet-2}V_2 \otimes \varepsilon)$ by  (\ref{4.8.01.4A}) and Theorem ??.    

 On the other hand, let us analyze 
\be \la{234}
 {\rm Ker}_{(w,4)}\Bigl(\delta: {\cal C}_{\bullet,4} \lra {\cal C}_{\bullet,3} \wedge {\cal C}_{\bullet,1} \bigoplus \Lambda^2 {\cal C}_{\bullet, 2}\Bigr).
  \ee  

Denote by ${\cal G}_{\bullet}$   the graded dual of the graded Lie algebra  $\varphi_N ({\rm L}_\bullet({\rm S_N}))$ for $N=1$. Recall that ${\cal C}_{\bullet, \bullet}$ is the asscociate graded for the depth filtration ${\cal F}$ of ${\cal G}_{\bullet}$.  It is a cofree graded Lie coalgebra with the generators on degrees $3,5,7,9,...$ \cite{B11}. In particular, 
 \be \la{CCOOKK}
  {\rm Coker}_{\mbox {depth $\leq 2$}} \Bigl(\delta: {\cal G}_{\bullet} \lra \Lambda^2 {\cal G}_{\bullet}\Bigr) =0.
   \ee

 Theorem \ref{4.26.01.3} or \cite[Theorem 1.2]{G98} imply that for the $(w,2)$ part of the cokernel of the coproduct  
 \be \la{233}
 {\rm Coker}_{(w,2)}\Bigl(\delta: {\cal C}_{\bullet,2} \lra \Lambda^2 {\cal C}_{\bullet, 1}\Bigr) = {\rm H}_{\rm cusp}^1({\rm GL}_2(\Z), S^{w-2}V_2 \otimes \varepsilon).
  \ee  
  Then (\ref{CCOOKK}) implies that we must have a subspace in the depth $>2$ part of ${\cal G}_{\bullet}$,   identified with the right hand side of (\ref{233}), 
   whose coproduct lies in $\Lambda^2 {\cal F}_{\leq 1}{\cal G}_{\bullet} $ and induces an isomorphism after the projection to the left hand side of (\ref{233}).  
    It was discovered by numerical calculations in \cite{BK96}, and proved in \cite{B13}, that  this   
     subspace appears in the  depth $4$ part of ${\cal G}_{\bullet}$.      This can be interpreted as the statement that 
 $$
     {\rm dim}~(\ref{234})\geq {\rm dim}~{\rm H}_{\rm cusp}^1({\rm GL}_2(\Z), S^{w-2}V_2 \otimes \varepsilon).
$$
 But the first result implies the opposite inequality: 
$$
     {\rm dim}~(\ref{234})\leq {\rm dim}~{\rm H}_{\rm cusp}^1({\rm GL}_2(\Z), S^{w-2}V_2 \otimes \varepsilon).
$$ 
 Combining the two results, we see that   ${\rm dim}~(\ref{234})$ coincides with the ${\rm H}^1$  of the weight $w$ part of the top complex. 
 This implies that the surjective map ${\cal D}_{\bullet, 4}  \lra {\cal C}_{\bullet, 4}$ must be an isomorphism, which means that the double shuffle relations are all relations between the motivic correlators in the depth $4$. 
    \end{proof}
    
\paragraph{Remark.} 
The extra summand $\Q$ in (\ref{4.8.01.4}), combined with Theorem \ref{MTTHHa} and \cite{B13}, would mean that there is a one new relation in the depth $4$. But this is incorrect, as direct calculations in the depth $m=4$ and the weights $w=4, 6$ show.

  \section{Generalizations} 

The   dg-algebra of higher modular symbols  has a  generalization. 
Namely, consider the symmetric space for the group ${\rm GL}_m(\C)$:
\be \la{CSS}
{\Bbb H}_{{\rm GL}_m(\C)}:= {\rm GL}_m(\C)/{\rm U}_m \cdot \R^*.
\ee
For example, ${\Bbb H}_{{\rm GL}_2(\C)}$ is the three-dimensional hyperbolic space. 

Take an imaginary quadratic field $K$ with the ring of integers ${\cal O}_K$. Let ${\rm L}_m$ be a rank $m$ free ${\cal O}_K-$module.  
 The symmetric space  (\ref{CSS}) is realized as the space of positive definite Hermitian forms on the complex vector space $V_m$ dual to 
${\rm L}_m\otimes \C$. Any vector $l$ of the lattice give rise to a degenerate Hermitian form $\varphi(l):= |(l, \ast)|^2$. The convex hull of these forms when $l$ run through all 
non-zero primitive vectors of ${\rm L}_m$ is an infinite polyhedron, invariant under the action of the group ${\rm GL}_m({\cal O}_K)$. 
It gives rise to  ${\rm GL}_m({\cal O}_K)-$equivariant  Voronoi cell decomposition of the symmetric space  (\ref{CSS}).  Repeating 
the sum of the plane trivalent tree construction, we get  $(2m-2)-$chains $\varphi(e_1, ... , e_m)$ in the symmetric space, 
assigned to bases $(e_1, ..., e_m)$ of ${\rm L}_m$. The convex hull = $\ast-$product is immediately generalised to the case of symmetric spaces (\ref{CSS}). 
Combining the two constructions, we get a length $m$ complex of higher modular symbols ${\cal M}{\cal S}_{{\rm L}_m}^*$.

Next, there is an analog of the rank $m$ modular complex ${\rm M}^*_{K; (m)}$, defined by    similar generators and  the double shuffle relations, plus 
the invariance under the action of the group of units ${\cal O}_K^*$. It is a complex of ${\rm Aut}({\rm L}_m) \stackrel{\sim}{=} {\rm GL}_m({\cal O}_K)-$modules. 

Given an ideal ${\cal N}\subset {\cal O}_K$, the congruence subgroup $\Gamma_1(m; {\cal N}) \subset  {\rm GL}_m({\cal O}_K)$ stabilizes  
the vector $(0, ..., 0,1)$ mod ${\cal N}$. So we can introduce a commutative dg-algebra\footnote{Here  an abuse of notation appears:  as usual, for $m=1$ the definition needs to be  adjusted.} 
$$
{\rm M}^*_{{\cal N}}:= \bigoplus_{w \geq m\geq 1}{\rm M}^*_{K; (m)}\otimes_{\Gamma_1(m; {\cal N})}S^{w-m}V_m.
$$
Omitting the second shuffle relations, we  get the relaxed modular dg-algebra $\widehat {\rm M}^*_{{\cal N}}$. 
Consider the level ${\cal N}$ dg-algebra of higher modular cycles
$$
{\cal M}{\cal S}^*_{{\cal N}}:= \bigoplus_{w \geq m\geq 1}{\cal M}{\cal S}^*_{K; (m)}\otimes_{\Gamma_1(m; {\cal N})}S^{w-m}V_m.
$$

The   sum over plane trivalent trees construction provides a map of dg-algebras: 
$$
\widehat {\rm M}^*_{{\cal N}} \lra {\cal M}{\cal S}^*_{{\cal N}}.
$$

On the other hand, given a CM-elliptic curve $E$ with ${\rm Aut}(E) = {\cal O}_K$, one can consider an analog  of the 
 dihedral Lie algebra ${\cal D}_{\bullet, \bullet}(E[{\cal N}])$ for the group $\G := E[{\cal N}]$ of  ${\cal N}-$torsion 
points of $E$. The formal parameters $t_i$ in its definition   
lie in  the rank one ${\cal O}_K-$module  $H_1(E)$. Then the  double shuffle relations have the same form. 
We also have the invariance under the action  of the unit group ${\cal O}^*_K$  on  $E[{\cal N}]$ and   $t_i$'s, and the distribution relations. 

The proof of  Theorem \ref{4.26.01.3ab} below follows  literally the proof of its cyclotomic analog. 

\begin{theorem} \label{4.26.01.3ab} 
Pick a generator $\zeta_{\cal N}$ of the ${\cal O}_K-$module  $E[{\cal N}]$. Then there is a canonical surjective map of  dg-algebras
$$
 \gamma: {\rm M}_{\cal N}^\ast ~~\lra ~~\Lambda^*\Bigl({\cal D}_{\bullet, \bullet}(E[{\cal N}])\Bigr).
$$
When the ideals ${\cal N}$ vary, we get a projective system of maps of dg-algebras. 
\end{theorem} 

The analog of $\pi_1^{\cal M}({\Bbb G}_m - \mu_N, v_0)$ is the motivic fundamental group $\pi_1^{\cal M}({E} - E[{\cal N}], v_{\rm av})$   with the averaged base point in the 
(almost) canonical tangent vectors at the ${\cal N}-$torsion points \cite{G08}.  It is so far a conjectural object, so has to be treated in realizations. 

The elliptic analog ${\cal C}_{\bullet, \bullet}({\cal N})$ of the cyclotomic Lie algebra is given by the associate graded for the depth filtration of the image of  the Galois group acting on $\pi_1^{\cal M}({E} - E[{\cal N}], v_{\rm av})$. It is 
  described by the elliptic motivic correlators at the ${\cal N}-$torsion points. 
We want to have a map of Lie coalgebras
$$
{\cal D}_{\bullet, \bullet}(E[{\cal N}]) \lra \widehat {\cal C}_{\bullet, \bullet}({\cal N}).
$$
The elliptic motivic correlators at the ${\cal N}-$torsion points provide such a map on the   generators. It satisfies  the dihedral and the first shuffle relation. 
This means that we get a map of dg-algebras
$$
\widehat \gamma_{\cal N}: \widehat {\rm M}_{\cal N}^\ast  \lra \Lambda^*\Bigl(\widehat {\cal C}_{\bullet, \bullet}({\cal N})\Bigr).
$$
The key problem is  to prove that this map descends to a map of the  dg-algebra ${\rm M}_{\cal N}^\ast $. 
This just means that one has to prove the "difficult" shuffle relations for   the elliptic  motivic correlators at   torsion points. 

For the  $w=m=2$ case,
  this program has been implemented  in \cite{G05} using  strong reciprocity laws  for elliptic curves rather than  the  elliptic motivic correlators. 
  Its most surprising consequence is that when ${\cal N}={\cal P}$ is a prime ideal in $\Z[i]$ or $\Z[\zeta_3]$, the structure of the Lie coalgebra 
  given by 
   \be \la{CP}
 {\cal C}_{2,2}({\cal P}) \oplus \widehat {\cal C}_{1,1}({\cal P})
   \ee
   is described by the geometry of the cell decomposition  of the Bianchi threefold ${\rm B}_{\cal P}:= {\cal H}_3/\Gamma(2; {\cal P})$. 
   It relates  the cohomology of the Lie coalgebra (\ref{CP}) to the cohomology of Bianchi threefold ${\rm B}_{\cal P}$. 
   
The full  depth $2$ elliptic story  has been recently developed by N. Malkin \cite{Ma19a}. 

\section{Appendix by Nikolay Malkin}

\paragraph{1.  End of the proof of   Theorem \ref{1more}a).} Namely, it remains to prove   the following result.

\begin{theorem} \la{8.1}  
The map $\widetilde \psi_6^{(4)}: {\cal M}_6^{(4)}/d{\cal M}_7^{(4)} \lra V_6^{(4)}/dV_7^{(4)}$ is injective.
\label{tps-inj}
\end{theorem}

Recall the map    $\widehat  \psi^{(4)}_6:\widehat {\rm M}_6^{(4)}\to V_6^{(4)}$ is  
  given on generators by 
$$
\langle e_1,e_2,e_3,e_4,e_5 \rangle \mapsto {\rm Cycle}_5 \varphi(e_1,e_2,e_3,e_4,e_5,f_{23},f_{45}).
$$ 
Lemma \ref{hps-inj} is the main computational part of the proof, establishing injectivity of $\widehat \psi_6^{(4)}$ on the $\widehat {\cal M}_6^{(4)}$ summand of $ {\cal M}_6^{(4)}$. The injectivity is extended to all of ${\cal M}_6^{(4)}$ in Lemma \ref{ops-inj}. The proof of Theorem \ref{8.1} is concluded by showing no nontrivial kernel is gained from passing to the quotients.

\bl
The map $\widehat \psi_6^{(4)}:\widehat {\cal M}_6^{(4)} \lra V_6^{(4)}$ is injective.
\label{hps-inj}
\el
\begin{proof}
The set of vectors $\{e_1,e_2,e_3,e_4,e_5\}$ can be uniquely recovered up to sign from the set of vertices of a summand $\{e_1,e_2,e_3,e_4,e_5,f_{23},f_{45}\}$, since they are those five elements that are part of only one set of three vectors generating a subspace of rank 2. In this way, both $\widehat {\cal M}_6^{(4)}$ and $V_6^{(4)}$ are decomposed as direct sums of components indexed by extended bases modulo sign, and $\widehat \psi_6^{(4)}$ respects this decomposition. Therefore, it suffices to check injectivity on the portion $P_6^{(4)}$ of $\widehat {\cal M}_6^{(4)}$ generated by extended bases consisting of $\pm e_1,\pm e_2,\pm e_3,\pm e_4,\pm e_5$ for fixed $e_1,e_2,e_3,e_4,e_5$, a finite-dimensional space. 

We first aim to find an upper bound on the dimension of $P_6^{(4)}$. Let $F_6^{(4)}$ be the  vector space with the basis given  by the extended bases consisting of only the $\pm e_i$, so that 
$$
P_6^{(4)}= F_6^{(4)}/\langle \mbox{dihedral, first shuffle relations} \rangle.
$$
 We claim that (the images of) the following six elements span $P_6^{(4)}$:
\begin{equation}
\{\langle e_{\sigma(1)},e_{\sigma(2)},e_{\sigma(3)},e_4,e_5\rangle:\sigma\in S_3\}.
\label{pgen}
\end{equation}
Observe that the $e_i$ in any extended basis containing only the $\pm e_i$ must be taken with the same sign. By the dihedral symmetry relations, the sign may be assumed positive. Further, it is easy to see that any extended basis is equivalent either one of the elements (\ref{pgen}) or its negative, or
\begin{equation} \label{pgen2}
\pm\{\langle e_{\sigma(1)},e_{\sigma(2)},e_4,e_{\sigma(3)},e_5\rangle:\sigma\in S_3\}.
\end{equation}
By the (1,3)-first shuffle relation, there is the following equality in $P_6^{(4)}$:
\begin{eqnarray*}
\langle e_{\sigma(3)},e_{\sigma(1)},e_{\sigma(2)},e_4,e_5\rangle+ & \\ 
\langle e_{\sigma(1)},e_{\sigma(3)},e_{\sigma(2)},e_4,e_5 \rangle+ & \\
\langle e_{\sigma(1)},e_{\sigma(2)},e_{\sigma(3)},e_4,e_5\rangle= &
-\langle e_{\sigma(1)},e_{\sigma(2)},e_4,e_{\sigma(3)},e_5\rangle.
\end{eqnarray*}
Thus an element of the form (\ref{pgen2}) can be expressed as a linear combination of elements (\ref{pgen}). This proves the claim and implies that $\dim P_6^{(4)}\leq6$.

Next, we compute the dimension of the image of the restriction of $\widehat \psi_6^{(4)}$ to $P_6^{(4)}$, which is contained in the subspace of $V_6^{(4)}$ generated by the cells  
$$
 \varphi(e_{\sigma(1)},e_{\sigma(2)},e_{\sigma(3)},e_{\sigma(4)},e_{\sigma(5)},f_{\sigma(2)\sigma(3)},f_{\sigma(4)\sigma(5)}),\;\sigma\in S_5.
$$
  The dimension of this subspace is $\frac{1}{2}\binom53\binom32=15$, with a basis indexed by the pairs of $(i,j)$ occurring as vertices $f_{ij}$. We can express the images under $\widehat \psi_6^{(4)}$ of the 6 elements (\ref{pgen}) in terms of these 15 cells, yielding a $6\times15$ matrix 
\be \nonumber
\begin{matrix}
&123&213&132&312&231&321\\
 \varphi(e_1,e_2,e_3,e_4,e_5,f_{12},f_{34}) &  1 & -1 &  0 &  0 &  0 &  0& \\ 
 \varphi(e_1,e_2,e_3,e_4,e_5,f_{12},f_{35}) &  0 &  0 &  0 & -1 &  0 &  1& \\ 
 \varphi(e_1,e_2,e_3,e_4,e_5,f_{12},f_{45}) & -1 &  1 &  0 &  1 &  0 & -1& \\
 \varphi(e_1,e_2,e_3,e_4,e_5,f_{13},f_{24}) &  0 &  0 & -1 &  1 &  0 &  0& \\ 
 \varphi(e_1,e_2,e_3,e_4,e_5,f_{13},f_{25}) &  0 &  1 &  0 &  0 & -1 &  0& \\ 
 \varphi(e_1,e_2,e_3,e_4,e_5,f_{13},f_{45}) &  0 & -1 &  1 & -1 &  1 &  0& \\ 
 \varphi(e_1,e_2,e_3,e_4,e_5,f_{14},f_{23}) &  0 &  0 &  0 &  0 & -1 &  1& \\ 
 \varphi(e_1,e_2,e_3,e_4,e_5,f_{14},f_{25}) &  0 &  0 &  0 &  0 &  1 &  0& (*)\\ 
 \varphi(e_1,e_2,e_3,e_4,e_5,f_{14},f_{35}) &  0 &  0 &  0 &  0 &  0 & -1& (*)\\ 
 \varphi(e_1,e_2,e_3,e_4,e_5,f_{15},f_{23}) &  1 &  0 & -1 &  0 &  0 &  0& \\ 
 \varphi(e_1,e_2,e_3,e_4,e_5,f_{15},f_{24}) &  0 &  0 &  1 &  0 &  0 &  0& (*)\\ 
 \varphi(e_1,e_2,e_3,e_4,e_5,f_{15},f_{34}) & -1 &  0 &  0 &  0 &  0 &  0& (*)\\ 
 \varphi(e_1,e_2,e_3,e_4,e_5,f_{23},f_{45}) &  1 &  0 & -1 &  0 & -1 &  1& \\
 \varphi(e_1,e_2,e_3,e_4,e_5,f_{24},f_{35}) &  0 &  0 &  0 &  1 &  0 &  0& (*)\\ 
 \varphi(e_1,e_2,e_3,e_4,e_5,f_{25},f_{34}) &  0 &  1 &  0 &  0 &  0 &  0& (*)
\end{matrix}
\ee
that has full rank (see the rows $(*)$), so $\dim\widehat \psi_6^{(4)}(P_6^{(4)})=6$. Together with $\dim P_6^{(4)}\leq 6$, this completes the proof that $\widehat \psi_6^{(4)}$ is injective.
\end{proof}

Recall from Section 6 the definition 
$$
{\cal M}_6^{(4)}=\widehat {\rm M}_6^{(4)}\oplus\bigoplus_{({\rm L_3},w)}{\cal M}_5^{(3)}({\rm L_3})\wedge[w].
$$ 
Here the sum is over decompositions ${\rm L_4}={\rm L_3}\oplus\{w\}$. We can extend the map $\widehat \psi_6^{(4)}:\widehat {\rm M}_6^{(4)}\to V_6^{(4)}$ to 
$\overline \psi_6^{(4)}:{\cal M}_6^{(4)}\to V_6^{(4)}$ by setting $\overline \psi_6^{(4)}(s\wedge[w])=\widetilde \psi_5^{(3)}(s)* \varphi(w)$ on the direct summand 
${\cal M}_5^{(3)}({\rm L_3})\wedge[w]$, where $\widetilde \psi_5^{(3)}$ is the map in Theorem 5.4 composed with the inclusion $V({\rm L_3})\hra V({\rm L_4})$ induced by ${\rm L_3}\hra {\rm L_4}$. Explicitly, 
\be \label{l3w-phi}
\begin{split}
\overline \psi_6^{(4)}(\langle e_1,e_2,e_3,f_{45}\rangle \wedge \langle e_4\rangle ) &=\varphi(e_1,e_2,e_3,f_{45},f_{12},f_{23})*\varphi(e_4)\nonumber\\
&=- \varphi(e_1,e_2,e_3,e_4,f_{12},f_{23},f_{45}).\\
\end{split}
\ee
Then there is a commutative diagram
\begin{equation}
\xymatrix{{\cal M}_6^{(4)}\ar[r]\ar[d]_{\overline \psi_6^{(4)}}&{\cal M}_6^{(4)}/d{\cal M}_7^{(4)}\ar[d]^{\widetilde \psi_6^{(4)}}\\ V_6^{(4)} \ar[r]&V_6^{(4)}/dV_7^{(4)}}.
\label{tilde-ol-square}
\end{equation}
Lemma \ref{hps-inj} can be strengthened to:

\bl
The map $\overline \psi_6^{(4)}: {\cal M}^{(4)}_6\lra V_6^{(4)}$ is injective.
\label{ops-inj}
\el
\begin{proof}
We see from (\ref{l3w-phi}) that 
\be \la{Mf}
\overline \psi_6^{(4)}\Bigl(\bigoplus_{({\rm L_3},w)}{\cal M}^{(3)}_6({\rm L_3})\wedge[w]\Bigr)
\ee coincides with the subspace of $V_6^{(4)}$ generated by the 6-cells of type iii). From a given 6-cell of type iii), ${\rm L_3}$ and $\langle w \rangle$ can be uniquely recovered. Indeed, in its set of vertices $\{e_1,e_2,e_3,e_4,f_{45},f_{12},f_{23}\}$, $e_4$ is the only vector that is not contained in a set of three vectors generating a sublattice of rank 2. Therefore, the $\overline \psi_6^{(4)} \Bigl({\cal M}^{(3)}_6({\rm L_3})\wedge[w]\Bigr)$ intersect trivially for distinct $({\rm L_3},w)$. As a consequence of Theorem 5.4a), $\overline \psi_6^{(4)}$ is injective on each ${\cal M}^{(4)}_6({\rm L_3})\wedge[w]$, so it is injective on their sum.

Furthermore, $ {\overline \psi_6^{(4)}}_{|\widehat {\rm M}^{(4)}_6}=\widehat \psi_6^{(4)}$ is injective by Lemma \ref{hps-inj}, and its image is contained in the subspace of $V_6^{(4)}$ generated by the 6-cells of type ii), which intersects trivially with the subspace (\ref{Mf}). It follows that $\overline \psi_6^{(4)}$ is injective on all ${\cal M}^{(4)}_6$.
\end{proof}

\begin{proof}[Proof of Theorem \ref{tps-inj}]
The left map in (\ref{tilde-ol-square}) is injective by Lemma \ref{ops-inj}. Therefore, it suffices to show that 
$$
\overline \psi_6^{(4)}\Bigl(d{\cal M}_7^{(4)}\Bigr)=dV_7^{(4)}\cap\overline \psi_6^{(4)}\Bigl(d{\cal M}_6^{(4)}\Bigr).
$$
Note that $\widetilde \psi_6^{(4)}$ being well-defined, shown in i), amounts to the inclusion ``$\subseteq$'', and it was computed in i) that the left side is generated by elements of the form 
$$
\text{(boundary of 7-cell of type ii))}+\text{(sum of $\pm$ boundaries of 7-cells of type i))}.
\label{six-one-two}
$$
So the quotient 
$$
\frac{dV_7^{(4)}\cap\overline \psi_6^{(4)}\Bigl(d{\cal M}_6^{(4)}\Bigr)}{\overline \psi_6^{(4)}\Bigl(d{\cal M}_7^{(4)}\Bigr)}
$$
 is generated by sums of boundaries of 7-cells of type i) only. On the other hand, by the definition of $\overline \psi_6^{(4)}$, it is generated by 6-cells of types ii) and iii) only.

The composition 
$$
d_1:V_7^{(4)}\stackrel{d}{\lra} dV_6^{(4)}\to V_6^{(4)}/(\text{6-cells of type ii),iii),iv)})
$$
 is an isomorphism on $\langle \text{7-cells of type i)} \rangle$, since the set of vertices of a 6-cell of type i), given by the cell $ \varphi(e_1,e_2,e_3,e_4,e_5,f_{34},f_{45})$,  uniquely determines $f_{12}$ up to sign (cf.\ Lemma \ref{3.26.01.1}). This shows that the above quotient is zero, and $\widetilde \psi_6^{(4)}$ is injective.
\end{proof}

\paragraph{2.  A shorter end of the proof of   Theorem \ref{1more}b).} Let us first restate the result. 

\bt \la{6.5b)}
The complex ${\rm Coker}~ \widetilde \psi_\bullet^{(4)}$ is acyclic.
\label{tps-qi}
\et

Recall that ${\rm L_4}={\rm L_3}\oplus  \langle w\rangle$, and  $C({\rm L_3},w)$ is the subcomplex of ${\rm Coker}~\widetilde \psi_\bullet^{(4)}$ generated by  
\begin{equation}
\text{$ \varphi(v_1,\dots,v_k,w)$ where ${\rm span}_\Z(v_1,\dots,v_k)={\rm L_3}$.}
\label{c-cells}
\end{equation}
By Lemma \ref{101}, the subcomplex $C({\rm L_3},w)$ is acyclic.
\bl \label{c-acyc}
 The subcomplex of ${\rm Coker}~ \widetilde \psi_\bullet^{(4)}$ generated by cells of type (\ref{c-cells}) is isomorphic to $\oplus_{({\rm L_3},w)}C({\rm L_3},w)$.
\el
\begin{proof}
The subcomplex of $ V_\bullet^{(4)}/dV_7^{(4)}$ generated by the cells of type (\ref{c-cells}) is generated by the 3-cells, the 4-cells of types ii) and iii), the 5-cells of types iii), and the 6-cells of type iii). The images of the 3-cells and the 4-cells of type iii) in ${\rm coker}~\widetilde \psi_\bullet^{(4)}$ are zero, since these cells are in the image of $~\widetilde \psi_\bullet^{(4)}$. For a cell $ \varphi(v_1,\dots,v_k,w)$ of the form (\ref{c-cells}) which is a 4-cell of type ii), a 5-cell of type iii), or a 6-cell of type iii), the lattice ${\rm L_3}$ and $\pm w$ are uniquely determined. This shows that the subcomplexes $C({\rm L_3},w)$ intersect trivially, so their sum is direct.
\end{proof}

Taking into account Proposition \ref{3.25.01.14},    it remains to prove

\bl
The complex $\overline{\rm Cok}~\widetilde \psi_\bullet^{(4)}$ is acyclic in degree 6.
\el
\begin{proof}
We see from the list of cells that $\overline{\rm Cok} ~\widetilde \psi_\bullet^{(4)}$ is generated by 6-cells of types i) and ii). On the other hand, in the proof of Theorem \ref{tps-inj} we saw that 
${\rm Im} ~\psi_\bullet^{(4)}$ contains elements of the form (\ref{six-one-two}), so $\overline{\rm Coker}~\widetilde \psi_6^{(4)}$ is generated by 6-cells of type ii) only. 

One has the differential of a 6-cell of type ii):
$$
\overline d:\overline \varphi(e_1,e_2,e_3,e_4,e_5,f_{12},f_{34})\mapsto\overline\varphi(e_1,e_2,e_3,e_4,e_5,f_{12})-\overline \varphi(e_1,e_2,e_3,e_4,e_5,f_{34}).
$$
It is a sum of 5-cells of type i). Indeed, each of the terms $\overline\varphi(e_1,\dots,\widehat{e_i},\dots,e_5,f_{12},f_{34})$ is zero in $\overline{\rm Cok}~\widetilde \psi_\bullet^{(4)}$, since it is a 5-cell of type iv) when $i=5$ and of type iii) when $i=1,2,3,4$.

A 5-cell of type i) $\overline \varphi(e_1,e_2,e_3,e_4,e_5,f_{12})$ determines the set $\{e_1,e_2,e_3,e_4,e_5\}$ up to sign, and there are no relations on the 5-cells of type i) in $\overline
{\rm Cok}~\widetilde \psi_5^{(4)}$. From this it is easy to see that the kernel of $\overline d$ is generated by ${\rm Cycle}_5\overline\varphi(e_1,e_2,e_3,e_4,e_5,f_{12},f_{34})$, which is in the image of $\psi_6^{(4)}$, so zero in $\overline{\rm Cok}~\widetilde \psi_6^{(4)}$. Therefore, $\overline d:\overline{\rm Cok}~\widetilde \psi_6^{(4)}\to\overline {\rm Cok}~\widetilde \psi_5^{(4)}$ is injective, as desired.
\end{proof}

\end{document}